\documentclass[a4paper]{amsart}
\usepackage{amscd}
\usepackage{amsmath}
\usepackage{amssymb}
\usepackage{mathrsfs}
\usepackage{amsthm}
\usepackage{bbm}
\usepackage{stmaryrd}
\usepackage{tikz-cd}
\usepackage[colorlinks,citecolor=blue,urlcolor=blue]{hyperref}

\usepackage[T1]{fontenc}
\usepackage[utf8]{inputenc}

\usepackage{lipsum}

\makeatletter
\renewcommand\subsection{\@startsection{subsection}{2}%
  \z@{-.5\linespacing\@plus-.7\linespacing}{.5\linespacing}%
  {\normalfont\scshape}}
\renewcommand\subsubsection{\@startsection{subsubsection}{3}%
  \z@{.5\linespacing\@plus.7\linespacing}{-.5em}%
  {\normalfont\scshape}}
\makeatother

\makeatletter
\@namedef{subjclassname@2020}{%
  \textup{2020} Mathematics Subject Classification}
\makeatother

\numberwithin{equation}{section} \swapnumbers

\newtheorem{satz}{Satz}[section]

\newtheorem{theorem}[satz]{Theorem}
\newtheorem{proposition}[satz]{Proposition}
\newtheorem{corollary}[satz]{Corollary}
\newtheorem{lemma}[satz]{Lemma}

\newtheorem{definition}[satz]{Definition}

\newtheorem{remark}[satz]{Remark}

\newtheorem{exercise}[satz]{Exercise}
\newtheorem*{solution}{Solution}

\newcommand{\bbb}{\mathbb{B}}

\newcommand{\bbe}{\mathbb{E}}

\newcommand{\bbn}{\mathbb{N}}
\newcommand{\bbp}{\mathbb{P}}

\newcommand{\bbr}{\mathbb{R}}

\newcommand{\bbx}{\mathbb{X}}
\newcommand{\bby}{\mathbb{Y}}
\newcommand{\bbz}{\mathbb{Z}}

\newcommand{\calc}{\mathcal{C}}

\newcommand{\calf}{\mathcal{F}}

\newcommand{\cali}{\mathcal{I}}

\newcommand{\scrc}{\mathscr{C}}
\newcommand{\scrd}{\mathscr{D}}

\newcommand{\bfb}{\mathbf{B}}
\newcommand{\bfx}{\mathbf{X}}
\newcommand{\bfy}{\mathbf{Y}}

\newcommand{\Sym}{{\rm Sym}}
\newcommand{\An}{{\rm Anti}}
\newcommand{\Id}{{\rm Id}}
\newcommand{\Lip}{{\rm Lip}}
\newcommand{\tr}{{\rm tr}}

\begin{document}

\title[Rough path theory and rough partial differential equations]{Rough path theory and an introduction to rough partial differential equations}
\author{Stefan Tappe}
\address{Albert Ludwig University of Freiburg, Department of Mathematical Stochastics, Ernst-Zermelo-Stra\ss{}e 1, D-79104 Freiburg, Germany}
\email{stefan.tappe@math.uni-freiburg.de}
\date{11 May, 2026}
\thanks{The author gratefully acknowledges financial support from the Deutsche Forschungsgemeinschaft (DFG, German Research Foundation) -- project number 444121509.}
\begin{abstract}
The goal of these notes is to provide an introduction to rough partial differential equations. For this purpose, we will present the theory of rough paths to the extend as it is required. Applications to stochastic partial differential equations are presented as well.
\end{abstract}

\maketitle\thispagestyle{empty}

\tableofcontents

\section{Introduction}

The theory of rough paths (see the textbook \cite{Friz-Hairer}, and also the lecture notes \cite{Allan}) provides a tool for defining integrals
\begin{align}\label{intro-integral}
\int_0^t Y_s d \mathbf{X}_s
\end{align}
for paths $\mathbf{X}$ which may be quite \emph{rough}, as opposed to the classical situation, where $\mathbf{X}$ is a smooth function. This in turn enables us to study differential equations
\begin{align}\label{intro-RDE}
\left\{
\begin{array}{rcl}
dY_t & = & f_0(t,Y_t) dt + f(t,Y_t) d \bfx_t
\\ Y_0 & = & \xi,
\end{array}
\right.
\end{align}
or even partial differential equations
\begin{align}\label{intro-RPDE}
\left\{
\begin{array}{rcl}
dY_t & = & (A Y_t + f_0(t,Y_t)) dt + f(t,Y_t) d \bfx_t
\\ Y_0 & = & \xi
\end{array}
\right.
\end{align}
with driving signals $\mathbf{X}$ which are allowed to be \emph{rough}. The essential idea is consider paths $\mathbf{X} = (X,\bbx)$ with a second order process $\bbx$ such that the paths $X$ and $\bbx$ satisfy certain algebraic constraints. The theory of rough paths is purely deterministic. However, as we will see, one prominent example of a rough driving signal $\mathbf{X}$ is given by typical sample paths of a Brownian motion, which can be used for the study of stochastic (partial) differential equations. We refer, e.g., to \cite{Gubinelli-Tindel, Hairer, Friz-Oberhauser, Teichmann, Friz-Oberhauser-2, Gerasimovics, Hesse-Neamtu-local, Gerasimovics-et-al, Hesse-Neamtu-global, Tappe-rough-manifolds, Tappe} for results about rough and stochastic partial differential equations.

The goal of these notes is to provide an introduction to rough partial differential equations. For this purpose, we will present the theory of rough paths to the extend as it is required. More precisely, we will introduce spaces of rough paths $\mathbf{X}$ and the definition of the \emph{rough integral} \eqref{intro-integral}, which is also called the \emph{Gubinelli integral}. After these and some further preparations, we will treat rough differential equations (RDEs) of the type \eqref{intro-RDE}.

Up to this point, the theory of rough paths is completely deterministic. In the next step, we will show that typical sample paths of a Brownian motion provide examples of rough paths, and use this for the study of stochastic differential equations (SDEs).

Afterwards, we will provide the announced introduction to rough \emph{partial} differential equations (RPDEs) of the type \eqref{intro-RPDE}, and draw consequences for stochastic \emph{partial} differential equations (SPDEs).

Let us provide a more detailed outlook of the upcoming sections of these notes. In Section \ref{sec-notation} we briefly explain some frequently used notation, and in Section \ref{sec-tensor-product} we recall the required results about the tensor product. After these preparations, in Section \ref{sec-rough-paths} we introduce spaces of rough paths, and in Section \ref{sec-integral} we review the Gubinelli integral. Further details about these topics can be found in \cite{Friz-Hairer} and \cite{Allan}. However, at some points we provide more details, so that readers may benefit from studying the aforementioned material of these notes.

Afterwards, in Section \ref{sec-composition} we provide some regularity results about compositions of paths with functions. Then, in Section \ref{sec-RDEs} we treat RDEs with time-inhomogeneous coefficients and present existence and uniqueness results. The material of these two sections is based on \cite{Tappe}. More precisely, Section \ref{sec-composition} essentially contains the material from \cite[Sec. 2.5, 2.6.1 and 2.6.2]{Tappe} and Section \ref{sec-RDEs} is essentially a simplified version of \cite[Sec. 2.9]{Tappe}, where RPDEs have been studied.

After these completely deterministic concepts and results, in Section \ref{sec-BB-rough-path} we show that typical sample paths of a Brownian motion provide examples of rough paths. Here we follow the standard material about rough path theory, as presented in \cite{Friz-Hairer} and \cite{Allan}, where further details can be found. However, also here we provide some more details at certain points, so that reading this section may be beneficial. Using these findings, in Section \ref{sec-SDEs} we briefly present existence and uniqueness results for SDEs.

Afterwards, we proceed with the announced introduction of RPDEs. More precisely, we present two approaches to RPDEs of the type \eqref{intro-RPDE} with an operator $A$ being the generator of a $C_0$-semigroup $(S_t)_{t \geq 0}$. In Section \ref{sec-RPDEs} we outline an approach which has often been used in the literature (see, e.g. \cite{Gubinelli-Tindel, Hairer, Gerasimovics, Hesse-Neamtu-local, Hesse-Neamtu-global}), and which is based on a mild sewing lemma. Even though we do not present all details, we point out the crucial differences to the concepts and results from our earlier Section \ref{sec-integral}, where we have treated the Gubinelli integral.

In Section \ref{sec-RPDEs-Tappe} we present another approach to RPDEs, which has been implemented in \cite{Tappe}. While the approach presented in Section \ref{sec-RPDEs} requires assumptions on the semigroup $(S_t)_{t \geq 0}$ like analyticity, the results from Section \ref{sec-RPDEs-Tappe} hold true for arbitrary $C_0$-semigroups $(S_t)_{t \geq 0}$, but on the other hand we have to impose additional conditions on the coefficients $f_0$ and $f$. For this approach we use and extend earlier results from these notes, in particular those from Sections \ref{sec-integral} and \ref{sec-RDEs}, and arrive at an existence and uniqueness result for RPDEs. We also draw consequences for SPDEs, which extends the material from Sections \ref{sec-BB-rough-path} and \ref{sec-SDEs}.

\section{Frequently used notation}\label{sec-notation}

In this section we briefly explain some frequently used notation. The symbols $V$ and $W$ typically denote Banach spaces endowed with their respective norms, always written as $| \cdot |$. The space $L(V,W)$ of all continuous linear operators $T : V \to W$ is also a Banach space. We will also use the notation $L(V) = L(V,V)$. On the product space $V \times W$ we will always consider the norm
\begin{align*}
|v,w| = |v| + |w|,
\end{align*}
which makes $V \times W$ a Banach space. For another Banach space $E$ we denote by $L^{(2)}(V \times W, E)$ the space of all continuous bilinear operators $T : V \times W \to E$.

For a bounded function $f : V \to W$ we denote by $\| f \|_{\infty}$ the supremum norm. For $n \in \bbn$ we denote by $C^n = C^n(V,W)$ the space of all mappings $f : V \to W$ which are $n$-times continuously Fr\'{e}chet differentiable. Furthermore, we denote by $C_b^n = C_b^n(V,W)$ the subspace of all $f \in C^n$ such that
\begin{align*}
\| f \|_{C_b^n} := \sum_{k=0}^n \| D^k f \|_{\infty} < \infty.
\end{align*}
The space $\Lip(V,W)$ denotes the space of all Lipschitz continuous functions; that is, the space of all functions $f : V \to W$ such that
\begin{align*}
| f |_{\Lip} := \sup_{\genfrac{}{}{0pt}{}{x,y \in V}{x \neq y}} \frac{|f(x) - f(y)|}{|x-y|} < \infty.
\end{align*}
Let $T \in \bbr_+$ and $\alpha \in (0,1]$ be arbitrary. We denote by $\calc^{\alpha}([0,T],V)$ the space of all $\alpha$-H\"{o}lder continuous functions; that is, the space of all functions $X : [0,T] \to V$ such that
\begin{align*}
\| X \|_{\alpha} := \sup_{\genfrac{}{}{0pt}{}{s,t \in [0,T]}{s \neq t}} \frac{|X_{s,t}|}{|t-s|^{\alpha}} < \infty,
\end{align*}
where we use the notation
\begin{align}\label{notation-increments}
X_{s,t} := X_t - X_s, \quad s,t \in [0,T]
\end{align}
for the increments. Furthermore, we denote by $\calc_2^{\alpha}([0,T]^2,V)$ the space of all continuous functions with on-diagonal $\alpha$-H\"{o}lder regularity; that is, the space of all continuous functions $\bbx : [0,T]^2 \to V$ such that
\begin{align*}
\| \bbx \|_{\alpha} := \sup_{\genfrac{}{}{0pt}{}{s,t \in [0,T]}{s \neq t}} \frac{| \bbx_{s,t}|}{|t-s|^{\alpha}} < \infty.
\end{align*}

\begin{remark}\label{rem-Hoelder}
Let $X : [0,T] \to V$ and $\bbx : [0,T]^2 \to V$ be two continuous functions such that
\begin{align*}
X_{s,t} = \bbx_{s,t} \quad \text{for all $s,t \in [0,T].$}
\end{align*}
Then we have $X \in \calc^{\alpha}([0,T],V)$ if and only if $\bbx \in \calc_2^{\alpha}([0,T]^2,V)$, and in this case $\| X \|_{\alpha} = \| \bbx \|_{\alpha}$.
\end{remark}

\section{The tensor product}\label{sec-tensor-product}

In this section we recall the required results about the tensor product.

\subsection{The tensor product of Banach spaces}

In this section we provide the required background about the tensor product of Banach spaces. The essential idea is to identify bilinear operators in Banach spaces with linear operators defined on another suitable Banach space. This is made precise by the following result.

\begin{theorem}\label{thm-tensor-product}
For two Banach spaces $V$ and $W$ there are another Banach space $V \otimes W$ and a continuous bilinear mapping $\otimes : V \times W \to V \otimes W$ such that:
\begin{enumerate}
\item $| v \otimes w | \leq | v | \cdot | w |$ for all $v \in V$ and $w \in W$.

\item $L^{(2)}(V \times W,X) \cong L(V \otimes W,X)$ for every further Banach space $X$.
\end{enumerate}
\end{theorem}

Let $V$ and $W$ be two Banach spaces. We call the Banach space $V \otimes W$ the \emph{tensor product} of $V$ and $W$, and the bilinear mapping $\otimes$ is called the \emph{tensor product mapping}. Moreover, elements of the form $v \otimes w$ with $v \in V$ and $w \in W$ are called \emph{tensors}.

Actually, the tensor product $V \otimes W$ is not unique, but it is unique up to an isomorphism, which is substantiated by the following result.

\begin{proposition}
Let $V \bar{\otimes} W$ be another Banach space, and let $\bar{\otimes} : V \times W \to V \bar{\otimes} W$ be another continuous bilinear mapping such that the properties from Theorem \ref{thm-tensor-product} are fulfilled. Them there exists a linear isomorphism $\Psi : V \otimes W \to V \bar{\otimes} W$ such that $\bar{\otimes} = \Psi \circ \otimes$.
\end{proposition}

We proceed with further well-known results about the tensor product.

\begin{proposition}\label{prop-identifications}
We have the identifications
\begin{align*}
L(V \otimes V,W) \cong L^{(2)}(V \times V, W) \cong L(V,L(V,W)).
\end{align*}
\end{proposition}

\begin{proposition}
We have $| v \otimes w | = |w \otimes v|$ for all $v,w \in V$.
\end{proposition}

\begin{proposition}\label{prop-symmetry-operator}
There is a unique linear operator $x \mapsto x^*$ from $L(V \otimes V)$ such that
\begin{align*}
(v \otimes w)^* = w \otimes v \quad \text{for all $v,w \in V$.}
\end{align*}
Furthermore, we have $x^{**} = x$ for all $x \in V \otimes V$.
\end{proposition}

We call the linear operator $x \mapsto x^*$ the \emph{symmetry operator}.

\begin{definition}
Let $x \in V \otimes V$ be arbitrary.
\begin{enumerate}
\item We call $x$ \emph{symmetric} if $x^* = x$.

\item We call $x$ \emph{antisymmetric} if $x^* = -x$.
\end{enumerate}
\end{definition}

\begin{definition}\label{def-sym-operator}
We define the two linear operators $\Sym, \An \in L(V \otimes V)$ as
\begin{align*}
\Sym(x) := \frac{x + x^*}{2} \quad \text{and} \quad \An(x) := \frac{x - x^*}{2}
\end{align*}
\end{definition}

\begin{proposition}
Every element $x \in V \otimes V$ admits a unique decomposition $x = y + z$ with a symmetric element $y \in V \otimes V$ and an antisymmetric element $z \in V \otimes V$; it is given by $y = \Sym(x)$ and $z = \An(x)$.
\end{proposition}

\subsection{The tensor product of finite dimensional spaces}

In this section we consider the finite dimensional situation. As the following result shows, in this case the tensor product is given by a space of matrices.

\begin{proposition}\label{prop-tensor-product-Rn}
We have $\bbr^m \otimes \bbr^n \cong \bbr^{m \times n}$, and the tensor product mapping is given by the dyadic product
\begin{align}\label{dyadic-product}
\otimes : \bbr^m \times \bbr^n \to \bbr^{m \times n}, \quad v \otimes w := v \cdot w^{\top}.
\end{align}
\end{proposition}

\begin{remark}
Consequently, for all $v \in \bbr^m$ and $w \in \bbr^n$ the components of the matrix $v \otimes w \in \bbr^{m \times n}$ are given by
\begin{align*}
(v \otimes w)_{ij} = v_i \cdot w_j \quad \text{for all $ = 1,\ldots,m$ and $j=1,\ldots,n$.}
\end{align*}
\end{remark}

The following well-known result shows that the tensor product may also be regarded as a space of linear operators.

\begin{proposition}\label{prop-lin-operator-matrix}
We have $L(\bbr^n,\bbr^m) \cong \bbr^{m \times n}$.
\end{proposition}

In view of the following exercise, recall the symmetry operator from Proposition \ref{prop-symmetry-operator}.

\begin{exercise}
Show that in case $V = \bbr^m$ the symmetry operator $A \mapsto A^*$ from $L(\bbr^{m \times m})$ is given by the transpose matrix $A^* = A^{\top}$.
\end{exercise}

\begin{solution}
Recalling the definition \eqref{dyadic-product} of the dyadic product, for all $v,w \in \bbr^m$ we obtain
\begin{align*}
(v \otimes w)^{\top} = (v \cdot w^{\top})^{\top} = w \cdot v^{\top} = w \otimes v.
\end{align*}
\end{solution}

\section{Spaces of rough paths}\label{sec-rough-paths}

In this section we provide the required background about spaces of rough paths. Let $V$ be a Banach space, and let $T \in \bbr_+$ be a finite time horizon.

\subsection{H\"{o}lder rough paths}

In this section we introduce spaces of H\"{o}lder rough paths, and provide some of their properties. In the sequel, a function $X : [0,T] \to V$ will be called a \emph{path}. For the increments of a path $X : [0,T] \to V$ we will use the notation \eqref{notation-increments}.

\begin{definition}
Let $\alpha \in (\frac{1}{3},\frac{1}{2}]$ be arbitrary. We define the space $\scrc^{\alpha}([0,T],V)$ of all \emph{$\alpha$-H\"{o}lder rough paths} (over $V$) as the space of all pairs $\mathbf{X} = (X,\bbx)$ with paths $X : [0,T] \to V$ and $\bbx : [0,T]^2 \to V \otimes V$ such that the following three conditions are fulfilled:
\begin{enumerate}
\item We have $X \in \calc^{\alpha}([0,T],V)$; that is
\begin{align*}
\| X \|_{\alpha} := \sup_{\genfrac{}{}{0pt}{}{s,t \in [0,T]}{s \neq t}} \frac{|X_{s,t}|}{|t-s|^{\alpha}} < \infty.
\end{align*}
\item We have $\bbx \in \calc_2^{2\alpha}([0,T]^2,V \otimes V)$; that is
\begin{align*}
\| \bbx \|_{2\alpha} := \sup_{\genfrac{}{}{0pt}{}{s,t \in [0,T]}{s \neq t}} \frac{|\bbx_{s,t}|}{|t-s|^{2\alpha}} < \infty.
\end{align*}
\item The algebraic relation
\begin{align}\label{Chen-relation}
 \bbx_{s,t} - \bbx_{s,u} - \bbx_{u,t} = X_{s,u} \otimes X_{u,t} \quad \text{for all $s,u,t \in [0,T]$,}
\end{align}
which we call \emph{Chen's relation}, is satisfied.
\end{enumerate}
\end{definition}

\begin{remark}
The path $\bbx$ of a rough path $\mathbf{X} = (X,\bbx)$ is also called the associated \emph{second order process} or \emph{L\'{e}vy area}.
\end{remark}

\begin{remark}
Consider the product space
\begin{align}\label{product-Hoelder}
P^{\alpha}([0,T],V) := \calc^{\alpha}([0,T],V) \times \calc_2^{2\alpha}([0,T]^2, V \otimes V).
\end{align}
Then $\scrc^{\alpha}([0,T],V)$ consists of all $\mathbf{X} = (X,\bbx) \in P^{\alpha}([0,T],V)$ such that Chen's relation \eqref{Chen-relation} is fulfilled. In particular, it is not a vector space, because for $\mathbf{X},\mathbf{Y} \in \scrc^{\alpha}([0,T],V)$ the sum $\mathbf{X} + \mathbf{Y}$ may not satisfy Chen's relation \eqref{Chen-relation}.
\end{remark}

\begin{remark}
At this point, let us already mention some consequences of Chen's relation \eqref{Chen-relation}, which we will derive:
\begin{enumerate}
\item For every rough path $\mathbf{X} = (X,\bbx)$ we have $\bbx_{t,t} = 0$ for all $t \in [0,T]$, and the values of $\bbx_{t,s}$ for $s < t$ are determined by $X_{s,t}$ and $\bbx_{s,t}$; see identities \eqref{Chen-rule-1} and \eqref{Chen-rule-2}.

\item The values of $X_t$ and $\bbx_{0,t}$ for $t \in [0,T]$ completely determine the values of a rough path $\mathbf{X} = (X,\bbx)$; see identity \eqref{Chen-rule-3} and Corollary \ref{cor-determine-path}. This actually justifies to call $\bfx$ a rough \emph{path}.

\item The second order process $\bbx$ of a rough path $\mathbf{X} = (X,\bbx)$ is determined by $X$ up to the increments of an additional one-parameter function $F : [0,T] \to V \otimes V$; see Corollary \ref{cor-F-second-order}.

\item Let $\mathbf{X} = (X,\bbx)$ be a rough path such that $X$ is of finite variation. Then the second order process $\bbx$ is given by
\begin{align}\label{second-order-FV}
\bbx_{s,t} = \int_s^t X_{s,r} \otimes d X_r, \quad s,t \in [0,T],
\end{align}
of course, up to the increments of an additional one-parameter function $F : [0,T] \to V \otimes V$. We refer to Exercise \ref{exercise-FV-path} for more details.

\item To anticipate a bit, in Section \ref{sec-integral} we will introduce the rough integral as the limit \eqref{candidate}. Chen's relation \eqref{Chen-relation} is crucial in order to ensure that this limit actually exists, and that the rough integral has the desired properties such as estimate \eqref{estimate-third-order}. We refer to Exercise \ref{exercise-sewing-lemma} for the precise result.
\end{enumerate}
\end{remark}

In the following exercise we derive some consequences of Chen's relation \eqref{Chen-relation}.

\begin{exercise}
Let $\mathbf{X} = (X,\bbx)$ be a pair of paths $X : [0,T] \to V$ and $\bbx : [0,T]^2 \to V \otimes V$ such that Chen's relation \eqref{Chen-relation} is fulfilled.
\begin{enumerate}
\item[(a)] Show the identity
\begin{align}\label{Chen-rule-1}
\bbx_{t,t} = 0 \quad \text{for all $t \in [0,T]$.}
\end{align}
\item[(b)] Show the identity
\begin{align}\label{Chen-rule-2}
\bbx_{s,t} = X_{s,t} \otimes X_{s,t} - \bbx_{t,s} \quad \text{for all $s,t \in [0,T]$.}
\end{align}
\item[(c)] Show the identity
\begin{align}\label{Chen-rule-3}
\bbx_{s,t} = \bbx_{0,t} - \bbx_{0,s} - X_{0,s} \otimes X_{s,t} \quad \text{for all $s,t \in [0,T]$.}
\end{align}
\end{enumerate}
\end{exercise}

\begin{solution}
\begin{enumerate}
\item[(a)] By Chen's relation \eqref{Chen-relation} with $s = u = t$ we have
\begin{align*}
\bbx_{t,t} = - (\bbx_{t,t} - \bbx_{t,t} - \bbx_{t,t}) = X_{t,t} \otimes X_{t,t} = 0 \quad \text{for all $t \in [0,T]$.}
\end{align*}
\item[(b)] By Chen's relation \eqref{Chen-relation} with $s = t$ and part (a) we obtain
\begin{align*}
- \bbx_{t,u} - \bbx_{u,t} = \bbx_{t,t} - \bbx_{t,u} - \bbx_{u,t} = X_{t,u} \otimes X_{u,t} = - X_{u,t} \otimes X_{u,t},
\end{align*}
and hence
\begin{align*}
\bbx_{u,t} = X_{u,t} \otimes X_{u,t} - \bbx_{t,u} \quad \text{for all $u,t \in [0,T]$.}
\end{align*}
\item[(c)] By Chen's relation \eqref{Chen-relation} with $s = 0$ and $u=s$ we have
\begin{align*}
\bbx_{0,t} - \bbx_{0,s} - \bbx_{s,t} = X_{0,s} \otimes X_{s,t} \quad \text{for all $s,t \in [0,T]$.}
\end{align*}
\end{enumerate}
\end{solution}

As an immediate consequence of identity \eqref{Chen-rule-3}, we obtain that a path $\bfx = (X,\bbx)$ is already completely determined by its values $(X_t,\bbx_{0,t})$ for $t \in [0,T]$, provided Chen's relation \eqref{Chen-relation} is fulfilled.

\begin{corollary}\label{cor-determine-path}
Let $\mathbf{X} = (X,\bbx)$ and $\mathbf{Y} = (Y,\bby)$ be two pairs with paths $X,Y : [0,T] \to V$ and $\bbx,\bby : [0,T]^2 \to V \otimes V$, which both satisfy Chen's relation \eqref{Chen-relation}. If $(X_t,\bbx_{0,t}) = (Y_t,\bby_{0,t})$ for all $t \in [0,T]$, then we have $\bfx = \bfy$.
\end{corollary}

The following exercise shows that the seemingly weaker identity \eqref{Chen-rule-3} already implies Chen's relation \eqref{Chen-relation}. This is useful when verifying Chen's relation for given paths $X : [0,T] \to V$ and $\bbx : [0,T]^2 \to V \otimes V$; see, e.g., Exercise \ref{exercise-FV-path}.

\begin{exercise}\label{exercise-Chen-sufficient}
Let $\mathbf{X} = (X,\bbx)$ be a pair of paths $X : [0,T] \to V$ and $\bbx : [0,T]^2 \to V \otimes V$. Show that Chen's relation \eqref{Chen-relation} is fulfilled if and only if we have \eqref{Chen-rule-3}.
\end{exercise}

\begin{solution}
We only need to show that \eqref{Chen-rule-3} implies Chen's relation \eqref{Chen-relation}. Indeed, for all $s,u,t \in [0,T]$ we have
\begin{align*}
\bbx_{s,t} - \bbx_{s,u} - \bbx_{u,t} &= \bbx_{0,t} - \bbx_{0,s} - X_{0,s} \otimes X_{s,t}
\\ &\quad - \big( \bbx_{0,u} - \bbx_{0,s} - X_{0,s} \otimes X_{s,u} \big)
\\ &\quad - \big( \bbx_{0,t} - \bbx_{0,u} - X_{0,u} \otimes X_{u,t} \big)
\\ &= -X_{0,s} \otimes (X_t - X_s) + X_{0,s} \otimes (X_u - X_s) + X_{0,u} \otimes (X_t - X_u)
\\ &= X_{0,s} \otimes (X_u - X_t) + X_{0,u} \otimes (X_t - X_u)
\\ &= (X_{0,u} - X_{0,s}) \otimes (X_t - X_u) = X_{s,u} \otimes X_{u,t}.
\end{align*}
\end{solution}

The following exercise provides further identities for the second order process of a rough path.

\begin{exercise}
Let $\mathbf{X} = (X,\bbx)$ be a pair of paths $X : [0,T] \to V$ and $\bbx : [0,T]^2 \to V \otimes V$ such that Chen's relation \eqref{Chen-relation} is fulfilled. We fix time points $0 \leq s = \tau_0 < \tau_1 < \ldots < \tau_N = t \leq T$ for some $N \in \bbn$.
\begin{enumerate}
\item[(a)] Show the identity
\begin{align*}
\bbx_{s,t} = \sum_{i=0}^{N-1} \bbx_{\tau_i,\tau_{i+1}} + \sum_{\genfrac{}{}{0pt}{}{i,j = 0}{j < i}}^{N-1} X_{\tau_j,\tau_{j+1}} \otimes X_{\tau_i,\tau_{i+1}}.
\end{align*}
\item[(b)] Show the identity
\begin{align*}
\bbx_{s,t} = \sum_{i=0}^{N-1} \big( \bbx_{\tau_i,\tau_{i+1}} + X_{s,\tau_i} \otimes X_{\tau_i,\tau_{i+1}} \big).
\end{align*}
\end{enumerate}
\end{exercise}

\begin{solution}
\begin{enumerate}
\item[(a)] Using identity \eqref{Chen-rule-3} twice we obtain
\begin{align*}
\bbx_{s,t} &= \bbx_{\tau_0,\tau_N} = \bbx_{0,\tau_N} - \bbx_{0,\tau_0} - X_{0,\tau_0} \otimes X_{\tau_0,\tau_N}
\\ &= \sum_{i=0}^{N-1} ( \bbx_{0,\tau_{i+1}} - \bbx_{0,\tau_i} ) - X_{0,\tau_0} \otimes \sum_{i=0}^{N-1} X_{\tau_i,\tau_{i+1}}
\\ &= \sum_{i=0}^{N-1} ( \bbx_{0,\tau_{i+1}} - \bbx_{0,\tau_i} ) - \sum_{i=0}^{N-1} ( X_{0,\tau_i} - X_{\tau_0,\tau_i} ) \otimes X_{\tau_i,\tau_{i+1}}
\\ &= \sum_{i=0}^{N-1} ( \bbx_{0,\tau_{i+1}} - \bbx_{0,\tau_i} ) - \sum_{i=0}^{N-1} \bigg( X_{0,\tau_i} - \sum_{j=0}^{i-1} X_{\tau_j,\tau_{j+1}} \bigg) \otimes X_{\tau_i,\tau_{i+1}}
\\ &= \sum_{i=0}^{N-1} \big( \bbx_{0,\tau_{i+1}} - \bbx_{0,\tau_i} - X_{0,\tau_i} \otimes X_{\tau_i,\tau_{i+1}} \big) + \sum_{\genfrac{}{}{0pt}{}{i,j = 0}{j < i}}^{N-1} X_{\tau_j,\tau_{j+1}} \otimes X_{\tau_i,\tau_{i+1}}
\\ &= \sum_{i=0}^{N-1} \bbx_{\tau_i,\tau_{i+1}} + \sum_{\genfrac{}{}{0pt}{}{i,j = 0}{j < i}}^{N-1} X_{\tau_j,\tau_{j+1}} \otimes X_{\tau_i,\tau_{i+1}}.
\end{align*}
\item[(b)] By identity \eqref{Chen-rule-1} and Chen's relation \eqref{Chen-relation} we deduce that
\begin{align*}
\\ \bbx_{s,t} &= \bbx_{s,\tau_N} =  \bbx_{s,\tau_N} - \bbx_{s,s} = \bbx_{s,\tau_N} -  \bbx_{s,\tau_0}
\\ &= \sum_{i=0}^{N-1} \big( \bbx_{s,\tau_{i+1}} - \bbx_{s,\tau_i} \big) = \sum_{i=0}^{N-1} \big( \bbx_{\tau_i,\tau_{i+1}} + \bbx_{s,\tau_{i+1}} - \bbx_{s,\tau_i} - \bbx_{\tau_i,\tau_{i+1}} \big)
\\ &= \sum_{i=0}^{N-1} \big( \bbx_{\tau_i,\tau_{i+1}} + X_{s,\tau_i} \otimes X_{\tau_i,\tau_{i+1}} \big).
\end{align*}
\end{enumerate}
\end{solution}

It arises the question whether the second order process $\bbx$ of a rough path $\mathbf{X} = (X,\bbx)$ is uniquely determined $X$. This is not true, but, as the following exercise shows, due to Chen's relation \eqref{Chen-relation} it determines the second order process $\bbx$ up to the increments of an additional one-parameter function.

\begin{exercise}\label{exercise-F-second-order}
Let $\mathbf{X} = (X,\bbx)$ and $\bar{\mathbf{X}} = (X,\bar{\bbx})$ be two pairs of paths with the same first order process $X : [0,T] \to V$ and paths $\bbx,\bar{\bbx} : [0,T]^2 \to V \otimes V$ such that $\mathbf{X}$ satisfies Chen's relation \eqref{Chen-relation}. Show that the following statements are equivalent:
\begin{enumerate}
\item[(i)] $\bar{\mathbf{X}}$ satisfies Chen's relation \eqref{Chen-relation}.

\item[(ii)] There exists a function $F : [0,T] \to V \otimes V$ such that
\begin{align}\label{F-second-order}
\bar{\bbx}_{s,t} = \bbx_{s,t} + F_{s,t} \quad \text{for all $s,t \in [0,T]$.}
\end{align}
\end{enumerate}
\end{exercise}

\begin{solution}
(ii) $\Rightarrow$ (i): Using Chen's relation \eqref{Chen-relation} for $\mathbf{X}$, for all $s,u,t \in [0,T]$ we have
\begin{align*}
\bar{\bbx}_{s,t} - \bar{\bbx}_{s,u} - \bar{\bbx}_{u,t} &= (\bbx_{s,t} + F_t - F_s) - (\bbx_{s,u} + F_u - F_s) - (\bbx_{u,t} + F_t - F_u)
\\ &= \bbx_{s,t} - \bbx_{s,u} - \bbx_{u,t} = X_{s,u} \otimes X_{u,t}.
\end{align*}
(i) $\Rightarrow$ (ii): We define $G : [0,T]^2 \to V \otimes V$ as $G := \bbx - \bar{\bbx}$. By \eqref{Chen-rule-2} we have
\begin{align*}
G_{s,t} &= \bbx_{s,t} - \bar{\bbx}_{s,t} = X_{s,t} \otimes X_{s,t} - \bbx_{t,s} - ( X_{s,t} \otimes X_{s,t} - \bar{\bbx}_{t,s})
\\ &= - ( \bbx_{t,s} - \bar{\bbx}_{t,s} ) = - G_{t,s} \quad \text{for all $s,t \in [0,T]$.}
\end{align*}
By Chen's relation \eqref{Chen-relation}, for all $s,u,t \in [0,T]$ we have
\begin{align*}
\bbx_{s,t} - \bbx_{s,u} - \bbx_{u,t} = X_{s,u} \otimes X_{u,t} = \bar{\bbx}_{s,t} - \bar{\bbx}_{s,u} - \bar{\bbx}_{u,t},
\end{align*}
and hence
\begin{align*}
G_{s,t} &= \bbx_{s,t} - \bar{\bbx}_{s,t} = \bbx_{s,u} + \bbx_{u,t} - \bar{\bbx}_{s,u} - \bar{\bbx}_{u,t}
\\ &= ( \bbx_{s,u} - \bar{\bbx}_{s,u} ) + ( \bbx_{u,t} - \bar{\bbx}_{u,t} ) = G_{s,u} + G_{u,t} = G_{u,t} - G_{u,s}.
\end{align*}
Now, we define $F : [0,T] \to V \otimes V$ as $F_t := G_{0,t}$, $t \in [0,T]$. With $u=0$ the identity above gives us
\begin{align*}
\bar{\bbx}_{s,t} - \bbx_{s,t} = G_{s,t} = G_{0,t} - G_{0,s} = F_t - F_s \quad \text{for all $s,t \in [0,T]$.}
\end{align*}
\end{solution}

\begin{remark}\label{rem-F-second-order}
Identity \eqref{F-second-order} with $s=0$ shows that the function $F : [0,T] \to V \otimes V$ satisfies
\begin{align*}
F_t = F_0 + \bar{\bbx}_{0,t} - \bbx_{0,t}, \quad t \in [0,T].
\end{align*}
In particular, in case $F_0 = 0$ it is uniquely determined.
\end{remark}

\begin{corollary}\label{cor-F-second-order}
Let $\mathbf{X} = (X,\bbx)$ and $\bar{\mathbf{X}} = (X,\bar{\bbx})$ be two pairs of paths with the same first order process $X : [0,T] \to V$ and paths $\bbx,\bar{\bbx}: [0,T]^2 \to V \otimes V$ such that $\mathbf{X} \in \scrc^{\alpha}([0,T],V)$ for some $\alpha \in (\frac{1}{3},\frac{1}{2}]$. Then the following statements are equivalent:
\begin{enumerate}
\item[(i)] We have $\bar{\mathbf{X}} \in \scrc^{\alpha}([0,T],V)$.

\item[(ii)] There exists a function $F \in \calc^{2 \alpha}([0,T], V \otimes V)$ such that \eqref{F-second-order} is fulfilled.
\end{enumerate}
\end{corollary}

\begin{proof}
This is an immediate consequence of Exercise \ref{exercise-F-second-order}, where we note that by \eqref{F-second-order} we have $\bar{\bbx} \in \calc_2^{2 \alpha}([0,T]^2, V \otimes V)$ if and only if $F \in \calc^{2 \alpha}([0,T], V \otimes V)$, thanks to Remark \ref{rem-Hoelder}.
\end{proof}

The next exercise shows how we can construct a second order process $\bbx$ for a finite variation path $X$.

\begin{exercise}\label{exercise-FV-path}
Let $X : [0,T] \to V$ be a continuous path of finite variation. We consider $\mathbf{X} = (X,\bbx)$, where $\bbx : [0,T]^2 \to V \otimes V$ is defined as \eqref{second-order-FV}. Show that the following statements are true:
\begin{enumerate}
\item[(a)] Chen's relation \eqref{Chen-relation} is fulfilled.

\item[(b)] If $X \in \calc^{2\alpha}([0,T],V)$ for some $\alpha \in (\frac{1}{3},\frac{1}{2}]$, then we have $\mathbf{X} \in \scrc^{\alpha}([0,T],V)$.
\end{enumerate}
\end{exercise}

\begin{solution}
\begin{enumerate}
\item[(a)] In view of Exercise \ref{exercise-Chen-sufficient} it suffices to show \eqref{Chen-rule-3}. Indeed, for all $s,t \in [0,T]$ we have
\begin{align*}
\bbx_{s,t} &= \int_s^t X_{s,r} \otimes d X_r = \int_s^t ( X_{0,r} - X_{0,s} ) \otimes dX_r
\\ &= \int_s^t X_{0,r} \otimes dX_r - X_{0,s} \otimes (X_t - X_s)
\\ &= \int_0^t X_{0,r} \otimes dX_r - \int_0^s X_{0,r} \otimes dX_r - X_{0,s} \otimes X_{s,t}
\\ &= \bbx_{0,t} - \bbx_{0,s} - X_{0,s} \otimes X_{s,t}.
\end{align*}
\item[(b)] Of course, we also have $X \in \calc^{\alpha}([0,T],V)$. Let $s,t \in [0,T]$ be arbitrary. Then we have
\begin{align*}
\bbx_{s,t} = \lim_{|\Pi| \to 0} \sum_{[u,v] \in \Pi} X_{s,u} \otimes X_{u,v},
\end{align*}
where $\Pi$ denotes any partition of the interval $[s,t]$. Denoting by $V(X)$ the total variation of $X$ on $[0,T]$, for any partition $\Pi$ of $[s,t]$ we have
\begin{align*}
\sum_{[u,v] \in \Pi} |X_{s,u} \otimes X_{u,v}| &\leq \sum_{[u,v] \in \Pi} |X_{s,u}| \, |X_{u,v}| \leq \sum_{[u,v] \in \Pi} \| X \|_{2\alpha} |u-s|^{2\alpha} |X_{u,v}|
\\ &\leq \| X \|_{2\alpha} |t-s|^{2\alpha} \sum_{[u,v] \in \Pi} |X_{u,v}| \leq \| X \|_{2\alpha} V(X) |t-s|^{2\alpha},
\end{align*}
and hence $\bbx \in \calc_2^{2\alpha}([0,T],V \otimes V)$.
\end{enumerate}
\end{solution}

As an immediate consequence of Exercise \ref{exercise-FV-path} we obtain the following result:

\begin{corollary}
Let $X \in C^1([0,T],V)$ be a smooth path. We consider $\mathbf{X} = (X,\bbx)$, where $\bbx : [0,T]^2 \to V \otimes V$ is defined as
\begin{align*}
\bbx_{s,t} := \int_s^t X_{s,r} \otimes \dot{X}_r \, dr, \quad s,t \in [0,T].
\end{align*}
Then we have $(X,\bbx) \in \scrc^{\alpha}([0,T],V)$ for all $\alpha \in (\frac{1}{3}, \frac{1}{2}]$.
\end{corollary}

The next goal is to introduce a topology on the space of rough paths. Let $\alpha \in (\frac{1}{3},\frac{1}{2}]$ be arbitrary. Note that
\begin{align*}
\interleave \mathbf{X} \interleave_{\alpha} := \| X \|_{\alpha} + \| \bbx \|_{2 \alpha}, \quad \mathbf{X} = (X,\bbx)
\end{align*}
defines a seminorm on the product space $P^{\alpha}([0,T],V)$ introduced in \eqref{product-Hoelder}. Therefore, the mapping
\begin{align*}
\varrho_{\alpha}(\mathbf{X},\mathbf{Y}) = \interleave \mathbf{X} - \mathbf{Y} \interleave_{\alpha}
\end{align*}
provides a pseudometric on $\scrc^{\alpha}([0,T],V)$.

\begin{lemma}\label{lemma-alpha-beta}
Let $\beta \in (\frac{1}{3},\frac{1}{2}]$ and $\alpha \in (0,\beta)$ be arbitrary. Then for each rough path $\mathbf{X} = (X,\bbx) \in \scrc^{\beta}([0,T],V)$ we also have $\mathbf{X} \in \scrc^{\alpha}([0,T],V)$ and the estimate
\begin{align*}
\interleave \mathbf{X} \interleave_{\alpha} \leq \| X \|_{\beta} T^{\beta-\alpha} + \| \mathbb{X} \|_{2\beta} T^{2(\beta-\alpha)}.
\end{align*}
\end{lemma}

\begin{proof}
Since $\alpha < \beta$, we have $\mathbf{X} \in \scrc^{\alpha}([0,T],V)$. Hence, the calculation
\begin{align*}
\interleave \mathbf{X} \interleave_{\alpha} &= \| X \|_{\alpha} + \| \bbx \|_{2\alpha}
\\ &\leq \| X \|_{\beta} T^{\beta-\alpha} + \| \mathbb{X} \|_{2\beta} T^{2(\beta-\alpha)}
\end{align*}
establishes the proof.
\end{proof}

Note that the mapping
\begin{align*}
| \mathbf{X} |_{\alpha} := | X_0 | + \interleave \bfx \interleave_{\alpha} = | X_0 | + \| X \|_{\alpha} + \| \bbx \|_{2 \alpha}, \quad \mathbf{X} = (X,\bbx)
\end{align*}
defines a norm on the product space $P^{\alpha}([0,T],V)$. Indeed, for $\mathbf{X} = (X,\bbx) \in P^{\alpha}([0,T],V)$ with $| \mathbf{X} |_{\alpha} = 0$ we obtain
\begin{align*}
X_0 &= 0,
\\ X_t - X_0 &= 0 \quad \text{for all $t \in [0,T]$,}
\\ \bbx_{s,t} &= 0 \quad \text{for all $s,t \in [0,T]$,}
\end{align*}
and hence $\mathbf{X} = 0$. Consequently, the mapping
\begin{align}\label{metric-rough-paths}
d_{\alpha}(\mathbf{X},\mathbf{Y}) = | \mathbf{X} - \mathbf{Y} |_{\alpha}
\end{align}
provides a metric on $P^{\alpha}([0,T],V)$, and thus in particular on $\scrc^{\alpha}([0,T],V)$. The following auxiliary result shows that convergence with respect to the metric $d_{\alpha}$ implies uniform convergence of the two components.

\begin{lemma}\label{lemma-estimate-supremum}
For all $\mathbf{X} = (X,\bbx) \in P^{\alpha}([0,T],V)$ we have
\begin{align*}
\| X \|_{\infty} + \| \bbx \|_{\infty} \leq ( 1 + T^{\alpha} + T^{2 \alpha} ) | \mathbf{X} |_{\alpha}.
\end{align*}
\end{lemma}

\begin{proof}
For all $t \in [0,T]$ we have
\begin{align*}
| X_t | \leq |X_0| + |X_t - X_0| \leq |X_0| + \| X \|_{\alpha} t^{\alpha} \leq |X_0| + \| X \|_{\alpha} T^{\alpha}.
\end{align*}
Moreover, for all $s,t \in [0,T]$ we have
\begin{align*}
| \bbx_{s,t} | \leq \| \bbx \|_{2 \alpha} |t-s|^{2 \alpha} \leq \| \bbx \|_{2 \alpha} T^{2 \alpha}.
\end{align*}
These two inequalities give us
\begin{align*}
\| X \|_{\infty} + \| \bbx \|_{\infty} \leq |X_0| + T^{\alpha} \| X \|_{\alpha} + T^{2 \alpha} \| \bbx \|_{2 \alpha},
\end{align*}
completing the proof.
\end{proof}

After these preparations, we can show that the space of $\alpha$-H\"{o}lder rough paths is a complete metric space.

\begin{proposition}\label{prop-complete-metric-space}
The space $( \scrc^{\alpha}([0,T],V), d_{\alpha} )$ is a complete metric space.
\end{proposition}

\begin{proof}
Note that the product space $P^{\alpha}([0,T],V)$ endowed with the norm $| \cdot |_{\alpha}$ is a Banach space, and hence a complete metric space. Thus, we have to show that $\scrc^{\alpha}([0,T],V)$ is a closed subset. Let $(\mathbf{X}^n)_{n \in \bbn} = (X^n,\bbx^n)_{n \in \bbn} \subset \scrc^{\alpha}([0,T],V)$ be a sequence, and let $\mathbf{X} = (X,\bbx) \in P^{\alpha}([0,T],V)$ be such that $d_{\alpha}(\mathbf{X}^n,\mathbf{X}) \to 0$. Since $\otimes : V \times V \to V \otimes V$ is a continuous bilinear operator (see Theorem \ref{thm-tensor-product}), by Lemma \ref{lemma-estimate-supremum} for all $s,u,t \in [0,T]$ we obtain
\begin{align*}
\bbx_{s,t} - \bbx_{s,u} - \bbx_{u,t} &= \lim_{n \to \infty} \big( \bbx_{s,t}^n - \bbx_{s,u}^n - \bbx_{u,t}^n \big)
\\ &= \lim_{n \to \infty} \big( X_{s,u}^n \otimes X_{u,t}^n \big) = X_{s,u} \otimes X_{u,t},
\end{align*}
and hence $\bfx \in \scrc^{\alpha}([0,T],V)$.
\end{proof}

\subsection{The bracket of a rough path}

In this section we introduce the bracket of a rough path, which will be useful in order to characterize weakly geometric rough paths in the upcoming section. Let $\mathbf{X} = (X,\bbx) \in \scrc^{\alpha}([0,T],V)$ be a rough path for some index $\alpha \in ( \frac{1}{3}, \frac{1}{2} ]$. In view of the following definition, recall the linear operator $\Sym \in L(V \otimes V)$ from Definition \ref{def-sym-operator}, which assigns the symmetric part of an element $x \in V \otimes V$.

\begin{definition}
We introduce the following two paths:
\begin{enumerate}
\item We define the \emph{(one-parameter) bracket} $[ \mathbf{X} ] : [0,T] \to {\rm Sym}(V \otimes V)$ as
\begin{align}\label{bracket-def-1}
[ \mathbf{X} ]_t := X_{0,t} \otimes X_{0,t} - 2 \, {\rm Sym}(\bbx_{0,t}), \quad t \in [0,T].
\end{align}
\item We define the \emph{(two-parameter) bracket} $[ \mathbf{X} ] : [0,T]^2 \to {\rm Sym}(V \otimes V)$ as
\begin{align}\label{bracket-def-2}
[ \mathbf{X} ]_{s,t} := X_{s,t} \otimes X_{s,t} - 2 \, {\rm Sym}(\bbx_{s,t}), \quad s,t \in [0,T].
\end{align}
\end{enumerate}
\end{definition}

Note that $[ \mathbf{X} ]_0 = 0$ due to \eqref{Chen-rule-1}. The following exercise provides the connection between the two brackets.

\begin{exercise}\label{exercise-bracket}
Show that $[ \mathbf{X} ]_{s,t} = [ \mathbf{X} ]_t - [ \mathbf{X} ]_s$ for all $s,t \in [0,T]$.
\end{exercise}

\begin{solution}
Let $s,t \in [0,T]$ be arbitrary. Then we have
\begin{align*}
X_{s,t} = X_t - X_s = (X_t - X_0) - (X_s - X_0) = X_{0,t} - X_{0,s},
\end{align*}
Thus, by \eqref{Chen-rule-3} we obtain
\begin{align*}
\bbx_{s,t} &= \bbx_{0,t} - \bbx_{0,s} - X_{0,s} \otimes X_{s,t}
\\ &= \bbx_{0,t} - \bbx_{0,s} - X_{0,s} \otimes X_{0,t} + X_{0,s} \otimes X_{0,s}.
\end{align*}
Moreover, we have
\begin{align*}
{\rm Sym}(v \otimes w) = \frac{1}{2} \big( v \otimes w + w \otimes v \big) \quad \text{for all $v,w \in V$.}
\end{align*}
Therefore, we obtain
\begin{align*}
[ \mathbf{X} ]_{s,t} &= X_{s,t} \otimes X_{s,t} - 2 \, {\rm Sym}(\bbx_{s,t})
\\ &= (X_{0,t} - X_{0,s}) \otimes (X_{0,t} - X_{0,s})
\\ &\quad - 2 \, {\rm Sym} \big( \bbx_{0,t} - \bbx_{0,s} - X_{0,s} \otimes X_{0,t} + X_{0,s} \otimes X_{0,s} \big)
\\ &= X_{0,t} \otimes X_{0,t} - X_{0,t} \otimes X_{0,s} - X_{0,s} \otimes X_{0,t} + X_{0,s} \otimes X_{0,s}
\\ &\quad - 2 \, {\rm Sym}(\bbx_{0,t}) + 2 \, {\rm Sym}(\bbx_{0,s}) + 2 \, {\rm Sym}(X_{0,s} \otimes X_{0,t}) - 2 X_{0,s} \otimes X_{0,s}
\\ &= \big( X_{0,t} \otimes X_{0,t} - 2 \, {\rm Sym}(\bbx_{0,t}) \big) - \big( X_{0,s} \otimes X_{0,s} - 2 \, {\rm Sym}(\bbx_{0,s}) \big)
\\ &= [ \mathbf{X} ]_t - [ \mathbf{X} ]_s.
\end{align*}
\end{solution}

In view of the next exercise, recall Corollary \ref{cor-F-second-order}.

\begin{exercise}\label{exercise-bracket-F}
Let $\mathbf{X} = (X,\bbx) \in \scrc^{\alpha}([0,T],V)$ and $\bar{\mathbf{X}} = (X,\bar{\bbx}) \in \scrc^{\alpha}([0,T],V)$ be two rough paths with the same first order process $X$ such that for some function $F \in \calc^{2 \alpha}([0,T], V \otimes V)$ we have \eqref{F-second-order}. Show that
\begin{align*}
[ \bar{\mathbf{X}} ]_t = [ \mathbf{X} ]_t - 2 \, \Sym(F_{0,t}) \quad \text{for all $t \in [0,T]$.}
\end{align*}
\end{exercise}

\begin{solution}
Using \eqref{F-second-order}, for all $t \in [0,T]$ we have
\begin{align*}
[ \bar{\mathbf{X}} ]_t &= X_{0,t} \otimes X_{0,t} - 2 \, {\rm Sym}(\bar{\bbx}_{0,t}) = X_{0,t} \otimes X_{0,t} - 2 \, {\rm Sym}(\bbx_{0,t} + F_{0,t})
\\ &= X_{0,t} \otimes X_{0,t} - 2 \, {\rm Sym}(\bbx_{0,t}) - 2 \, \Sym (F_{0,t}) = [ \mathbf{X} ]_t - 2 \, \Sym(F_{0,t}).
\end{align*}
\end{solution}

\subsection{Weakly geometric rough paths}

In this section we introduce weakly geometric rough paths. Let us fix an index $\alpha \in ( \frac{1}{3}, \frac{1}{2} ]$.

\begin{definition}
We say that a rough path $\mathbf{X} = (X,\bbx) \in \scrc^{\alpha}([0,T],V)$ is \emph{weakly geometric} if the first order calculus condition
\begin{align}\label{chain-rule}
{\rm Sym}(\bbx_{s,t}) = \frac{1}{2} X_{s,t} \otimes X_{s,t} \quad \text{for all $s,t \in [0,T]$}
\end{align}
is fulfilled.
\end{definition}

Condition \eqref{chain-rule} is sometimes also called the \emph{chain rule}.

\begin{definition}
We denote by $\scrc_g^{\alpha}([0,T],V)$ the set of all weakly geometric rough paths.
\end{definition}

Let us recall from Proposition \ref{prop-complete-metric-space} that $\scrc^{\alpha}([0,T],V)$ endowed with the metric $d_{\alpha}$ given by \eqref{metric-rough-paths} is a complete metric space.

\begin{proposition}\label{prop-geometric-closed}
The set $\scrc_g^{\alpha}([0,T],V)$ is a closed subset of $\scrc^{\alpha}([0,T],V)$.
\end{proposition}

\begin{proof}
Let $(\bfx^n)_{n \in \bbn} = (X^n,\bbx^n)_{n \in \bbn} \subset \scrc_g^{\alpha}([0,T],V)$ be a sequence, and let $\bfx = (X,\bbx) \in \scrc^{\alpha}([0,T],V)$ be such that $d_{\alpha}(\bfx^n,\bfx) \to 0$. Since the operators $\Sym \in L(V \otimes V)$ and $\otimes : V \times V \to V \otimes V$ are continuous, by Lemma \ref{lemma-estimate-supremum} for all $s,t \in [0,T]$ we obtain
\begin{align*}
\Sym(\bbx_{s,t}) &= \Sym \Big( \lim_{n \to \infty} \bbx_{s,t}^n \Big) = \lim_{n \to \infty} \Sym ( \bbx_{s,t}^n )
\\ &= \frac{1}{2} \lim_{n \to \infty} \big( X_{s,t}^n \otimes X_{s,t}^n \big) = \frac{1}{2} X_{s,t} \otimes X_{s,t},
\end{align*}
and hence $\bfx \in \scrc_g^{\alpha}([0,T],V)$.
\end{proof}

As an immediate consequence of Proposition \ref{prop-complete-metric-space} and Proposition \ref{prop-geometric-closed}, we obtain that the space of weakly geometric rough paths is a complete metric space:

\begin{corollary}
The space $\scrc_g^{\alpha}([0,T],V)$ is a complete metric space.
\end{corollary}

Recall from Exercise \ref{exercise-FV-path} that continuous paths of finite variation can always be enhanced to a rough path by means of \eqref{second-order-FV}. In this situation, the following exercise characterizes weakly geometric rough paths as those rough paths for which the integration by parts formula holds true.

\begin{exercise}
We consider the rough path $\mathbf{X} = (X,\bbx) \in \scrc^{\alpha}([0,T],V)$, where $X \in \calc^{2 \alpha}([0,T],V)$ is a continuous path of finite variation, and $\bbx : [0,T]^2 \to V \otimes V$ is defined by \eqref{second-order-FV}. Show that the following statements are equivalent:
\begin{enumerate}
\item[(i)] The integration by parts formula
\begin{align}\label{int-parts}
{\rm Sym} \bigg( \int_s^t X_r \otimes dX_r \bigg) = \frac{1}{2} \big( X_t \otimes X_t - X_s \otimes X_s \big), \quad s,t \in [0,T]
\end{align}
holds true.

\item[(ii)] We have $\mathbf{X} \in \scrc_g^{\alpha}([0,T],V)$; that is, the rough path $\mathbf{X}$ is weakly geometric.
\end{enumerate}
\end{exercise}

\begin{solution}
Let $s,t \in [0,T]$ be arbitrary. Then we have
\begin{align*}
\bbx_{s,t} &= \int_s^t X_{s,r} \otimes d X_r = \int_s^t (X_r - X_s) \otimes d X_r
\\ &= \int_s^t X_r \otimes d X_r - X_s \otimes (X_t - X_s) = \int_s^t X_r \otimes d X_r - X_s \otimes X_{s,t}.
\end{align*}
Therefore, we obtain
\begin{align*}
\Sym(\bbx_{s,t}) = \Sym \bigg( \int_s^t X_r \otimes dX_r \bigg) - \frac{1}{2} \big( X_s \otimes X_{s,t} + X_{s,t} \otimes X_s \big),
\end{align*}
and hence \eqref{chain-rule} is equivalent to
\begin{align*}
\Sym \bigg( \int_s^t X_r \otimes dX_r \bigg) = \frac{1}{2} \big( X_s \otimes X_{s,t} + X_{s,t} \otimes X_s + X_{s,t} \otimes X_{s,t} \big).
\end{align*}
Moreover, we have
\begin{align*}
&X_s \otimes X_{s,t} + X_{s,t} \otimes X_s + X_{s,t} \otimes X_{s,t}
\\ &= X_s \otimes X_{s,t} + X_{s,t} \otimes X_t
\\ &= X_s \otimes (X_t - X_s) + (X_t - X_s) \otimes X_t = X_t \otimes X_t - X_s \otimes X_s,
\end{align*}
showing that \eqref{chain-rule} and \eqref{int-parts} are equivalent.
\end{solution}

Recall that from Proposition \ref{prop-tensor-product-Rn} that in case $V = \bbr^d$ we can identify $\bbr^d \otimes \bbr^d = \bbr^{d \times d}$, and that the tensor product mapping is given by
\begin{align*}
\otimes : \bbr^d \times \bbr^d \to \bbr^{d \times d}, \quad v \otimes w = v \cdot w^{\top}.
\end{align*}
The upcoming exercise demonstrates why \eqref{int-parts} is actually called \emph{integration by parts formula}.

\begin{exercise}
Show that in case $V = \bbr^d$ the integration by parts formula \eqref{int-parts} is satisfied if and only if
\begin{align*}
\int_s^t X_r^i d X_r^j + \int_s^t X_r^j d X_r^i = X_t^i X_t^j - X_s^i X_s^j, \quad s,t \in [0,T]
\end{align*}
for all $i,j=1,\ldots,d$.
\end{exercise}

\begin{solution}
Let $s,t \in [0,T]$ and $i,j=1,\ldots,d$ be arbitrary. Then we have
\begin{align*}
\bigg( \int_s^t X_r \otimes dX_r \bigg)_{ij} = \int_s^t X_r^i d X_r^j,
\end{align*}
and hence the left-hand side of \eqref{int-parts} is given by
\begin{align*}
\Sym \bigg( \int_s^t X_r \otimes dX_r \bigg)_{ij} = \frac{1}{2} \bigg( \int_s^t X_r^i d X_r^j + \int_s^t X_r^j d X_r^i \bigg).
\end{align*}
Moreover, the right-hand side of \eqref{int-parts} is given by
\begin{align*}
\frac{1}{2} \big( X_t \otimes X_t - X_s \otimes X_s \big)_{ij} = \frac{1}{2} \big( X_t^i X_t^j - X_s^i X_s^j \big).
\end{align*}
\end{solution}

The following auxiliary result provides a connection between weakly geometric rough paths and the bracket. It is an immediate consequence of \eqref{bracket-def-2}.

\begin{lemma}\label{lemma-geometric}
For a rough path $\bfx = (X,\bbx) \in \scrc^{\alpha}([0,T],V)$ the following statements are equivalent:
\begin{enumerate}
\item[(i)] We have $\mathbf{X} \in \scrc_g^{\alpha}([0,T],V)$; that is, the rough path $\mathbf{X}$ is weakly geometric.

\item[(ii)] We have $[ \mathbf{X} ] = 0$.
\end{enumerate}
\end{lemma}

\section{Integration against rough paths}\label{sec-integral}

In this section we recall the required results about the rough integral. Let $W$ be a Banach space, and let $T \in \bbr_+$ be a finite time horizon.

\subsection{The sewing lemma}

In this section we present the sewing lemma; it is one of the key results for the construction of the rough integral. Before providing this result, let us prepare some notation.

\begin{definition}
For every $n \in \bbn$ we denote by $\Delta_T^n \subset \bbr^n$ the simplex
\begin{align*}
\Delta_T^n := \{ t \in \bbr^n : 0 \leq t_1 \leq \ldots \leq t_n \leq T \}.
\end{align*}
\end{definition}

\begin{definition}\label{def-space-for-sewing}
Let $\alpha, \beta > 0$ be arbitrary. We denote by $\calc_2^{\alpha,\beta}([0,T],W)$ the space of all functions $\Xi : \Delta_T^2 \to W$ such that $\Xi_{t,t} = 0$ for all $t \in [0,T]$ and
\begin{align}\label{norm-alpha-beta-finite}
\| \Xi \|_{\alpha,\beta} := \| \Xi \|_{\alpha} + \| \delta \Xi \|_{\beta} < \infty.
\end{align}
Here $\| \Xi \|_{\alpha}$ denotes the H\"older norm
\begin{align}\label{Hoelder-norm-sewing-1}
\| \Xi \|_{\alpha} := \sup_{\genfrac{}{}{0pt}{}{(s,t) \in \Delta_T^2}{s < t}} \frac{|\Xi_{s,t}|}{|t-s|^{\alpha}}.
\end{align}
Furthermore, the function $\delta \Xi : \Delta_T^3 \to W$ is defined as
\begin{align}\label{delta-Xi-definition}
\delta \Xi_{s,u,t} := \Xi_{s,t} - \Xi_{s,u} - \Xi_{u,t}, \quad (s,u,t) \in \Delta_T^3,
\end{align}
and its H\"older norm is defined as
\begin{align}\label{delta-Xi-norm}
\| \delta \Xi \|_{\beta} := \sup_{\genfrac{}{}{0pt}{}{(s,u,t) \in \Delta_T^3}{s < u < t}} \frac{|\delta \Xi_{s,u,t}|}{|t-s|^{\beta}}.
\end{align}
\end{definition}

After introducing the space $\calc_2^{\alpha,\beta}([0,T],W)$, we are ready to present the sewing lemma; see \cite[Lemma 4.2]{Friz-Hairer} and its proof.

\begin{lemma}[Sewing lemma]\label{lemma-sewing}
Let $\alpha,\beta \in \bbr$ be such that $0 < \alpha \leq 1 < \beta$. Then there exists a unique mapping $\cali : \calc_2^{\alpha,\beta}([0,T],W) \to \calc^{\alpha}([0,T],W)$, which is characterized as follows:
\begin{enumerate}
\item We have $(\cali \Xi)_0 = 0$ for all $\Xi \in \calc_2^{\alpha,\beta}([0,T],W)$.

\item For every $\Xi \in \calc_2^{\alpha,\beta}([0,T],W)$ there exists a constant $C = C(\beta,\| \delta \Xi \|_{\beta}) > 0$ such that
\begin{align*}
|(\cali \Xi)_{s,t} - \Xi_{s,t}| \leq C|t-s|^{\beta} \quad \text{for all $(s,t) \in \Delta_T^2$.}
\end{align*}
\end{enumerate}
Moreover, we have $\cali \in L(\calc_2^{\alpha,\beta}([0,T],W), \calc^{\alpha}([0,T],W))$, and $\cali$ is given by
\begin{align*}
(\cali \Xi)_{s,t} = \lim_{|\Pi| \to 0} \sum_{[u,v] \in \Pi} \Xi_{u,v} \quad \text{for all $\Xi \in \calc_2^{\alpha,\beta}([0,T],W)$,}
\end{align*}
where $\Pi$ denotes any partition of the interval $[s,t]$.
\end{lemma}

\subsection{The Gubinelli derivative}

In this section we introduce the Gubinelli derivative and the associated space of controlled rough paths, which will serve as the space of integrands for the rough integral later on. Let $V$ be a Banach space, and let $X \in \calc^{\alpha}([0,T], V)$ be a H\"{o}lder continuous path for some index $\alpha \in (\frac{1}{3},\frac{1}{2}]$. It could, in particular be the first component of a rough path $\mathbf{X} = (X,\bbx) \in \scrc^{\alpha}([0,T],V)$, but presently the second order process $\bbx$ is not required. Furthermore, let $\bar{W}$ be another Banach space. Typically, we will choose $\bar{W} = L(V,W)$.

\begin{definition}\label{def-Gubinelli-derivative}
We introduce the following notions:
\begin{enumerate}
\item We say that a path $Y \in \calc^{\alpha}([0,T],\bar{W})$ is \emph{controlled} by $X$ if there exists $Y' \in \calc^{\alpha}([0,T], L(V,\bar{W}))$ such that for the remainder term $R^Y : [0,T]^2 \to \bar{W}$ given by
\begin{align}\label{remainder}
R_{s,t}^Y := Y_{s,t} - Y_s' X_{s,t} \quad \text{for all $s,t \in [0,T]$}
\end{align}
we have $R^Y \in \calc_2^{2 \alpha}([0,T]^2,\bar{W})$.

\item This defines the space of \emph{controlled rough paths}, and we write
\begin{align*}
(Y,Y') \in \scrd_X^{2 \alpha}([0,T], \bar{W}).
\end{align*}
\item We call $Y'$ a \emph{Gubinelli derivative} of $Y$ (with respect to $X$).
\end{enumerate}
\end{definition}

\begin{remark}
Alternatively, we can consider the remainder term $r^Y : [0,T] \to \bar{W}$ with an arbitrary starting value $r_0^Y \in \bar{W}$ and increments given by
\begin{align*}
r_{s,t}^Y = Y_{s,t} - Y_s' X_{s,t} \quad \text{for all $s,t \in [0,T]$.}
\end{align*}
According to Remark \ref{rem-Hoelder} we have $R^Y \in \calc_2^{2 \alpha}([0,T]^2,\bar{W})$ if and only if
\begin{align*}
r^Y \in \calc^{2 \alpha}([0,T],\bar{W}).
\end{align*}
\end{remark}

\begin{remark}
Note that, in contrast to the space $\scrc^{\alpha}([0,T],V)$ of rough paths, the space $\scrd_X^{2 \alpha}([0,T], \bar{W})$ of controlled rough paths actually is a vector space. Indeed, consider the product space
\begin{align*}
Q^{\alpha}([0,T],\bar{W}) := \calc^{\alpha}([0,T],\bar{W}) \times \calc^{\alpha}([0,T], L(V,\bar{W})).
\end{align*}
Then $\scrd_X^{2 \alpha}([0,T], \bar{W})$ consists of all $(Y,Y') \in Q^{\alpha}([0,T],\bar{W})$ such that the remainder term $R^Y$ given by \eqref{remainder} satisfies $R^Y \in \calc_2^{2 \alpha}([0,T]^2,\bar{W})$. Thus, the set $\scrd_X^{2 \alpha}([0,T], \bar{W})$ is a subspace of the vector space $Q^{\alpha}([0,T],\bar{W})$, and therefore itself a vector space.
\end{remark}

The following exercise shows that the Gubinelli derivative of a controlled rough path does not need to be unique.

\begin{exercise}\label{exercise-Gubinelli-2-alpha}
Suppose we even have $X \in \calc^{2 \alpha}([0,T], V)$, and let $Y \in \calc^{2\alpha}([0,T],\bar{W})$ be arbitrary. Show that for every path $Y' \in \calc^{\alpha}([0,T], L(V,\bar{W}))$ we have
\begin{align*}
(Y,Y') \in \scrd_X^{2 \alpha}([0,T], \bar{W}).
\end{align*}
\end{exercise}

\begin{solution}
For all $s,t \in [0,T]$ we have
\begin{align*}
| R_{s,t}^Y | \leq |Y_{s,t}| + | Y_s' | \cdot |X_{s,t}| \leq \| Y \|_{2 \alpha} |t-s|^{2 \alpha} + \| Y' \|_{\infty} \| X \|_{2 \alpha} |t-s|^{2 \alpha},
\end{align*}
and thus
\begin{align*}
\| R^Y \|_{2\alpha} \leq \| Y \|_{2\alpha} + \| Y' \|_{\infty} \| X \|_{2 \alpha} < \infty.
\end{align*}
\end{solution}

The following exercise establishes the connection to the classical derivative.

\begin{exercise}
Suppose that $V = \bbr$ and $X_t = t$ for all $t \in [0,T]$. Moreover, consider a smooth path $Y \in C^1([0,T],\bar{W})$ with its classical derivative $Y' \in C([0,T],\bar{W})$. Show that
\begin{align*}
(Y,Y') \in \scrd_X^{2 \alpha}([0,T], \bar{W})
\end{align*}
for all $\alpha \in (\frac{1}{3},\frac{1}{2}]$, where we use the identification $\bar{W} \cong L(\bbr,\bar{W})$
\end{exercise}

\begin{solution}
By first order Taylor expansion we have
\begin{align*}
Y_{s,t} = Y_s' X_{s,t} + r_ {s,t}^Y (t-s) \quad \text{for all $s,t \in [0,T]$}
\end{align*}
with a continuous map $r^Y : [0,T]^2 \to \bar{W}$. Hence, defining $R^Y : [0,T]^2 \to \bar{W}$ as
\begin{align*}
R_{s,t}^Y := r_{s,t}^Y (t-s), \quad s,t \in [0,T],
\end{align*}
we have
\begin{align*}
R_{s,t}^Y = Y_{s,t} - Y_s' X_{s,t} \quad \text{for all $s,t \in [0,T]$,}
\end{align*}
showing that $R^Y$ is just the corresponding remainder term, as defined in \eqref{remainder}. The path $Y$ is Lipschitz continuous, because it is of class $C^1$. Therefore, and since $X_t = t$ for all $t \in [0,T]$, for all $s,t \in [0,T]$ we obtain
\begin{align*}
| R_{s,t}^Y | &\leq |Y_{s,t}| + | Y_s' | \cdot |X_{s,t}| \leq | Y |_{\Lip} |t-s| + \| Y' \|_{\infty} |t-s|
\\ &\leq | Y |_{\Lip} \, T^{1-2\alpha} |t-s|^{2 \alpha} + \| Y' \|_{\infty} T^{1-2\alpha} |t-s|^{2 \alpha},
\end{align*}
and hence
\begin{align*}
\| R^Y \|_{2 \alpha} \leq ( | Y |_{\Lip} + \| Y' \|_{\infty} ) T^{1-2\alpha} < \infty.
\end{align*}
\end{solution}

Now, let us define a topology on the space of controlled rough paths.

\begin{definition}\label{def-seminorm-Gubinelli}
We endow the space $\scrd_X^{2 \alpha}([0,T], \bar{W})$ with the seminorm
\begin{align}\label{seminorm-Gubinelli}
\| Y,Y' \|_{X,2\alpha} := \| Y' \|_{\alpha} + \| R^Y \|_{2 \alpha}.
\end{align}
\end{definition}

Then the space $\scrd_X^{2 \alpha}([0,T], \bar{W})$ endowed with the norm
\begin{align*}
\interleave Y,Y' \interleave_{X,2\alpha} := |Y_0| + |Y_0'| + \| Y,Y' \|_{X,2\alpha}
\end{align*}
is a Banach space. We will also consider the seminorm
\begin{align*}
| Y,Y' |_{X,2\alpha} := |Y_0'| + \| Y,Y' \|_{X,2\alpha}.
\end{align*}

\begin{exercise}\label{exercise-reg-derivative-zero}
Show that for each $Y \in \calc^{2 \alpha}([0,T],\bar{W})$ we have
\begin{align}\label{second-comp-zero}
(Y,0) \in \scrd_X^{2 \alpha}([0,T],\bar{W})
\end{align}
with the seminorm given by
\begin{align}\label{norm-for-second-comp-zero}
\| Y,0 \|_{X,2\alpha} = \| Y \|_{2 \alpha}.
\end{align}
\end{exercise}

\begin{solution}
We have $Y' = 0$, and hence $R_{s,t}^Y = Y_{s,t}$ for all $s,t \in [0,T]$. Thus, by Remark \ref{rem-Hoelder} we deduce that $R^Y \in \calc_2^{2 \alpha}([0,T]^2,\bar{W})$, showing \eqref{second-comp-zero} and \eqref{norm-for-second-comp-zero}.
\end{solution}

We proceed with some inequalities regarding the seminorms and the norms of controlled rough paths.

\begin{lemma}\label{lemma-norm-of-Y}
For all $(Y,Y') \in \scrd_X^{2 \alpha}([0,T],\bar{W})$ we have
\begin{align}\label{norm-Y-1}
\| Y' \|_{\infty} &\leq C | Y,Y' |_{X,2\alpha},
\\ \label{norm-Y-2} \| Y \|_{\alpha} &\leq C | Y,Y' |_{X,2\alpha} ( \| X \|_{\alpha} + T^{\alpha} ),
\\ \label{norm-Y-3} \| Y \|_{\infty} &\leq C \interleave Y,Y' \interleave_{X,2\alpha} ( \| X \|_{\alpha} + 2 ),
\end{align}
where the constant $C > 0$ depends on $\alpha$ and $T$. Furthermore, for $T \leq 1$ the estimates \eqref{norm-Y-1}--\eqref{norm-Y-3} hold true with $C=1$.
\end{lemma}

\begin{proof}
For each $t \in [0,T]$ we have
\begin{align*}
|Y_t'| \leq |Y_0'| + |Y_t' - Y_0'| \leq |Y_0'| + \| Y' \|_{\alpha} t^{\alpha},
\end{align*}
and hence
\begin{align*}
\| Y' \|_{\infty} &\leq |Y_0'| + \| Y' \|_{\alpha} T^{\alpha} \leq \max \{ 1,T^{\alpha} \} \, | Y,Y' |_{X,2\alpha},
\end{align*}
proving \eqref{norm-Y-1}. Now, let $s,t \in [0,T]$ be arbitrary. Recalling that $Y_{s,t} = Y_s' X_{s,t} + R_{s,t}^Y$, we have
\begin{align*}
| Y_{s,t} | &\leq \| Y' \|_{\infty} |X_{s,t}| + |R_{s,t}^Y|
\\ &\leq \| Y' \|_{\infty} \| X \|_{\alpha} |t-s|^{\alpha} + \| R_{s,t}^Y \|_{2\alpha} |t-s|^{2\alpha}
\\ &\leq \| Y' \|_{\infty} \| X \|_{\alpha} |t-s|^{\alpha} + \| R_{s,t}^Y \|_{2\alpha} T^{\alpha} |t-s|^{\alpha}.
\end{align*}
Therefore, we obtain
\begin{align*}
\| Y \|_{\alpha} &\leq \| Y' \|_{\infty} \| X \|_{\alpha} + \| R^Y \|_{2\alpha} T^{\alpha}
\\ &\leq \max \{ 1,T^{\alpha} \} \, | Y,Y' |_{X,2\alpha} \| X \|_{\alpha} + \| R^Y \|_{2\alpha} T^{\alpha}
\\ &\leq  \max \{ 1,T^{\alpha} \} | Y,Y' |_{X,2\alpha} ( \| X \|_{\alpha} + T^{\alpha} ),
\end{align*}
showing \eqref{norm-Y-2}. Using this estimate we deduce that
\begin{align*}
\| Y \|_{\infty} &\leq |Y_0| + \| Y \|_{\alpha} T^{\alpha}
\\ &\leq \interleave Y,Y' \interleave_{X,2\alpha} + \max \{ 1,T^{\alpha} \} T^{\alpha} \interleave Y,Y' \interleave_{X,2\alpha} ( \| X \|_{\alpha} + T^{\alpha} )
\\ &\leq \interleave Y,Y' \interleave_{X,2\alpha} \big( \max \{ 1,T^{\alpha} \} T^{\alpha} \| X \|_{\alpha} + 1 + \max \{ 1,T^{\alpha} \} T^{2\alpha} \big),
\end{align*}
proving \eqref{norm-Y-3}.
\end{proof}

\begin{lemma}\label{lemma-pairs}
Let $E$ be another Banach space. Then for all
\begin{align*}
(Y,Y') \in \scrd_X^{2\alpha}([0,T],W) \quad \text{and} \quad (Z,Z') \in \scrd_X^{2\alpha}([0,T],E)
\end{align*}
we have the relation
\begin{align}\label{vector-in-D}
((Y,Z),(Y,Z)') \in \scrd_X^{2\alpha}([0,T],W \times E),
\end{align}
where the Gubinelli derivative is given by $(Y,Z)' = (Y',Z')$, and we have the estimates
\begin{align}\label{pairs-1}
\| (Y,Z), (Y,Z)' \|_{X,2\alpha} &\leq \| Y,Y' \|_{X,2\alpha} + \| Z,Z' \|_{X,2\alpha},
\\ \label{pairs-2} | (Y,Z), (Y,Z)' |_{X,2\alpha} &\leq | Y,Y' |_{X,2\alpha} + | Z,Z' |_{X,2\alpha},
\\ \label{pairs-3} \interleave (Y,Z), (Y,Z)' \interleave_{X,2\alpha} &\leq \interleave Y,Y' \interleave_{X,2\alpha} + \interleave Z,Z' \interleave_{X,2\alpha}.
\end{align}
\end{lemma}

\begin{proof}
Let $s,t \in [0,T]$ be arbitrary. Then we have
\begin{align*}
|(Y_t,Z_t)-(Y_s,Z_s)| &= |(Y_t-Y_s,Z_t-Z_s)| = |Y_t-Y_s| + |Z_t-Z_s|
\\ &\leq \big( \| Y \|_{\alpha} + \| Z \|_{\alpha} \big) |t-s|^{\alpha},
\end{align*}
showing that $(Y,Z) \in \calc^{\alpha}([0,T],W \times E)$.
Similarly, we have
\begin{align*}
|(Y_t',Z_t')-(Y_s',Z_s')| &= |(Y_t'-Y_s',Z_t'-Z_s')| = |Y_t'-Y_s'| + |Z_t'-Z_s'|
\\ &\leq \big( \| Y' \|_{\alpha} + \| Z' \|_{\alpha} \big) |t-s|^{\alpha},
\end{align*}
showing that $(Y',Z') \in \calc^{\alpha}([0,T],L(V,W \times E))$. Furthermore, we have
\begin{align*}
R_{s,t}^{(Y,Z)} = (Y_{s,t},Z_{s,t}) - (Y_s',Z_s') X_{s,t} = (R_{s,t}^Y,R_{s,t}^Z),
\end{align*}
and hence
\begin{align*}
|R_{s,t}^{(Y,Z)}| = |R_{s,t}^Y| + |R_{s,t}^Z| \leq \big( \| R^Y \|_{2\alpha} + \| R^Z \|_{2\alpha} \big) |t-s|^{2\alpha},
\end{align*}
proving \eqref{vector-in-D}. Moreover, the previous calculations reveal that
\begin{align*}
\| (Y',Z') \|_{\alpha} &\leq \| Y' \|_{\alpha} + \| Z' \|_{\alpha},
\\ \| R^{(Y,Z)} \|_{2\alpha} &\leq \| R^Y \|_{2\alpha} + \| R^Z \|_{2\alpha}.
\end{align*}
Therefore, the estimates \eqref{pairs-1}, \eqref{pairs-2} and \eqref{pairs-3} are an immediate consequence.
\end{proof}

\subsection{The Riemann-Stieltjes integral}

In this section we demonstrate how the Riemann-Stieltjes integral gives rise to a controlled rough path, and that the associated integral operator is a continuous linear operator between Banach spaces.

\begin{lemma}\label{lemma-reg-conv-Hoelder}
Let $Y : [0,T] \to W$ be measurable and bounded. We define the path $Z : [0,T] \to W$ as
\begin{align*}
Z_t := \int_0^t Y_s \, ds, \quad t \in [0,T].
\end{align*}
Then we have
\begin{align*}
|Z_{s,t}| \leq \| Y \|_{\infty} |t-s| \quad \text{for all $s,t \in [0,T]$.}
\end{align*}
\end{lemma}

\begin{proof}
Let $s,t \in [0,T]$ with $s \leq t$ be arbitrary. Then we have
\begin{align*}
Z_{s,t} = \int_0^t Y_r \, dr - \int_0^s Y_r \, dr = \int_s^t Y_r \, dr.
\end{align*}
Therefore, we obtain
\begin{align*}
|Z_{s,t}| \leq \bigg| \int_s^t Y_r \, dr \bigg| \leq \int_s^t |Y_r| \, dr \leq \|Y\|_{\infty} |t-s|,
\end{align*}
completing the proof.
\end{proof}

Now, let us fix a H\"{o}lder continuous path $X \in \calc^{\alpha}([0,T],V)$ for some index $\alpha \in (\frac{1}{3},\frac{1}{2}]$ with values in a Banach space $V$.

\begin{proposition}\label{prop-reg-conv-Hoelder}
The integral operator
\begin{align*}
\mathscr{I} : \big( C([0,T],W), \| \cdot \|_{\infty} \big) \to \big( \scrd_X^{2\alpha}([0,T],W), \interleave \cdot \interleave_{X,2\alpha} \big), \quad Y \mapsto \bigg( \int_0^{\cdot} Y_s \, ds, 0 \bigg)
\end{align*}
is a continuous linear operator between Banach spaces with operator norm bounded by $\| \mathscr{I} \| \leq T^{1-2\alpha}$.
\end{proposition}

\begin{proof}
Let $Y \in C([0,T],W)$ be arbitrary, and set $(Z,Z') := \mathscr{I}(Y)$. By Lemma \ref{lemma-reg-conv-Hoelder}, for all $s,t \in [0,T]$ we have
\begin{align*}
|Z_{s,t}| \leq \| Y \|_{\infty} |t-s| \leq T^{1-2\alpha} \| Y \|_{\infty} |t-s|^{2\alpha},
\end{align*}
showing that $Z \in \calc^{2\alpha}([0,T],W)$ with
\begin{align*}
\| Z \|_{2\alpha} \leq T^{1-2\alpha} \| Y \|_{\infty}.
\end{align*}
Note that $Z_0 = 0$ and $Z_0' = 0$. Therefore, by Exercise \ref{exercise-reg-derivative-zero} we obtain $(Z,Z') \in \scrd_X^{2\alpha}([0,T],W)$ and
\begin{align*}
\interleave Z,Z' \interleave_{X,2\alpha} = \| Z,Z' \|_{X,2\alpha} = \| Z \|_{2 \alpha} \leq T^{1-2\alpha} \| Y \|_{\infty},
\end{align*}
completing the proof.
\end{proof}

\subsection{The rough integral}

In this section we introduce the rough integral, and show that the associated integral operator is a continuous linear operator between Banach spaces. Let $V$ be a Banach space, and let $\bfx = (X,\bbx) \in \scrc^{\alpha}([0,T],V)$ be a rough path for some index $\alpha \in (\frac{1}{3},\frac{1}{2}]$. For any controlled rough path $(Y,Y') \in \scrd_X^{2 \alpha}([0,T],L(V,W))$ we consider the candidate
\begin{align}\label{candidate}
\int_0^1 Y_s \, d \bfx_s := \lim_{|\Pi| \to 0} \sum_{[s,t] \in \Pi} ( Y_s X_{s,t} + Y_s' \bbx_{s,t} )
\end{align}
for the \emph{rough integral} of $Y$ against $\mathbf{X}$.

\begin{remark}
In view of \eqref{candidate}, note that for all $s,t \in [0,T]$ we have
\begin{align*}
Y_s X_{s,t} + Y_s' \bbx_{s,t} \in W,
\end{align*}
because, using the identifications from Proposition \ref{prop-identifications}, we have
\begin{align*}
&Y_s \in L(V,W), \quad X_{s,t} \in V,
\\ &Y_s' \in L(V,L(V,W)) \cong L(V \otimes V,W) \quad \text{and} \quad \bbx_{s,t} \in V \otimes V.
\end{align*}
\end{remark}

\begin{remark}\label{rem-identifications-fin-dim}
Consider the finite dimensional situation $V = \bbr^d$ and $W = \bbr^m$. By Proposition \ref{prop-tensor-product-Rn} we have $\bbr^d \otimes \bbr^d \cong \bbr^{d \times d}$. Furthermore, by Proposition \ref{prop-lin-operator-matrix} we have $L(\bbr^d,\bbr^m) \cong \bbr^{m \times d}$, and hence by Proposition \ref{prop-identifications} we obtain
\begin{align*}
L(\bbr^d,L(\bbr^d,\bbr^m)) \cong L(\bbr^d \otimes \bbr^d, \bbr^m) \cong L(\bbr^{d \times d}, \bbr^m) \cong \bbr^{m \times d \times d}.
\end{align*}
Using these identifications we have
\begin{align*}
&Y_s \in \bbr^{m \times d}, \quad X_{s,t} \in \bbr^d,
\\ &Y_s' \in \bbr^{m \times d \times d} \quad \text{and} \quad \bbx_{s,t} \in \bbr^{d \times d},
\end{align*}
and the products in \eqref{candidate} are understood as matrix multiplications.
\end{remark}

The following exercise, which is essentially a consequence of Chen's relation \eqref{Chen-relation}, will be crucial in order to show that the rough integral \eqref{candidate} converges.

\begin{exercise}\label{exercise-sewing-lemma}
Let $(Y,Y') \in \scrd_X^{2 \alpha}([0,T],L(V,W))$ be a controlled rough path. We define the mapping
\begin{align}\label{Xi-mapping}
\Xi : \Delta_T^2 \to W, \quad \Xi_{s,t} := Y_s X_{s,t} + Y_s' \bbx_{s,t}.
\end{align}
Show that for all $(s,u,t) \in \Delta_T^3$ we have
\begin{align}\label{identity-for-integral}
\delta \Xi_{s,u,t} = - R_{s,u}^Y X_{u,t} - Y_{s,u}' \bbx_{u,t},
\end{align}
where $\delta \Xi : \Delta_T^3 \to W$ is given by \eqref{delta-Xi-definition}.
\end{exercise}

\begin{solution}
By Chen's relation \eqref{Chen-relation} we obtain
\begin{align*}
\delta \Xi_{s,u,t} &= \Xi_{s,t} - \Xi_{s,u} - \Xi_{u,t}
\\ &= Y_s X_{s,t} + Y_s' \bbx_{s,t} - Y_s X_{s,u} - Y_s' \bbx_{s,u} - Y_u X_{u,t} - Y_u' \bbx_{u,t}
\\ &= Y_s (X_t - X_s) + Y_s' \bbx_{s,t} - Y_s (X_u - X_s) - Y_s' \bbx_{s,u} - Y_u (X_t - X_u) - Y_u' \bbx_{u,t}
\\ &= Y_s (X_t - X_u) + Y_s' \bbx_{s,t} - Y_s' \bbx_{s,u} - Y_u (X_t - X_u) - Y_u' \bbx_{u,t}
\\ &= Y_{u,s} X_{u,t} + Y_s' (\bbx_{s,t} - \bbx_{s,u}) - Y_u' \bbx_{u,t}
\\ &= Y_{u,s} X_{u,t} + Y_s' (\bbx_{u,t} + X_{s,u} \otimes X_{u,t}) - Y_u' \bbx_{u,t}
\\ &= - ( Y_{s,u} - Y_s' X_{s,u} ) X_{u,t} - (Y_u' - Y_s') \bbx_{u,t}
\\ &= - R_{s,u}^Y X_{u,t} - Y_{s,u}' \bbx_{u,t}.
\end{align*}
For this calculation, recall from Proposition \ref{prop-identifications} that
\begin{align*}
L(V,L(V,W)) \cong L(V \otimes V,W),
\end{align*}
and that accordingly a linear operator $\Phi \in L(V,L(V,W))$ may be identified with the linear operator $\Psi \in L(V \otimes V,W)$ given by
\begin{align*}
\Psi(v \otimes w) = (\Phi v)w \quad \text{for all $v,w \in V$.}
\end{align*}
\end{solution}

Now, after establishing the identity \eqref{identity-for-integral}, we can prove convergence of the rough integral \eqref{candidate} by using the sewing lemma:

\begin{theorem}[Gubinelli]\label{thm-Gubinelli}
For each controlled rough path $$(Y,Y') \in \scrd_X^{2 \alpha}([0,T],L(V,W))$$ the $W$-valued rough integral \eqref{candidate} exists, and for all $s,t \in [0,T]$ we have
\begin{align}\label{estimate-third-order}
&\bigg| \int_s^t Y_r \, d\bfx_r - Y_s X_{s,t} - Y_s' \bbx_{s,t} \bigg|
\leq C \big( \| X \|_{\alpha} \| R^Y \|_{2 \alpha} + \| \bbx \|_{2 \alpha} \| Y' \|_{\alpha} \big) |t-s|^{3 \alpha}
\end{align}
with a constant $C > 0$ depending on $\alpha$.
\end{theorem}

\begin{proof}
Setting $\beta := 3 \alpha$, we have $0 < \alpha \leq 1 < \beta$. Let $(Y,Y') \in \scrd_X^{2 \alpha}([0,T],L(V,W))$ be an arbitrary controlled rough path, and consider the mapping \eqref{Xi-mapping}. We will show that $\Xi \in \calc_2^{\alpha,\beta}([0,T],W)$. Indeed, by \eqref{Chen-rule-1} we have $\Xi_{t,t} = 0$ for all $t \in [0,T]$. Furthermore, for all $s,t \in [0,T]$ with $s \leq t$ we have
\begin{align*}
| \Xi_{s,t} | &\leq | Y_s | \cdot | X_{s,t} | + | Y_s' | \cdot |\bbx_{s,t} |
\\ &\leq \| Y \|_{\infty} \| X \|_{\alpha} |t-s|^{\alpha} + \| Y' \|_{\infty} \| \bbx \|_{\alpha} |t-s|^{\alpha},
\end{align*}
and hence
\begin{align*}
\| \Xi \|_{\alpha} \leq \| X \|_{\alpha} \| Y \|_{\infty} + \| \bbx \|_{\alpha} \| Y' \|_{\infty} < \infty.
\end{align*}
Moreover, according to Exercise \ref{exercise-sewing-lemma}, for all $s,u,t \in [0,T]$ with $s \leq u \leq t$ we have
\begin{align*}
| \delta \Xi_{s,u,t} | &\leq |R_{s,u}^Y| \cdot |X_{u,t}| + | Y_{s,u}' | \cdot | \bbx_{u,t} |
\\ &\leq \| R^Y \|_{2 \alpha} |u-s|^{2 \alpha} \cdot \| X \|_{\alpha} |t-u|^{\alpha} + \| Y' \|_{\alpha} |u-s|^{\alpha} \cdot \| \bbx \|_{2 \alpha} |t-u|^{2 \alpha}
\\ &\leq \| R^Y \|_{2 \alpha} \| X \|_{\alpha} |t-s|^{3 \alpha} + \| Y' \|_{\alpha} \| \bbx \|_{2 \alpha} |t-s|^{3 \alpha},
\end{align*}
and hence
\begin{align*}
\| \delta \Xi \|_{\beta} \leq \| X \|_{\alpha} \| R^Y \|_{2 \alpha} + \| \bbx \|_{2 \alpha} \| Y' \|_{\alpha} < \infty,
\end{align*}
showing that $\Xi \in \calc_2^{\alpha,\beta}([0,T],W)$. Therefore, applying the sewing lemma (Lemma \ref{lemma-sewing}) completes the proof.
\end{proof}

\begin{definition}
We call \eqref{candidate} the \emph{rough integral} (or \emph{Gubinelli integral}) of $Y$ against $\mathbf{X}$.
\end{definition}

We proceed with some properties of the Gubinelli integral. As we will see, it gives rise to a continuous linear operator between Banach spaces.

\begin{proposition}\label{prop-Gubinelli}
For all $(Y,Y') \in \scrd_X^{2 \alpha}([0,T],L(V,W))$ we have
\begin{align*}
(Z,Z') := \bigg( \int_0^{\cdot} Y_s \, d \bfx_s, Y \bigg) \in \scrd_X^{2 \alpha}([0,T],W),
\end{align*}
and the estimate
\begin{align*}
\| Z,Z' \|_{X,2 \alpha} &\leq \| Y \|_{\alpha} + \| Y' \|_{\infty} \| \bbx \|_{2 \alpha}
\\ &\quad + C T^{\alpha} \big( \| X \|_{\alpha} \| R^Y \|_{2 \alpha} + \| \bbx \|_{2 \alpha} \| Y' \|_{\alpha} \big)
\end{align*}
with a constant $C > 0$ depending on $\alpha$.
\end{proposition}

\begin{proof}
We define $\Sigma : [0,T]^2 \to W$ as $\Sigma_{s,t} := Z_t - Z_s$ for all $s,t \in [0,T]$. Furthermore, we define $\Sigma^1, \Sigma^2, \Sigma^3 : [0,T]^2 \to W$ as
\begin{align*}
\Sigma_{s,t}^1 &:= Y_s X_{s,t}, \quad \Sigma_{s,t}^2 := Y_s' \bbx_{s,t} \quad \text{and}
\\ \Sigma_{s,t}^3 &:= \int_s^t Y_s \, d\bfx_r - Y_s X_{s,t} - Y_s' \bbx_{s,t}.
\end{align*}
Then we have $\Sigma = \Sigma^1 + \Sigma^2 + \Sigma^3$ as well as $\Sigma^1 \in \calc_2^{\alpha}([0,T]^2,W)$ and $\Sigma^2 \in \calc_2^{2\alpha}([0,T]^2,W)$. Moreover, by Theorem \ref{thm-Gubinelli} we have $\Sigma^3 \in \calc_2^{3\alpha}([0,T]^2,W)$. This shows $\Sigma \in \calc_2^{\alpha}([0,T]^2,W)$, and hence $Z \in \calc^{\alpha}([0,T],W)$ due to Remark \ref{rem-Hoelder}.

Since $Z' = Y$, it is a path $Z' \in \calc^{\alpha}([0,T], L(V,W))$. Concerning the remainder term, for all $s,t \in [0,T]$ we have
\begin{align*}
R_{s,t}^Z = Z_{s,t} - Z_s' X_{s,t} = Z_{s,t} - Y_s X_{s,t} = Y_s' \bbx_{s,t} + ( Z_{s,t} - Y_s X_{s,t} - Y_s' \bbx_{s,t} ),
\end{align*}
and hence, by Theorem \ref{thm-Gubinelli} we obtain
\begin{align*}
\| R^Z \|_{2 \alpha} \leq \| Y' \|_{\infty} \| \bbx \|_{2 \alpha} + C T^{\alpha} \big( \| X \|_{\alpha} \| R^Y \|_{2 \alpha} + \| \bbx \|_{2 \alpha} \| Y' \|_{\alpha} \big)
\end{align*}
with a constant $C > 0$ depending on $\alpha$. It follows that $(Z,Z') \in \scrd_X^{2 \alpha}([0,T],W)$ and
\begin{align*}
\| Z,Z' \|_{X,2\alpha} &= \| Z' \|_{\alpha} + \| R^Z \|_{2 \alpha} = \| Y \|_{\alpha} + \| R^Z \|_{2 \alpha}
\\ &\leq \| Y \|_{\alpha} + \| Y' \|_{\infty} \| \bbx \|_{2 \alpha} + C T^{\alpha} \big( \| X \|_{\alpha} \| R^Y \|_{2 \alpha} + \| \bbx \|_{2 \alpha} \| Y' \|_{\alpha} \big),
\end{align*}
completing the proof.
\end{proof}

\begin{proposition}\label{prop-conv-rough}
The rough integral operator
\begin{align*}
\mathscr{J} : \big( \scrd_X^{2 \alpha}([0,T],L(V,W)), \interleave \cdot \interleave_{X,2\alpha} \big) \to \big( \scrd_X^{2 \alpha}([0,T],W), \interleave \cdot \interleave_{X,2\alpha} \big)
\end{align*}
given by
\begin{align*}
(Y,Y') \mapsto \bigg( \int_0^{\cdot} Y_s \, d \bfx_s, Y \bigg)
\end{align*}
is a continuous linear operator between Banach spaces with operator norm bounded by
\begin{align*}
\| \mathscr{J} \| \leq C (  \interleave \mathbf{X} \interleave_{\alpha} + T^{\alpha} ),
\end{align*}
where the constant $C > 0$ depends on $\alpha$ and $T$, but does not depend on $T \leq 1$.
\end{proposition}

\begin{proof}
Let $(Y,Y') \in \scrd_X^{2 \alpha}([0,T],L(V,W))$ be arbitrary, and set $(Z,Z') := \mathscr{J}(Y,Y')$. Noting that $Z_0 = 0$ and $Z_0' = Y_0$, by Proposition \ref{prop-Gubinelli} and Lemma \ref{lemma-norm-of-Y} we obtain
\begin{align*}
\interleave Z,Z' \interleave_{X,2\alpha} &= |Z_0| + |Z_0'| + \| Z,Z' \|_{X,2\alpha}
\\ &\leq |Y_0| + \| Y \|_{\alpha} + \| Y' \|_{\infty} \| \bbx \|_{2\alpha} + C \big( \| X \|_{\alpha} \| R^Y \|_{2 \alpha} + \| \bbx \|_{2 \alpha} \| Y' \|_{\alpha} \big)
\\ &\lesssim C (  \interleave \mathbf{X} \interleave_{\alpha} + T^{\alpha} ) \interleave Y,Y' \interleave_{X,2\alpha},
\end{align*}
where the constant $C > 0$ depends on $\alpha$ and $T$, but does not depend on $T \leq 1$.
\end{proof}

\section{Compositions of paths with functions}\label{sec-composition}

In this section we deal with compositions of regular and rough paths with functions. Let us fix a finite time horizon $T \in \bbr_+$, and let $W,\bar{W}$ be Banach spaces.

\subsection{Compositions of regular paths with functions}\label{sec-comp-reg}

In this section we briefly deal with compositions of regular paths with functions.

\begin{definition}
We denote by $\Lip([0,T] \times W, \bar{W})$ the space of all continuous functions $f_0 : [0,T] \times W \to \bar{W}$ such that for some constant $L > 0$ we have
\begin{align*}
| f_0(t,y) - f_0(t,z) | \leq L |y-z| \quad \text{for all $t \in [0,T]$ and $y,z \in W$.}
\end{align*}
\end{definition}

Note that for each $f_0 \in \Lip([0,T] \times W, \bar{W})$ the seminorm
\begin{align*}
\| f_0 \|_{\Lip} := \sup_{t \in [0,T]} |f_0(t,0)| + \sup_{t \in [0,T]} | f_0(t,\cdot) |_{\Lip}
\end{align*}
is finite.

\begin{remark}
Let $g_0 \in \Lip(W,\bar{W})$ be arbitrary, and define the function
\begin{align*}
f_0 : [0,T] \times W \to \bar{W}, \quad f_0(t,y) := g_0(y).
\end{align*}
Then we have $f_0 \in \Lip([0,T] \times W, \bar{W})$ with $\| f_0 \|_{\Lip} = |g_0(0)| + | g_0 |_{\Lip}$.
\end{remark}

\begin{lemma}\label{lemma-lin-growth}
Let $f_0 \in \Lip([0,T] \times W, \bar{W})$ be arbitrary. Then we have
\begin{align}\label{Lip-1}
| f_0(t,y) - f_0(t,z) | &\leq \| f_0 \|_{\Lip} |y-z|, \quad \text{$y,z \in W$ and $t \in [0,T]$,}
\\ \label{Lip-2} |f_0(t,y)| &\leq \| f_0 \|_{\Lip} (1 + |y|), \quad \text{$y \in W$ and $t \in [0,T]$.}
\end{align}
\end{lemma}

\begin{proof}
Setting $L := \sup_{t \in [0,T]} | f_0(t,\cdot) |_{\Lip}$, we have
\begin{align*}
| f_0(t,y) - f_0(t,z) | &\leq L |y-z|, \quad \text{$y,z \in W$ and $t \in [0,T]$,}
\end{align*}
showing \eqref{Lip-1}. Now, for all $y \in W$ and $t \in [0,T]$ we obtain
\begin{align*}
|f_0(t,y)| &\leq |f_0(t,y) - f_0(t,0)| + |f_0(t,0)|
\\ &\leq L |y| + \sup_{s \in [0,T]} |f_0(s,0)| \leq \| f_0 \|_{\Lip} (1 + |y|),
\end{align*}
proving \eqref{Lip-2}.
\end{proof}

For a path $Y : [0,T] \to W$ and a function $f_0 : [0,T] \times W \to \bar{W}$ we denote by $f_0(Y) : [0,T] \to \bar{W}$ the new path
\begin{align}\label{new-path-f0}
f_0(Y)_t := f_0(t,Y_t), \quad t \in [0,T].
\end{align}
If the path $Y$ is continuous and $f_0 \in \Lip([0,T] \times W, \bar{W})$, then the path $f_0(Y)$ is continuous, and hence bounded.

\subsection{Compositions of rough paths with time-dependent functions}

In this section we consider compositions of controlled rough paths with time-dependent smooth functions.

\begin{definition}
Let $\gamma \in (0,1]$ and $n \in \bbn$ be arbitrary. We denote by $C_b^{\gamma,n}([0,T] \times W, \bar{W})$ the space of all functions $f : [0,T] \times W \to \bar{W}$ such that:
\begin{enumerate}
\item We have $f(t,\cdot) \in C_b^n(W,\bar{W})$ for all $t \in [0,T]$.

\item We have $f(\cdot,y),\ldots,D_y^{n-1}f(\cdot,y) \in \calc^{\gamma}([0,T],\bar{W})$ for each $y \in W$.

\item We have
\begin{align*}
\| f \|_{C_b^{\gamma,n}} := \sup_{t \in [0,T]} \| f(t,\cdot) \|_{C_b^n} + \sum_{k=0}^{n-1} \sup_{y \in W} \| D_y^k f(\cdot,y) \|_{\gamma} < \infty.
\end{align*}
\end{enumerate}
\end{definition}

\begin{remark}\label{rem-f-time-dependent}
Let $g \in C_b^n(W,\bar{W})$ be arbitrary, and define the function
\begin{align*}
f : [0,T] \times W \to \bar{W}, \quad f(t,y) := g(y).
\end{align*}
Then we have $f \in C_b^{\gamma,n}([0,T] \times W, \bar{W})$ with $\| f \|_{C_b^{\gamma,n}} = \| g \|_{C_b^n}$.
\end{remark}

\begin{proposition}\label{prop-comp-Hoelder-beta}
Let $\beta, \gamma \in (0,1]$ with $\beta \leq \gamma$ be arbitrary. Furthermore, let $Y \in \calc^{\beta}([0,T],W)$ and $f \in C_b^{\gamma,1}([0,T] \times W, \bar{W})$ be arbitrary. Then we have $f(Y) \in \calc^{\beta}([0,T],\bar{W})$, where $f(Y) : [0,T] \to \bar{W}$ denotes the path
\begin{align*}
f(Y)_t := f(t,Y_t), \quad t \in [0,T].
\end{align*}
\end{proposition}

\begin{proof}
Let $s,t \in [0,T]$ be arbitrary. Then we have
\begin{align*}
| f(t,Y_t) - f(s,Y_s) | &\leq | f(t,Y_t) - f(t,Y_s) | + | f(t,Y_s) - f(s,Y_s) |
\\ &\leq \| f(t,\cdot) \|_{C_b^1} | Y_t - Y_s | + \| f(\cdot,Y_s) \|_{\gamma} |t-s|^{\gamma}
\\ &\leq \| f \|_{C_b^{\gamma,1}} \| Y \|_{\beta} |t-s|^{\beta} + \| f \|_{C_b^{\gamma,1}} |t-s|^{\gamma}
\\ &\leq \| f \|_{C_b^{\gamma,1}} (\| Y \|_{\beta} + T^{\gamma-\beta}) |t-s|^{\beta},
\end{align*}
showing that $f(Y) \in \calc^{\beta}([0,T],\bar{W})$.
\end{proof}

Now, let $V$ be another Banach space, and let $X \in \calc^{\alpha}([0,T], V)$ be a H\"{o}lder continuous path for some index $\alpha \in (\frac{1}{3},\frac{1}{2}]$. For a controlled rough path $(Y,Y') \in \scrd_X^{2\alpha}([0,T],W)$ and a function $f \in C_b^{2\alpha,2}([0,T] \times W, \bar{W})$ we denote by $f(Y) : [0,T] \to \bar{W}$ the path
\begin{align}\label{comp-f-def-1}
f(Y)_t := f(t,Y_t), \quad t \in [0,T],
\end{align}
and, in view of the chain rule for differentiation, we denote by $f(Y)' : [0,T] \to L(V,\bar{W})$ the path
\begin{align}\label{comp-f-def-2}
f(Y)_t' := D_y f(t,Y_t) Y_t', \quad t \in [0,T].
\end{align}
The upcoming result shows that for a controlled rough path $(Y,Y')$ and a function $f$ of class $C_b^{2\alpha,2}$ the composition is also a controlled rough path.

\begin{proposition}\label{prop-comp-f}
Let $(Y,Y') \in \scrd_X^{2\alpha}([0,T],W)$ and $f \in C_b^{2\alpha,2}([0,T] \times W, \bar{W})$ be arbitrary. Then we have
\begin{align}\label{comp-f-in-D}
(f(Y),f(Y)') \in \scrd_X^{2\alpha}([0,T],\bar{W}),
\end{align}
where the paths $f(Y)$ and $f(Y)'$ are given by \eqref{comp-f-def-1} and \eqref{comp-f-def-2}. Furthermore, we have
\begin{align}\label{est-comp-Hoelder-1}
\| f(Y) \|_{\alpha} &\leq \| f \|_{C_b^{2\alpha,2}} ( \| Y \|_{\alpha} + T^{\alpha} ),
\\ \label{est-comp-Hoelder-2} \| f(Y)' \|_{\alpha} &\leq \| f \|_{C_b^{2\alpha,2}} \big( \| Y' \|_{\alpha} + \| Y \|_{\alpha} \| Y' \|_{\infty} + T^{\alpha} \| Y' \|_{\infty} \big),
\\ \label{est-comp-rem} \| R^{f(Y)} \|_{2\alpha} &\leq \| f \|_{C_b^{2\alpha,2}} \bigg( 1 + \frac{1}{2} \| Y \|_{\alpha}^2 + \| R^Y \|_{2\alpha} \bigg),
\end{align}
and we have the estimates
\begin{align}\label{comp-f-in-D-est}
| f(Y),f(Y)' |_{X,2\alpha} &\leq C \big( 1 + | Y,Y' |_{X,2\alpha}^2 \big),
\\ \label{comp-f-in-D-est-2} \interleave f(Y),f(Y)' \interleave_{X,2\alpha} &\leq C \big( 1 + \interleave Y,Y' \interleave_{X,2\alpha}^2 \big),
\end{align}
where the constant $C > 0$ depends on $\alpha$, $T$, $X$ and $\| f \|_{C_b^{2\alpha,2}}$. Moreover, for $T \leq 1$ the constant $C$ does not depend on $T$.
\end{proposition}

\begin{proof}
Let $s,t \in [0,T]$ be arbitrary. Then we have
\begin{align*}
| f(t,Y_t) - f(s,Y_s) | &\leq | f(t,Y_t) - f(t,Y_s) | + | f(t,Y_s) - f(s,Y_s) |
\\ &\leq \| f(t,\cdot) \|_{C_b^2} | Y_t - Y_s | + \| f(\cdot,Y_s) \|_{2 \alpha} |t-s|^{2\alpha}
\\ &\leq \| f \|_{C_b^{2\alpha,2}} \| Y \|_{\alpha} |t-s|^{\alpha} + \| f \|_{C_b^{2\alpha,2}} |t-s|^{2\alpha}
\\ &\leq \| f \|_{C_b^{2\alpha,2}} (\| Y \|_{\alpha} + T^{\alpha}) |t-s|^{\alpha},
\end{align*}
showing $f(Y) \in \calc^{\alpha}([0,T],\bar{W})$ and the estimate \eqref{est-comp-Hoelder-1}. Furthermore, we have
\begin{align*}
&|D_y f(t,Y_t) Y_t' - D_y f(s,Y_s) Y_s'|
\\ &\leq |D_y f(t,Y_t) Y_t' - D_y f(t,Y_t) Y_s'| + |D_y f(t,Y_t) Y_s' - D_y f(t,Y_s) Y_s'|
\\ &\quad + |D_y f(t,Y_s) Y_s' - D_y f(s,Y_s) Y_s'|
\\ &\leq | D_y f(t,Y_s) | \, |Y_t' - Y_s'| + |D_y f(t,Y_t) - D_y f(t,Y_s)| \, |Y_s'|
\\ &\quad + |D_y f(t,Y_s) - D_y f(s,Y_s)| \, |Y_s'|
\\ &\leq \| f(t,\cdot) \|_{C_b^2} \| Y' \|_{\alpha} |t-s|^{\alpha} + \| f(t,\cdot) \|_{C_b^2} |Y_t-Y_s| \, \|Y'\|_{\infty}
\\ &\quad + \| D_y f(\cdot,Y_s) \|_{2 \alpha} |t-s|^{2\alpha} \|Y'\|_{\infty}
\\ &\leq \| f \|_{C_b^{2\alpha,2}} \| Y' \|_{\alpha} |t-s|^{\alpha} + \| f \|_{C_b^{2\alpha,2}} \| Y \|_{\alpha} |t-s|^{\alpha} \|Y'\|_{\infty}
\\ &\quad + \| f \|_{C_b^{2\alpha,2}} \|Y'\|_{\infty} |t-s|^{2\alpha}
\\ &\leq \| f \|_{C_b^{2\alpha,2}} \big( \| Y' \|_{\alpha} + \| Y \|_{\alpha} \| Y' \|_{\infty} + T^{\alpha} \| Y' \|_{\infty} \big) |t-s|^{\alpha},
\end{align*}
showing $f(Y)' \in \calc^{\alpha}([0,T],L(V,\bar{W}))$ and the estimate \eqref{est-comp-Hoelder-2}. Moreover, by Taylor's theorem we obtain
\begin{align*}
|R_{s,t}^{f(Y)}| &= | f(Y)_{s,t} - f(Y)_s' X_{s,t} |
\\ &= | f(t,Y_t) - f(s,Y_s) - D_y f(s,Y_s) Y_s' X_{s,t} |
\\ &\leq | f(t,Y_t) - f(s,Y_t) | + | f(s,Y_t) - f(s,Y_s) - D_y f(s,Y_s) Y_{s,t} |
\\ &\quad + | D_y f(s,Y_s) (Y_{s,t} - Y_s' X_{s,t}) |
\\ &\leq \| f(\cdot,Y_t) \|_{2 \alpha} |t-s|^{2\alpha} + \frac{1}{2} \| f(s,\cdot) \|_{C_b^2} |Y_t-Y_s|^2 + \| f(s,\cdot) \|_{C_b^2} |R_{s,t}^Y|
\\ &\leq \| f \|_{C_b^{2\alpha,2}} |t-s|^{2 \alpha} + \frac{1}{2} \| f \|_{C_b^{2\alpha,2}} \| Y \|_{\alpha}^2 |t-s|^{2\alpha}
+ \| f \|_{C_b^{2\alpha,2}} \| R^Y \|_{2\alpha} |t-s|^{2\alpha}
\\ &= \| f \|_{C_b^{2\alpha,2}} \bigg( 1 + \frac{1}{2} \| Y \|_{\alpha}^2 + \| R^Y \|_{2\alpha} \bigg) |t-s|^{2\alpha},
\end{align*}
proving \eqref{comp-f-in-D} and \eqref{est-comp-rem}. Moreover, note that
\begin{align*}
|f(Y)_0| &= |f(0,Y_0)| \leq \| f \|_{C_b^{2\alpha,2}},
\\ |f(Y)_0'| &= |D_y f(0,Y_0) Y_0'| \leq |D_y f(0,Y_0)| \, |Y_0'| \leq \| f \|_{C_b^{2\alpha,2}} | Y,Y' |_{X,2\alpha}.
\end{align*}
Together with Lemma \ref{lemma-norm-of-Y} we obtain \eqref{comp-f-in-D-est} and \eqref{comp-f-in-D-est-2}.
\end{proof}

As an immediate consequence of Remark \ref{rem-f-time-dependent} and Proposition \ref{prop-comp-f}, we obtain the following well-known result (cf. \cite[Lemma 7.3]{Friz-Hairer}) for time-independent functions:

\begin{corollary}
For $X \in \calc^{\alpha}([0,T], V)$, $(Y,Y') \in \scrd_X^{2\alpha}([0,T],W)$ and $f \in C_b^2(W, \bar{W})$ we have
\begin{align*}
(f(Y),f(Y)') \in \scrd_X^{2\alpha}([0,T],\bar{W}).
\end{align*}
\end{corollary}

\subsection{Compositions of rough paths with bilinear operators}

In this section we consider compositions of controlled rough paths with bilinear operators. As in the previous section, let $V$ be a Banach space, and let $X \in \calc^{\alpha}([0,T], V)$ be a H\"{o}lder continuous path for some index $\alpha \in (\frac{1}{3},\frac{1}{2}]$.

First of all, let us consider compositions with linear operators. For a controlled rough path $(Y,Y') \in \scrd_X^{2\alpha}([0,T],W)$ and a continuous linear operator $\varphi \in L(W,\bar{W})$ we denote by $\varphi(Y) : [0,T] \to \bar{W}$ the path
\begin{align*}
\varphi(Y)_t := \varphi(Y_t), \quad t \in [0,T],
\end{align*}
and we denote by $\varphi(Y)' : [0,T] \to L(V,\bar{W})$ the path
\begin{align*}
\varphi(Y)_t' := \varphi(Y_t'), \quad t \in [0,T].
\end{align*}
The upcoming result shows that for a controlled rough path $(Y,Y')$ and a continuous linear operator $\varphi$ the composition is also a controlled rough path.

\begin{proposition}\label{prop-comp-linear-0}
Let $(Y,Y') \in \scrd_X^{2\alpha}([0,T],W)$ and $\varphi \in L(W,\bar{W})$ be arbitrary. Then we have
\begin{align}\label{comp-linear-0-1}
(\varphi(Y),\varphi(Y)') \in \scrd_X^{2\alpha}([0,T],\bar{W})
\end{align}
and the estimates
\begin{align}\label{comp-linear-0-2}
\| \varphi(Y), \varphi(Y)' \|_{X,2\alpha} &\leq | \varphi | \, \| Y,Y' \|_{X,2\alpha},
\\ \label{comp-linear-0-3} | \varphi(Y), \varphi(Y)' |_{X,2\alpha} &\leq | \varphi | \, | Y,Y' |_{X,2\alpha},
\\ \label{comp-linear-0-4} \interleave \varphi(Y), \varphi(Y)' \interleave_{X,2\alpha} &\leq | \varphi | \, \interleave Y,Y' \interleave_{X,2\alpha}.
\end{align}
\end{proposition}

\begin{proof}
Let $s,t \in [0,T]$ be arbitrary. Then we have
\begin{align*}
| \varphi(Y_t') - \varphi(Y_s') | = |\varphi(Y_t' - Y_s')| \leq | \varphi | \, |Y_t' - Y_s'| \leq | \varphi | \, \| Y' \|_{\alpha} |t-s|^{\alpha},
\end{align*}
showing $\| \varphi(Y') \|_{\alpha} \leq | \varphi | \, \| Y' \|_{\alpha}$. Furthermore, we have
\begin{align*}
R_{s,t}^{\varphi(Y)} = \varphi(Y_{s,t}) - \varphi(Y_s') X_{s,t} = \varphi(Y_{s,t} - Y_s' X_{s,t}) = \varphi(R_{s,t}^Y),
\end{align*}
and hence, we obtain
\begin{align*}
| R_{s,t}^{\varphi(Y)} | \leq | \varphi | \, |R_{s,t}^Y| \leq | \varphi | \, \| R^Y \|_{2\alpha} |t-s|^{2 \alpha}.
\end{align*}
This shows $\| R^{\varphi(Y)} \|_{2 \alpha} \leq | \varphi | \, \| R^Y \|_{2 \alpha}$, and we deduce \eqref{comp-linear-0-1}. Noting that $|\varphi(Y_0)| \leq | \varphi | \, |Y_0|$ and $|\varphi(Y_0')| \leq | \varphi | \, |Y_0'|$, the desired estimates \eqref{comp-linear-0-2}, \eqref{comp-linear-0-3} and \eqref{comp-linear-0-4} follow.
\end{proof}

Now, let $E$ be another Banach space. For controlled rough paths $(Y,Y') \in \scrd_X^{2\alpha}([0,T],W)$, $(Z,Z') \in \scrd_X^{2\alpha}([0,T],E)$ and a continuous bilinear operator $B \in L^{(2)}(W \times E,\bar{W})$ we denote by $B(Y,Z) : [0,T] \to \bar{W}$ the path
\begin{align*}
B(Y,Z)_t := B(Y_t,Z_t), \quad t \in [0,T],
\end{align*}
and, in view of the Leibniz rule for differentiation, we denote by $B(Y,Z)' : [0,T] \to L(V,\bar{W})$ the path
\begin{align*}
B(Y,Z)_t' := B(Y_t,Z_t') + B(Y_t',Z_t), \quad t \in [0,T].
\end{align*}
The upcoming result shows that for two controlled rough paths $(Y,Y')$, $(Z,Z')$ and a continuous bilinear operator $B$ the composition is also a controlled rough path.

\begin{proposition}\label{prop-bilinear}
Let $(Y,Y') \in \scrd_X^{2\alpha}([0,T],W)$, $(Z,Z') \in \scrd_X^{2\alpha}([0,T],E)$ and $B \in L^{(2)}(W \times E,\bar{W})$ be arbitrary. Then we have
\begin{align}\label{bilinear-controlled}
(B(Y,Z),B(Y,Z)') \in \scrd_X^{2\alpha}([0,T],\bar{W}),
\end{align}
and the estimate
\begin{align}\label{est-bilinear}
\interleave B(Y,Z),B(Y,Z)' \interleave_{X,2\alpha} \leq C \interleave Y,Y' \interleave_{X,2\alpha} \interleave Z,Z' \interleave_{X,2\alpha},
\end{align}
where the constant $C > 0$ depends on $\alpha$, $T$, $X$ and $| B |$. Moreover, for $T \leq 1$ the constant $C$ does not depend on $T$.
\end{proposition}

\begin{proof}
Let $s,t \in [0,T]$ be arbitrary. Then we have
\begin{align*}
| B(Y_t,Z_t) - B(Y_s,Z_s) | &\leq |B(Y_t,Z_t) - B(Y_t,Z_s)| + |B(Y_t,Z_s) - B(Y_s,Z_s)|
\\ &= |B(Y_t,Z_{s,t})| + |B(Y_{s,t},Z_s)|
\\ &\leq | B | \big( |Y_t| \, |Z_{s,t}| + |Y_{s,t}| \, |Z_s| \big)
\\ &\leq | B | \big( \| Y \|_{\infty} \| Z \|_{\alpha} + \| Y \|_{\alpha} \| Z \|_{\infty} \big) |t-s|^{\alpha},
\end{align*}
showing that $B(Y,Z) \in C^{\alpha}([0,T],\bar{W})$. Furthermore, we have
\begin{align*}
| B(Y_t,Z_t') - B(Y_s,Z_s') | &\leq | B(Y_t,Z_t') - B(Y_t,Z_s') | + | B(Y_t,Z_s') - B(Y_s,Z_s') |
\\ &= | B(Y_t,Z_{s,t}') | + | B(Y_{s,t},Z_s') |
\\ &\leq | B | \big( |Y_t| \, |Z_{s,t}'| + |Y_{s,t}| \, |Z_s'| \big)
\\ &\leq | B | \big( \| Y \|_{\infty} \| Z' \|_{\alpha} + \| Y \|_{\alpha} \| Z' \|_{\infty} \big) |t-s|^{\alpha}
\end{align*}
as well as
\begin{align*}
| B(Y_t',Z_t) - B(Y_s',Z_s) | &\leq | B(Y_t',Z_t) - B(Y_t',Z_s) | + | B(Y_t',Z_s) - B(Y_s',Z_s) |
\\ &= | B(Y_t',Z_{s,t}) | + | B(Y_{s,t}',Z_s) |
\\ &\leq | B | \big( |Y_t'| \, |Z_{s,t}| + |Y_{s,t}'| \, |Z_s| \big)
\\ &\leq | B | \big( \| Y' \|_{\infty} \| Z \|_{\alpha} + \| Y' \|_{\alpha} \| Z \|_{\infty} \big) |t-s|^{\alpha},
\end{align*}
showing that $B(Y,Z)' \in C^{\alpha}([0,T],L(V,\bar{W})$. Moreover, we have
\begin{align*}
R_{s,t}^{B(Y,Z)} &= B(Y_t,Z_t) - B(Y_s,Z_s) - B(Y_s,Z_s)' X_{s,t}
\\ &= B(Y_t,Z_{s,t}) + B(Y_{s,t},Z_s) - \big( B(Y_s,Z_s' X_{s,t}) + B(Y_s' X_{s,t}, Z_s) \big)
\\ &= B(Y_{s,t}-Y_s' X_{s,t},Z_s) + B(Y_s,Z_{s,t}- Z_s' X_{s,t}) + B(Y_{s,t},Z_{s,t})
\\ &= B(R_{s,t}^Y,Z_s) + B(Y_s,R_{s,t}^Z) + B(Y_{s,t},Z_{s,t}),
\end{align*}
and hence
\begin{align*}
| R_{s,t}^{B(Y,Z)} | &\leq | B | \big( |R_{s,t}^Y| \, |Z_s| + |Y_s| \, | R_{s,t}^Z | + |Y_{s,t}| \, |Z_{s,t}| \big)
\\ &\leq | B | \big( \| R^Y \|_{2 \alpha} \| Z \|_{\infty} + \| Y \|_{\infty} \| R^Y \|_{2 \alpha} + \| Y \|_{\alpha} \| Z \|_{\alpha} \big) |t-s|^{2 \alpha},
\end{align*}
which shows \eqref{bilinear-controlled}. Finally, noting that
\begin{align*}
|B(Y_0,Z_0)| &\leq | B | \, |Y_0| \, |Z_0|,
\\ |B(Y_0,Z_0')| &\leq | B | \, |Y_0| \, |Z_0'|,
\\ |B(Y_0',Z_0)| &\leq | B | \, |Y_0'| \, |Z_0|,
\end{align*}
applying Lemma \ref{lemma-norm-of-Y} proves \eqref{est-bilinear}.
\end{proof}

Before proceeding with Proposition \ref{prop-diff-f}, we prepare some auxiliary results.

\begin{lemma}\label{lemma-chain-rule}
Let $E,F$ be two Banach spaces. Furthermore, let $h : W \to L(W,F)$ be of class $C^1$, and let $B \in L(E,W)$ be arbitrary. Then the mapping
\begin{align*}
h_B : W \to L(E,F), \quad h_B(y) := h(y) B
\end{align*}
is of class $C^1$, and we have
\begin{align*}
D h_B(y)v = (D h(y)v) B \quad \text{for all $y,v \in W$.}
\end{align*}
\end{lemma}

\begin{proof}
We have $h_B = \ell \circ h$, where $\ell : L(W,F) \to L(E,F)$ is given by $\ell(z) = zB$. Note that $\ell$ is a linear operator. Moreover, we have
\begin{align*}
| \ell(z) | \leq | z | \, |B| \quad \text{for all $z \in L(W,F)$,}
\end{align*}
showing that $\ell$ is continuous. Therefore, the mapping $h_B$ is of class $C^1$, and by the chain rule we obtain
\begin{align*}
D h_B(y)v &= D (\ell \circ h)(y)v = D \ell(h(y)) \circ D h(y) v
\\ &= \ell \circ D h(y) v = (D h(y)v) B,
\end{align*}
completing the proof.
\end{proof}

\begin{lemma}\label{lemma-bilinear-fg-time}
Let $f \in C_b^{2 \alpha,3}([0,T] \times W,\bar{W})$ be arbitrary. Then there exists $g \in C_b^{2 \alpha,2}([0,T] \times (W \times W), L(W,\bar{W}))$ with $\| g \|_{C_b^{2 \alpha,2}} \leq \| f \|_{C_b^{2 \alpha,3}}$ such that
\begin{align}\label{identity-Taylor}
f(t,y_1) - f(t,y_2) = g(t,y) (y_1-y_2)
\end{align}
for all $t \in [0,T]$ and all $y = (y_1,y_2) \in W \times W$.
\end{lemma}

\begin{proof}
We define $g : [0,T] \times W \times W \to L(W,\bar{W})$ as
\begin{align*}
g(t,y) := \int_0^1 D_2 f(t,\theta y_1 + (1-\theta)y_2) \, d\theta.
\end{align*}
Then by Taylor's theorem identity \eqref{identity-Taylor} is satisfied for all $t \in [0,T]$ and all $y = (y_1,y_2) \in W \times W$. For $\theta \in [0,1]$ we consider the linear operator $B_{\theta} : W \times W \to W$ given by
\begin{align*}
B_{\theta}(y) := \theta y_1 + (1-\theta) y_2, \quad y = (y_1,y_2) \in W \times W.
\end{align*}
Then for all $y = (y_1,y_2) \in W \times W$ we have
\begin{align*}
|B_{\theta}(y)| = |\theta y_1 + (1-\theta)y_2| \leq \theta |y_1| + (1-\theta) |y_2| \leq |y_1| + |y_2| = |y|.
\end{align*}
Therefore, for each $\theta \in [0,1]$ we have $B_{\theta} \in L(W \times W,W)$ with
\begin{align}\label{B-norm}
|B_{\theta}| \leq 1 \quad \text{for all $\theta \in [0,1]$.}
\end{align}
Now, we fix an arbitrary $t \in [0,T]$. Then we have
\begin{align*}
g(t,y) = \int_0^1 D_2 f(t,B_{\theta}(y)) \, d\theta, \quad y \in W \times W.
\end{align*}
Therefore, we have $g(t,\cdot) \in C^2(W \times W,L(W,\bar{W}))$. Let $y \in W \times W$ be arbitrary. By the chain rule we have
\begin{align*}
D_y g(t,y) &= \int_0^1 D_y \big( D_2 f(t,B_{\theta}(y)) \big) \, d\theta = \int_0^1 D_2^2 f(t,B_{\theta}(y)) B_{\theta} \, d\theta,
\end{align*}
Moreover, by Lemma \ref{lemma-chain-rule} and the chain rule for all $v \in W \times W$ we obtain
\begin{align*}
D_y^2 g(t,y)v &= \int_0^1 D_y \big( D_2^2 f(t,B_{\theta}(y)) B_{\theta} \big) v \, d\theta
\\ &= \int_0^1 \big( D_2^3 f(t,B_{\theta}(y)) B_{\theta}(v) \big) B_{\theta} \, d\theta.
\end{align*}
Therefore, noting \eqref{B-norm} we have $g(t,\cdot) \in C_b^2(W \times W,L(W,\bar{W}))$ with
\begin{align}\label{g-fct-1}
\| D_y^k g(t,\cdot) \|_{\infty} \leq \| D_y^{k+1} f(t,\cdot) \|_{\infty}, \quad k=0,1,2.
\end{align}
Now, let $y \in W \times W$ be arbitrary. Furthermore, let $s,t \in [0,T]$ be arbitrary. Then we have
\begin{align*}
|g(t,y)-g(s,y)| &\leq \int_0^1 |D_2 f(t,B_{\theta}(y)) - D_2 f(s,B_{\theta}(y)) | \, d\theta
\\ &\leq \int_0^1 \| D_2 f(\cdot,B_{\theta}(y)) \|_{2\alpha} |t-s|^{2\alpha} \, d\theta.
\end{align*}
Moreover, by \eqref{B-norm} we obtain
\begin{align*}
|D_y g(t,y) - D_y g(s,y)| &\leq \int_0^1 |D_2^2 f(t,B_{\theta}(y)) B_{\theta} - D_2^2 f(s,B_{\theta}(y)) B_{\theta} | \, d\theta
\\ &\leq \int_0^1 |D_2^2 f(t,B_{\theta}(y)) - D_2^2 f(s,B_{\theta}(y)) | \, d\theta
\\ &\leq \int_0^1 \| D_2^2 f(\cdot,B_{\theta}(y)) \|_{2\alpha} |t-s|^{2\alpha} \, d\theta.
\end{align*}
Consequently, we have $g(\cdot,y), D_y g(\cdot,y) \in \calc^{2\alpha}([0,T],W)$ with
\begin{align}\label{g-fct-2}
\sup_{y \in W} \| D_y^k g(\cdot,y) \|_{2\alpha} \leq \sup_{y \in W} \| D_y^{k+1} f(\cdot,y) \|_{2\alpha}, \quad k=0,1.
\end{align}
From \eqref{g-fct-1} and \eqref{g-fct-2} we obtain $\| g \|_{C_b^{2\alpha,2}} \leq \| f \|_{C_b^{2 \alpha,3}} < \infty$, which completes the proof.
\end{proof}

The next result will be useful for our analysis of RDEs in the upcoming section.

\begin{proposition}\label{prop-diff-f}
Let $(Y,Y'), (Z,Z') \in \scrd_X^{2\alpha}([0,T],W)$ and $f \in C_b^{2\alpha,3}([0,T] \times W, \bar{W})$ be arbitrary. Then we have
\begin{align*}
\interleave (f(Y),f(Y)')-(f(Z),f(Z)') \interleave_{X,2\alpha} &\leq C \big( 1 + | Y,Y' |_{X,2\alpha}^2 + | Z,Z' |_{X,2\alpha}^2 \big)
\\ &\quad \interleave (Y,Y') - (Z,Z') \interleave_{X,2\alpha},
\end{align*}
where the constant $C > 0$ depends on $\alpha$, $T$, $X$ and $\| f \|_{C_b^{2 \alpha,3}}$. Moreover, for $T \leq 1$ the constant $C$ does not depend on $T$.
\end{proposition}

\begin{proof}
According to Lemma \ref{lemma-bilinear-fg-time} there exists a mapping $g \in C_b^{2 \alpha,2}([0,T] \times (W \times W), L(W,\bar{W}))$ with
\begin{align}\label{g-f-norms}
\| g \|_{C_b^{2 \alpha,2}} \leq \| f \|_{C_b^{2 \alpha,3}}
\end{align}
such that
\begin{align*}
f(t,y) - f(t,z) = g(t,y,z)(y-z), \quad t \in [0,T] \text{ and } y,z \in W.
\end{align*}
We define the bilinear operator $B : L(W,\bar{W}) \times W \to \bar{W}$ as
\begin{align*}
B(\Phi,y) := \Phi y.
\end{align*}
Then for all $(\Phi,y) \in L(W,\bar{W}) \times W$ we have
\begin{align*}
|B(\Phi,y)| \leq |\Phi| \, |y|.
\end{align*}
Therefore, we have $B \in L^{(2)}(L(W,\bar{W}) \times W, \bar{W})$ with $|B| \leq 1$. Hence, using Proposition \ref{prop-bilinear} we obtain
\begin{align*}
&\interleave (f(Y),f(Y)')-(f(Z),f(Z)') \interleave_{X,2\alpha}
\\ &= \interleave g(Y,Z)(Y-Z), (g(Y,Z)(Y-Z))' \interleave_{X,2\alpha}
\\ &= \interleave B(g(Y,Z),Y-Z), (B(g(Y,Z),Y-Z))' \interleave_{X,2\alpha}
\\ &\leq C \interleave g(Y,Z), g(Y,Z)' \interleave_{X,2\alpha} \interleave (Y,Y') - (Z,Z') \interleave_{X,2\alpha},
\end{align*}
where the constant $C > 0$ depends on $\alpha$, $T$ and $X$, and does not depend on $T \leq 1$. Furthermore, by Proposition \ref{prop-comp-f} and Lemma \ref{lemma-pairs} we have
\begin{align*}
\interleave g(Y,Z), g(Y,Z)' \interleave_{X,2\alpha} &= |g(0,Y_0,Z_0)| + |g(Y,Z), g(Y,Z)'|_{X,2\alpha}
\\ &\leq C ( 1 + | (Y,Z), (Y,Z)' |_{X,2\alpha}^2 )
\\ &\leq C ( 1 + |Y,Y'|_{X,2\alpha}^2 + |Z,Z'|_{X,2\alpha}^2 ),
\end{align*}
where the constant $C > 0$ depends on $\alpha$, $T$, $X$ and $\| g \|_{C_b^{2 \alpha,2}}$, and does not depend on $T \leq 1$. Taking into account \eqref{g-f-norms}, this completes the proof.
\end{proof}

\section{Rough differential equations}\label{sec-RDEs}

In this section we deal with rough differential equations. Let $V,W$ be Banach spaces, and fix a finite time horizon $T \in \bbr_+$. Consider the time-inhomogeneous rough differential equation (RDE)
\begin{align}\label{RDE}
\left\{
\begin{array}{rcl}
dY_t & = & f_0(t,Y_t) dt + f(t,Y_t) d \bfx_t
\\ Y_0 & = & \xi
\end{array}
\right.
\end{align}
with a rough path $\bfx = (X,\bbx) \in \scrc^{\alpha}([0,T],V)$ for some index $\alpha \in (\frac{1}{3},\frac{1}{2}]$, and appropriate mappings $f_0 : [0,T] \times W \to W$ and $f : [0,T] \times W \to L(V,W)$.

\subsection{Local and global solutions}

In this section we provide the definitions of local and global solutions to the RDE \eqref{RDE}. In what follows, we assume that $f_0$ is continuous, and that $f$ is of class $C_b^{2\alpha,2}$.

\begin{definition}\label{def-solution}
Let $\xi \in W$ be arbitrary. A path $(Y,Y') \in \scrd_X^{2 \alpha}([0,T_0],W)$ for some $T_0 \in (0,T]$ is called a \emph{local solution} to the RDE \eqref{RDE} with $Y_0 = \xi$ if $Y' = f(Y)$ and
\begin{align}\label{mild-solution}
Y_t = \xi + \int_0^t f_0(s,Y_s) ds + \int_0^t f(s,Y_s) d\mathbf{X}_s, \quad t \in [0,T_0].
\end{align}
If we can choose $T_0 = T$, then we also call $(Y,Y')$ a \emph{(global) solution} to the RDE \eqref{RDE} with $Y_0 = \xi$.
\end{definition}

\begin{remark}\label{rem-well-defined}
Note that the right-hand side of \eqref{mild-solution} is well-defined. Indeed, let $(Y,Y') \in \scrd_X^{2 \alpha}([0,T_0],W)$ for some $T_0 \in (0,T]$ be arbitrary. We define the paths $\Gamma(Y), \Psi(Y) : [0,T_0] \to W$ as
\begin{align}\label{map-part-2}
\Gamma(Y)_t &:= \int_0^t f_0(s,Y_s) \, ds,
\\ \label{map-part-3} \Psi(Y)_t &:= \int_0^t f(s,Y_s) \, d \bfx_s.
\end{align}
The path $f_0(Y) : [0,T_0] \to W$ defined according to \eqref{new-path-f0} is continuous. Therefore, by Proposition \ref{prop-reg-conv-Hoelder} it follows that
\begin{align*}
(\Gamma(Y),0) \in \scrd_X^{2\alpha}([0,T_0],W).
\end{align*}
Moreover, by Proposition \ref{prop-comp-f} with $\bar{W} = L(V,W)$ we have
\begin{align*}
(f(Y),f(Y)') \in \scrd_X^{2\alpha}([0,T_0],L(V,W)),
\end{align*}
and hence by Proposition \ref{prop-conv-rough} it follows that
\begin{align*}
(\Psi(Y),f(Y)) \in \scrd_X^{2\alpha}([0,T_0],W).
\end{align*}
Consequently, the right-hand side of \eqref{mild-solution}, which is given by
$$ \xi + \Gamma(Y) + \Psi(Y), $$
is an element of $\calc^{\alpha}([0,T_0],W)$.
\end{remark}

\begin{remark}\label{rem-der-solutions}
Let $\xi \in W$ be arbitrary, and let $(Y,Y') \in \scrd_X^{2 \alpha}([0,T_0],W)$ be a local solution to the RDE \eqref{RDE} with $Y_0 = \xi$ for some $T_0 \in (0,T]$. Then we have $Y' = f(Y)$, and hence by Proposition \ref{prop-comp-f} we obtain
\begin{align*}
f(Y)' = Df(Y)Y' = Df(Y)f(Y),
\end{align*}
where $Df(Y)f(Y) : [0,T_0] \to L(V,L(V,W))$ denotes the path
\begin{align*}
Df(Y)f(Y)_t := D_y f(t,Y_t) f(t,Y_t), \quad t \in [0,T_0].
\end{align*}
\end{remark}

\subsection{The space for the fixed point problem}

Note that equation \eqref{mild-solution} may be regarded as a fixed point problem. In this section we will analyze the space for this fixed point problem. We assume that $f_0 \in \Lip([0,T] \times W,W)$ and $f \in C_b^{2\alpha,3}([0,T] \times W, L(V,W))$, and we fix an initial condition $\xi \in W$. For every $t \in [0,T]$ we define the subset $\bbb_t \subset \scrd_X^{2\alpha}([0,t],W)$ as
\begin{align}\label{B-fixed-point}
\bbb_t := \{ (Y,Y') \in \scrd_X^{2\alpha}([0,t],W) : Y_0 = \xi, Y_0' = f(0,\xi), \| Y,Y' \|_{X,2\alpha;[0,t]} \leq 1 \}.
\end{align}

\begin{lemma}
For each $t \in [0,T]$ the set $\bbb_t$ endowed with the metric
\begin{align*}
d \big( (Y,Y'), (Z,Z') \big) := \| (Y,Y') - (Z,Z') \|_{X,2\alpha;[0,t]}, \quad (Y,Y'), (Z,Z') \in \bbb_t
\end{align*}
is a complete metric space.
\end{lemma}

\begin{proof}
The space $\scrd_X^{2\alpha}([0,t],W)$ endowed with the norm $\interleave \cdot \interleave_{X,2\alpha;[0,t]}$ is a Banach space, and hence a complete metric space. Furthermore, for all $(Y,Y'), (Z,Z') \in \bbb_t$ we have
\begin{align*}
d \big( (Y,Y'), (Z,Z') \big) &= \| Y-Z,Y'-Z' \|_{X,2\alpha;[0,t]}
\\ &= |Y_0-Z_0| + |Y_0' - Z_0'| +  \| Y-Z,Y'-Z' \|_{X,2\alpha;[0,t]}
\\ &= \interleave Y-Z,Y'-Z' \interleave_{X,2\alpha;[0,t]}.
\end{align*}
Moreover, the set $\bbb_t$ is a closed subset of $\scrd_X^{2\alpha}([0,t],W)$, completing the proof.
\end{proof}

For $t \in [0,T]$ we define the mapping $\Phi_t : \scrd_X^{2\alpha}([0,t],W) \to \scrd_X^{2\alpha}([0,t],W)$ as
\begin{align}\label{def-Phi}
\Phi_t(Y,Y') := (\xi,0) + (\Gamma(Y),0) + (\Psi(Y),f(Y)),
\end{align}
where the paths $\Gamma(Y), \Psi(Y) : [0,t] \to W$ are defined according to \eqref{map-part-2} and \eqref{map-part-3}. Note that the mapping $\Phi_t$ is well-defined due to Remark \ref{rem-well-defined}.

\begin{proposition}\label{prop-fixed-point-1}
For all $t \in [0,T]$ with $t \leq 1$ and all $(Y,Y') \in \bbb_t$ we have
\begin{align*}
\| \Phi_t(Y,Y') \|_{X,2\alpha;[0,t]} \leq C ( 1 + |\xi| ) t^{1-2\alpha} + C ( \interleave \mathbf{X} \interleave_{\alpha;[0,t]} + t^{\alpha} ),
\end{align*}
where the constant $C > 0$ depends on $\alpha$, $\mathbf{X}$, $\| f_0 \|_{\Lip}$ and $\| f \|_{C_b^{2\alpha,2}}$.
\end{proposition}

\begin{proof}
For convenience of notation, we will skip the subscript $[0,t]$ in the following calculations. Note that
\begin{align*}
\| \Phi_t(Y,Y') \|_{X,2\alpha} \leq \| \xi,0 \|_{X,2\alpha} + \| \Gamma(Y),0 \|_{X,2\alpha} + \| \Psi(Y),f(Y) \|_{X,2\alpha}.
\end{align*}
By Exercise \ref{exercise-reg-derivative-zero} we have
\begin{align*}
\| \xi,0 \|_{X,2\alpha} = \| \xi \|_{2 \alpha} = 0,
\end{align*}
because $\xi$ is a constant. Furthermore, by Proposition \ref{prop-reg-conv-Hoelder} we have
\begin{align*}
\| \Gamma(Y),0 \|_{X,2\alpha} &\leq \| f_0(Y) \|_{\infty} t^{1-2\alpha}.
\end{align*}
Noting that $Y_0 = \xi$, $Y_0' = f(0,\xi)$ and $\| Y,Y' \|_{X,2\alpha} \leq 1$, by Lemma \ref{lemma-lin-growth} and Lemma \ref{lemma-norm-of-Y} we obtain
\begin{align*}
\| f_0(Y) \|_{\infty} &\leq \| f_0 \|_{\Lip} ( 1 + \| Y \|_{\infty} )
\\ &\leq C \| f_0 \|_{\Lip} \big( 1 + \interleave Y,Y' \interleave_{X,2\alpha} ( \interleave \mathbf{X} \interleave_{\alpha} + 2 ) \big)
\\ &\leq C \| f_0 \|_{\Lip} \big( 1 + ( |\xi| + |f(0,\xi)| + 1) ( \interleave \mathbf{X} \interleave_{\alpha} + 2 ) \big)
\\ &\leq C \| f_0 \|_{\Lip} \big( 1 + ( |\xi| + \| f \|_{C_b^{2\alpha,2}} + 1) ( \interleave \mathbf{X} \interleave_{\alpha} + 2 ) \big),
\end{align*}
where the constant $C > 0$ depends on $\alpha$. Moreover, noting that $Y_0 = \xi$, $Y_0' = f(0,\xi)$ and $\| Y,Y' \|_{X,2\alpha} \leq 1$, by Proposition \ref{prop-conv-rough} and Proposition \ref{prop-comp-f} we have
\begin{align*}
\| \Psi(Y),f(Y) \|_{X,2\alpha} &\leq C \interleave f(Y), f(Y)' \interleave_{X,2\alpha} ( \interleave \mathbf{X} \interleave_{\alpha} + t^{\alpha} )
\\ &= C ( |f(0,\xi)| + | f(Y), f(Y)' |_{X,2\alpha} ) ( \interleave \mathbf{X} \interleave_{\alpha} + t^{\alpha} )
\\ &\lesssim C \big( 1 + | Y,Y' |_{X,2\alpha}^2 \big) ( \interleave \mathbf{X} \interleave_{\alpha} + t^{\alpha} )
\\ &\lesssim C \big( 1 + (|f(0,\xi)| + 1)^2 \big) ( \interleave \mathbf{X} \interleave_{\alpha} + t^{\alpha} )
\\ &\lesssim C ( \interleave \mathbf{X} \interleave_{\alpha} + t^{\alpha} ),
\end{align*}
where the constant $C > 0$, which changes from line to line, depends on $\alpha$, $\mathbf{X}$ and $\| f \|_{C_b^{2\alpha,2}}$. This completes the proof.
\end{proof}

\begin{proposition}\label{prop-fixed-point-2}
For all $t \in [0,T]$ with $t \leq 1$ and all $(Y,Y'), (Z,Z') \in \bbb_t$ we have
\begin{align*}
\| \Phi_t(Y,Y') - \Phi_t(Z,Z') \|_{X,2\alpha;[0,t]} &\leq C \| (Y,Y') - (Z,Z') \|_{X,2\alpha;[0,t]}
\\ &\qquad \times \big( \interleave \mathbf{X} \interleave_{\alpha;[0,t]} + t^{\alpha} + t^{1-2\alpha} \big),
\end{align*}
where the constant $C > 0$ depends on $\alpha$, $\mathbf{X}$, $\| f_0 \|_{\Lip}$ and $\| f \|_{C_b^{2\alpha,3}}$.
\end{proposition}

\begin{proof}
For convenience of notation, we will skip the subscript $[0,t]$ in the following calculations. Note that
\begin{align*}
\| \Phi(Y,Y') - \Phi(Z,Z') \|_{X,2\alpha} &\leq \| (\Gamma(Y),0) - (\Gamma(Z),0) \|_{X,2\alpha}
\\ &\quad + \| (\Psi(Y),f(Y)) - (\Psi(Z),f(Z)) \|_{X,2\alpha}.
\end{align*}
By Proposition \ref{prop-reg-conv-Hoelder} we have
\begin{align*}
\| (\Gamma(Y),0) - (\Gamma(Z),0) \|_{X,2\alpha} = \| \Gamma(Y) - \Gamma(Z), 0 \|_{X,2\alpha} \leq \| f_0(Y) - f_0(Z) \|_{\infty} t^{1-2\alpha}.
\end{align*}
Furthermore, noting that $Y_0 = Z_0$ and $Y_0' = Z_0'$, by Lemma \ref{lemma-lin-growth} and Lemma \ref{lemma-norm-of-Y} we obtain
\begin{align*}
\| f_0(Y) - f_0(Z) \|_{\infty} &\leq \| f_0 \|_{\Lip} \| Y-Z \|_{\infty}
\\ &\leq C \| f_0 \|_{\Lip} \| (Y,Y') - (Z,Z') \|_{X,2\alpha} ( \interleave \mathbf{X} \interleave_{\alpha} + 2),
\end{align*}
where the constant $C > 0$ depends on $\alpha$. Moreover, by Proposition \ref{prop-conv-rough} and Proposition \ref{prop-diff-f} we have
\begin{align*}
&\| (\Psi(Y),f(Y)) - (\Psi(Z),f(Z)) \|_{X,2\alpha}
\\ &\leq C \interleave (f(Y),f(Y)') - (f(Z),f(Z)') \interleave_{X,2\alpha} ( \interleave \mathbf{X} \interleave_{\alpha} + t^{\alpha} )
\\ &\lesssim C ( 1 + | Y,Y' |_{X,2\alpha}^2 + | Z,Z' |_{X,2\alpha}^2 )
\\ &\qquad \times \interleave (Y,Y') - (Z,Z') \interleave_{X,2\alpha} ( \interleave \mathbf{X} \interleave_{\alpha} + t^{\alpha} ),
\end{align*}
where the constant $C>0$, which changes from line to line, depends on $\alpha$, $\mathbf{X}$ and $\| f \|_{C_b^{2\alpha,3}}$. Now, noting that $Y_0 = Z_0 = \xi$, $Y_0' = Z_0' = f(0,\xi)$ and $$\| Y,Y' \|_{X,2\alpha}, \| Z,Z' \|_{X,2\alpha} \leq 1,$$ we obtain
\begin{align*}
| Y,Y' |_{X,2\alpha} &= |f(0,\xi)| + \| Y,Y' \|_{X,2\alpha} \leq \| f \|_{C_b^{2\alpha,3}} + 1,
\\ | Z,Z' |_{X,2\alpha} &= |f(0,\xi)| + \| Z,Z' \|_{X,2\alpha} \leq \| f \|_{C_b^{2\alpha,3}} + 1.
\end{align*}
Consequently, noting again that $Y_0 = Z_0$ and $Y_0' = Z_0'$ we have
\begin{align*}
\interleave (Y,Y') - (Z,Z') \interleave_{X,2\alpha} = \| (Y,Y') - (Z,Z') \|_{X,2\alpha},
\end{align*}
concluding the proof.
\end{proof}

\subsection{Auxiliary results}

In this section we derive auxiliary results which are required in order to establish existence and uniqueness of local solutions for the RDE \eqref{RDE}. Let $f_0 \in \Lip([0,T] \times W,W)$ and $f \in C_b^{2\alpha,2}([0,T] \times W, L(V,W))$ be arbitrary.

\begin{lemma}\label{lemma-Hoelder-order-2}
Let $\bfx = (X,\bbx) \in \scrc^{\alpha}([0,T],V)$ be a rough path for some index $\alpha \in (\frac{1}{3},\frac{1}{2}]$. Furthermore, let $\xi \in W$ be arbitrary, and let $(Y,Y') \in \scrd_X^{2 \alpha}([0,T_0],W)$ be a local solution to the RDE \eqref{RDE} with $Y_0 = \xi$ for some $T_0 \in (0,T]$. Then we have
\begin{align*}
\| f(Y)' \|_{\alpha} \leq C ( 1 + \| Y \|_{\alpha} ),
\end{align*}
where the constant $C > 0$ depends on $\alpha$, $T$ and $\| f \|_{C_b^{2\alpha,2}}$.
\end{lemma}

\begin{proof}
By Proposition \ref{prop-comp-f} we have
\begin{align*}
\| f(Y) \|_{\alpha} \leq \| f \|_{C_b^{2\alpha,2}} ( \| Y \|_{\alpha} + T^{\alpha} ).
\end{align*}
Therefore, noting that $Y' = f(Y)$, using Proposition \ref{prop-comp-f} again we obtain
\begin{align*}
\| f(Y)' \|_{\alpha} &\leq \| f \|_{C_b^{2\alpha,2}} \Big( \| f(Y) \|_{\alpha} + \| Y \|_{\alpha} \| f(Y) \|_{\infty} + T^{\alpha} \| f(Y) \|_{\infty} \Big)
\\ &\leq \| f \|_{C_b^{2\alpha,2}} \Big( \| f \|_{C_b^{2\alpha,2}} ( \| Y \|_{\alpha} + T^{\alpha} )
\\ &\qquad\qquad\qquad + \| Y \|_{\alpha} \| f \|_{C_b^{2\alpha,2}} + T^{\alpha} \| f \|_{C_b^{2\alpha,2}} \Big),
\end{align*}
completing the proof.
\end{proof}

In the following result, we will use the notation
\begin{align}\label{Hoelder-norm-in-I}
\| Z \|_{\alpha;I} := \sup_{\genfrac{}{}{0pt}{}{s,t \in I}{s \neq t}} \frac{|Z_{s,t}|}{|t-s|^{\alpha}}
\end{align}
for a subinterval $I \subset [0,T]$ and a path $Z \in \calc^{\alpha}([0,T],E)$ with values in some Banach space $E$. Accordingly, we will also use the notation
\begin{align}\label{Hoelder-norm-in-I-2}
\| \bbz \|_{\alpha;I} := \sup_{\genfrac{}{}{0pt}{}{s,t \in I}{s \neq t}} \frac{|\bbz_{s,t}|}{|t-s|^{\alpha}}
\end{align}
for a two-parameter path $\bbz \in \calc_2^{\alpha}([0,T]^2,E)$.

\begin{lemma}\label{lemma-remainder-mild-sol}
Let $\bfx = (X,\bbx) \in \scrc^{\alpha}([0,T],V)$ be a rough path for some index $\alpha \in (\frac{1}{3},\frac{1}{2}]$. Furthermore, let $\xi \in W$ be arbitrary, and let $(Y,Y') \in \scrd_X^{2 \alpha}([0,T_0],W)$ be a local solution to the RDE \eqref{RDE} with $Y_0 = \xi$ for some $T_0 \in (0,T]$. Let $s,t \in [0,T_0]$ with $s \leq t$ be arbitrary, and define the interval $I := [s,t]$. Then we have
\begin{align*}
| R_{s,t}^Y | &\leq C \Big( 1 + \| Y \|_{\alpha;I} + \| X \|_{\alpha;I} \| R^{f(Y)} \|_{2 \alpha;I}
\\ &\qquad\quad + \| \bbx \|_{2\alpha;I} ( 1 + \| Y \|_{\alpha;I} ) \Big) |t-s| + C | \bbx_{s,t} |,
\end{align*}
where the constant $C > 0$ depends on $\xi$, $\alpha$, $T$, $\| f_0 \|_{\Lip}$ and $\| f \|_{C_b^{2\alpha,2}}$.
\end{lemma}

\begin{proof}
For convenience of notation, we will skip the subscript $I$ in the following calculations. Then we have
\begin{align*}
Y_{s,t} &= \Gamma(Y)_{s,t} + \int_s^t f(r,Y_r) \, d\mathbf{X}_r,
\end{align*}
where the path $\Gamma(Y) : [0,T_0] \to W$ is defined according to \eqref{map-part-2}. Therefore, and since $Y' = f(Y)$, we have
\begin{align*}
|R_{s,t}^Y| &= | Y_{s,t} - Y_s' X_{s,t} | = | Y_{s,t} - f(Y)_s X_{s,t} |
\\ &\leq | \Gamma(Y)_{s,t} | + \bigg| \int_s^t f(Y)_r \, d \mathbf{X}_r - f(Y)_s X_{s,t} - D f(Y)f(Y)_s \bbx_{s,t} \bigg|
\\ &\quad + | D f(Y)f(Y)_s \bbx_{s,t} |.
\end{align*}
By Lemma \ref{lemma-reg-conv-Hoelder} and Lemma \ref{lemma-lin-growth} we have
\begin{align*}
| \Gamma(Y)_{s,t} | &\leq \| f_0(Y) \|_{\infty} |t-s|
\\ &\leq \| f_0 \|_{\Lip} (1 + \| Y \|_{\infty}) |t-s|
\\ &\leq \| f_0 \|_{\Lip} (1 + |Y_0| + \| Y \|_{\alpha} T^{\alpha}) |t-s|
\\ &\leq \| f_0 \|_{\Lip} (1 + |\xi| + \| Y \|_{\alpha} T^{\alpha}) |t-s|.
\end{align*}
Furthermore, by Theorem \ref{thm-Gubinelli} and Remark \ref{rem-der-solutions} we have
\begin{align*}
&\bigg| \int_s^t f(Y_r) \, d \mathbf{X}_r - f(Y_s) X_{s,t} - Df(Y_s)f(Y_s) \bbx_{s,t} \bigg|
\\ &\leq C \big( \| X \|_{\alpha} \| R^{f(Y)} \|_{2\alpha} + \| \bbx \|_{2\alpha} \| f(Y)' \|_{\alpha} \big) |t-s|^{3\alpha},
\end{align*}
where the constant $C > 0$ depends on $\alpha$. Moreover, we have
\begin{align*}
| Df(Y_s)f(Y_s) \bbx_{s,t} | \leq |Df(Y_s)| \, |f(Y_s)| \, | \bbx_{s,t} | \leq \| f \|_{C_b^{2\alpha,2}}^2 | \bbx_{s,t} |.
\end{align*}
Noting that $\alpha > \frac{1}{3}$, which implies $3 \alpha > 1$, using Lemma \ref{lemma-Hoelder-order-2} the proof is complete.
\end{proof}

\begin{lemma}\label{lemma-spaces-alpha-beta}
Let $\beta \in (\frac{1}{3},\frac{1}{2}]$ be arbitrary. We choose $\alpha \in (\frac{1}{3}, \beta)$ such that $\beta \leq \frac{3}{2} \alpha$, and let
\begin{align*}
\mathbf{X} = (X,\bbx) \in \scrc^{\beta}([0,T],V) \subset \scrc^{\alpha}([0,T],V)
\end{align*}
be a rough path. Furthermore, let $\xi \in W$, and let $(Y,Y') \in \scrd_X^{2 \alpha}([0,T_0],W)$ be a local solution to the RDE \eqref{RDE} with $Y_0 = \xi$ for some $T_0 \in (0,T]$. Then we even have $(Y,Y') \in \scrd_X^{2 \beta}([0,T_0],W)$.
\end{lemma}

\begin{proof}
Let $s,t \in [0,T_0]$ be arbitrary. Since $\beta \leq \frac{3}{2} \alpha < 2\alpha$, we have
\begin{align*}
|Y_{s,t}| &= |Y_s' X_{s,t} + R_{s,t}^Y| \leq \| Y' \|_{\infty} |X_{s,t}| + |R_{s,t}^Y|
\\ &\leq \| Y' \|_{\infty} \| X \|_{\beta} |t-s|^{\beta} + \| R^Y \|_{2\alpha} |t-s|^{2\alpha}
\\ &\leq \| Y' \|_{\infty} \| X \|_{\beta} |t-s|^{\beta} + \| R^Y \|_{2\alpha} T_0^{2\alpha - \beta} |t-s|^{\beta},
\end{align*}
showing that $Y \in \calc^{\beta}([0,T_0],W)$. Furthermore, from Proposition \ref{prop-comp-Hoelder-beta} we obtain
\begin{align*}
Y' = f(Y) \in \calc^{\beta}([0,T_0],L(V,W)).
\end{align*}
Moreover, $\beta \leq \frac{1}{2}$ implies $2 \beta \leq 1$, and hence by Lemma \ref{lemma-remainder-mild-sol} we have
\begin{align*}
| R_{s,t}^Y | &\leq C \Big( 1 + \| Y \|_{\alpha;I} + \| X \|_{\alpha;I} \| R^{f(Y)} \|_{2 \alpha;I}
\\ &\qquad\quad + \| \bbx \|_{2\alpha;I} ( 1 + \| Y \|_{\alpha;I} ) \Big) |t-s| + C | \bbx_{s,t} |
\\ &\leq C \Big( 1 + \| Y \|_{\alpha;I} + \| X \|_{\alpha;I} \| R^{f(Y)} \|_{2 \alpha;I}
\\ &\qquad\quad + \| \bbx \|_{2\alpha;I} ( 1 + \| Y \|_{\alpha;I} ) \Big) T^{1-2\beta} |t-s|^{2\beta} + C \| \bbx \|_{2 \beta} |t-s|^{2 \beta},
\end{align*}
where the constant $C > 0$ depends on $\xi$, $\alpha$, $T$, $\| f_0 \|$ and $\| f \|_{C_b^{2\alpha,2}}$. This shows $\| R^Y \|_{2\beta} < \infty$, completing the proof.
\end{proof}

\subsection{Existence and uniqueness of local solutions}

In this section we present a result concerning existence and uniqueness of local solutions to the RDE \eqref{RDE}.

\begin{theorem}\label{thm-RDGL-main-local}
Let $\bfx = (X,\bbx) \in \scrc^{\beta}([0,T],V)$ be a rough path for some index $\beta \in (\frac{1}{3},\frac{1}{2}]$, and let $f_0 \in \Lip([0,T] \times W,W)$ and $f \in C_b^{2\beta,3}([0,T] \times W, L(V,W))$ be mappings. Then for every $\xi \in W$ there exist $T_0 \in (0,T]$ and a unique local solution $(Y,Y') \in \scrd_X^{2 \beta}([0,T_0],W)$ to the RDE \eqref{RDE} with $Y_0 = \xi$.
\end{theorem}

\begin{proof}
We choose $\alpha \in (\frac{1}{3}, \beta)$ such that $\beta \leq \frac{3}{2} \alpha$. Since $\alpha < \beta$, we also have $\bfx \in \scrc^{\alpha}([0,T],V)$. For $t \in [0,T]$ we define the mapping $\Phi_t : \scrd_X^{2\alpha}([0,t],W) \to \scrd_X^{2\alpha}([0,t],W)$ according to \eqref{def-Phi}. By Propositions \ref{prop-fixed-point-1}, \ref{prop-fixed-point-2} and Lemma \ref{lemma-alpha-beta}, for every $t \in [0,T]$ with $t \leq 1$ and all $(Y,Y'), (Z,Z') \in \bbb_t$ we have
\begin{align*}
&\| \Phi_t(Y,Y') \|_{X,2\alpha;[0,t]} \leq C ( 1 + |\xi| ) t^{1-2\alpha}
\\ &\quad + C \big( \| X \|_{\beta,[0,t]} t^{\beta-\alpha} + \| \bbx \|_{2\beta,[0,t]} t^{2(\beta-\alpha)} + t^{\alpha} \big) \quad \text{and}
\\ &\| \Phi(Y,Y') - \Phi(Z,Z') \|_{X,2\alpha;[0,t]} \leq C \interleave (Y,Y') - (Z,Z') \interleave_{X,2\alpha,[0,t]}
\\ &\quad \times \big( \| X \|_{\beta,[0,t]} t^{\beta-\alpha} + \| \bbx \|_{2\beta,[0,t]} t^{2(\beta-\alpha)} + t^{\alpha} + t^{1-2\alpha} \big),
\end{align*}
where the constant $C > 0$ depends on $\alpha$, $\mathbf{X}$, $\| f_0 \|_{\Lip}$ and $\| f \|_{C_b^{2\alpha,3}}$. We choose $T_0 \in (0,T]$ with $T_0 \leq 1$ small enough such that
\begin{align*}
&C ( 1 + |\xi| ) T_0^{1-2\alpha} + C \big( \| X \|_{\beta,[0,T_0]} T_0^{\beta-\alpha} + \| \bbx \|_{2\beta,[0,T_0]} T_0^{2(\beta-\alpha)} + T_0^{\alpha} \big) \leq 1 \quad \text{and}
\\ &C \big( \| X \|_{\beta,[0,T_0]} T_0^{\beta-\alpha} + \| \bbx \|_{2\beta,[0,t]} T_0^{2(\beta-\alpha)} + T_0^{\alpha} + T_0^{1-2\alpha} \big) \leq \frac{1}{2}.
\end{align*}
Then $\Phi_{T_0} : \bbb_{T_0} \to \bbb_{T_0}$ is a contraction, and by the Banach fixed point theorem there is a unique solution $(Y,Y') \in \scrd_X^{2\alpha}([0,T_0],W)$ to the RDE \eqref{RDE} with $Y_0 = \xi$. Now, applying Lemma \ref{lemma-spaces-alpha-beta} completes the proof.
\end{proof}

\subsection{Further auxiliary results}

In this section we derive further auxiliary results which are required in order to establish existence and uniqueness of global solutions for the RDE \eqref{RDE}.

\begin{proposition}\label{prop-estimate}
Let $\bfx = (X,\bbx) \in \scrc^{\alpha}([0,T],V)$ be a rough path for some index $\alpha \in (\frac{1}{3},\frac{1}{2})$, let $f_0 \in \Lip([0,T] \times W,W)$ and $f \in C_b^{2\alpha,2}([0,T] \times W, L(V,W))$ be mappings, and let $\xi \in W$ be arbitrary. Then there exists a constant $K > 0$, depending on $\xi$, $\alpha$, $T$, $\mathbf{X}$, $\| f_0 \|_{\Lip}$ and $\| f \|_{C_b^{2\alpha,2}}$, such that for every solution $(Y,Y') \in \scrd_X^{2\alpha}([0,T],W)$ to the RDE \eqref{RDE} with $Y_0 = \xi$ we have
\begin{align*}
| Y_t | \leq K \quad \text{for all $t \in [0,T]$.}
\end{align*}
\end{proposition}

Before we proceed with the proof, let us agree on the following notation. Let us consider a one-parameter path $Z \in \calc^{\alpha}([0,T],E)$ or a two-parameter path $\bbz \in \calc_2^{\alpha}([0,T]^2,E)$ with values in some Banach space $E$. Then for any $h > 0$ the expressions $\| Z \|_{\alpha;h}$ and $\| \bbz \|_{\alpha;h}$ denote the supremum over all $\| Z \|_{\alpha;I}$ according to \eqref{Hoelder-norm-in-I}, and over all $\| \bbz \|_{\alpha;I}$ according to \eqref{Hoelder-norm-in-I-2}, where $I \subset [0,T]$ is any subinterval with $|I| \leq h$.

\begin{proof}[Proof of Proposition \ref{prop-estimate}]
Let $s,t \in [0,T]$ with $s \leq t$ be arbitrary, and define the interval $I := [s,t]$. By Lemma \ref{lemma-remainder-mild-sol} we have
\begin{align*}
| R_{s,t}^Y | &\leq C \Big( 1 + \| Y \|_{\alpha;I} + \| X \|_{\alpha;I} \| R^{f(Y)} \|_{2 \alpha;I}
\\ &\qquad\quad + \| \bbx \|_{2\alpha;I} ( 1 + \| Y \|_{\alpha;I} ) \Big) |t-s| + C \| \bbx \|_{2\alpha;I} |t-s|^{2\alpha},
\end{align*}
where the constant $C > 0$ depends on $\xi$, $\alpha$, $T$, $\| f_0 \|_{\Lip}$ and $\| f \|_{C_b^{2\alpha,2}}$. We set $\beta := 1-2\alpha$. Since $\alpha < \frac{1}{2}$, we have $\beta > 0$. Furthermore, let $h > 0$ be arbitrary. Then we have
\begin{equation}\label{R-1}
\begin{aligned}
\| R^Y \|_{2\alpha;h} &\leq C \| \bbx \|_{2\alpha;I} + C \Big( 1 + \| Y \|_{\alpha;h} + \| X \|_{\alpha;h} \| R^{f(Y)} \|_{2 \alpha;h}
\\ &\qquad\qquad\qquad\qquad + \| \bbx \|_{2\alpha;h} ( 1 + \| Y \|_{\alpha;h} ) \Big) h^{\beta}.
\end{aligned}
\end{equation}
By Proposition \ref{prop-comp-f} we have
\begin{align*}
\| R^{f(Y)} \|_{2\alpha;I} \leq \| f \|_{C_b^{2\alpha,2}} \bigg( 1 + \frac{1}{2} \| Y \|_{\alpha;I}^2 + \| R^Y \|_{2\alpha;I} \bigg),
\end{align*}
and hence
\begin{align}\label{R-2}
\| R^{f(Y)} \|_{2\alpha;h} \leq \| f \|_{C_b^{2\alpha,2}} \big( 1 + \| Y \|_{\alpha;h}^2 + \| R^Y \|_{2\alpha;h} \big).
\end{align}
Therefore, combining \eqref{R-1} and \eqref{R-2}, there is a constant $c_1 > 0$, only depending on $\alpha$, $T$ and $\| f \|_{C_b^{2\alpha,2}}$, such that
\begin{align*}
\| R^Y \|_{2\alpha;h} &\leq c_1 \| \bbx \|_{2\alpha;I}
\\ &\quad + c_1 \Big( 1 + \| Y \|_{\alpha;h} + \| X \|_{\alpha;h} ( 1 + \| Y \|_{\alpha;h}^2 + \| R^Y \|_{2\alpha;h} )
\\ &\qquad\qquad + \| \bbx \|_{2\alpha;h} ( 1 + \| Y \|_{\alpha;h} ) \Big) h^{\beta}.
\end{align*}
Thus, there is a constant $c_2 > 0$, only depending on $\alpha$, $T$ and $\| f \|_{C_b^{2\alpha,2}}$, such that for all $h \in (0,1]$ we obtain
\begin{equation}\label{R-3}
 \begin{aligned}
\| R^Y \|_{2\alpha;h} &\leq c_2 \big( \| X \|_{\alpha;h} + \| \bbx \|_{2 \alpha; h} + 1 \big) + c_2 \| X \|_{\alpha;h} h^{\beta} \| Y \|_{\alpha;h}^2
\\ &\quad + c_2 \| X \|_{\alpha;h} h^{\beta} \| R^Y \|_{2\alpha;h} + c_2 \big( \| \bbx \|_{2\alpha;h} + 1 \big) h^{\beta} \| Y \|_{\alpha;h}.
\end{aligned}
\end{equation}
Now we consider $h \in (0,1]$ so small such that
\begin{align}\label{choose-h}
c_2 \| X \|_{\alpha} h^{\beta} \leq \frac{1}{2} \quad \text{and} \quad c_2 \big( \| X \|_{\alpha} + \| \bbx \|_{2 \alpha} + 1 \big)^{1/2} h^{\beta} \leq \frac{1}{2}.
\end{align}
Then by \eqref{R-3} and \eqref{choose-h} we have
\begin{align*}
\| R^Y \|_{2\alpha;h} &\leq c_2 \big( \| X \|_{\alpha;h} + \| \bbx \|_{2 \alpha; h} + 1 \big) + \frac{1}{2} \| Y \|_{\alpha;h}^2 + \frac{1}{2} \| R^Y \|_{2\alpha;h}
\\ &\quad + \frac{1}{2} \big( \| \bbx \|_{2\alpha;h} + 1 \big)^{1/2} \| Y \|_{\alpha;h}.
\end{align*}
Therefore, using the elementary inequality $xy \leq \frac{x^2}{2} + \frac{y^2}{2}$, $x,y \in \bbr$ we obtain
\begin{equation}\label{R-4}
\begin{aligned}
\| R^Y \|_{2\alpha;h} &\leq 2 c_2 \big( \| X \|_{\alpha;h} + \| \bbx \|_{2 \alpha; h} + 1 \big) + \| Y \|_{\alpha;h}^2
\\&\quad + \big( \| \bbx \|_{2\alpha;h} + 1 \big)^{1/2} \| Y \|_{\alpha;h}
\\ &\leq c_3 \big( \| X \|_{\alpha;h} + \| \bbx \|_{2 \alpha; h} + 1 \big) + \frac{3}{2} \| Y \|_{\alpha;h}^2
\end{aligned}
\end{equation}
with a constant $c_3 > 0$, only depending on $\alpha$, $T$ and $\| f \|_{C_b^{2\alpha,2}}$. On the other hand, since $Y_{s,t} = f(Y_s)X_{s,t} + R_{s,t}^Y$, we have
\begin{align*}
|Y_t - Y_s| \leq |f(Y_s)X_{s,t}| + |R_{s,t}^Y| \leq \|f\|_{C_b^{2\alpha,2}} \| X \|_{\alpha} |t-s|^{\alpha} + \| R^Y \|_{2 \alpha} |t-s|^{2\alpha},
\end{align*}
and hence
\begin{align*}
\| Y \|_{\alpha;h} \leq C \big( \| X \|_{\alpha;h} + \| R^Y \|_{2\alpha;h} h^{\alpha} \big),
\end{align*}
where the constant $C > 0$ depends on $\| f \|_{C_b^{2\alpha,2}}$. Note that $\beta < \alpha$. Indeed, since $3\alpha > 1$, we have $\beta = 1-2\alpha < 3\alpha - 2\alpha = \alpha$. Therefore, we have
\begin{align*}
\| Y \|_{\alpha;h} \leq C \big( \| X \|_{\alpha;h} + \| R^Y \|_{2\alpha;h} h^{\beta} \big).
\end{align*}
Thus, using \eqref{R-4} and \eqref{choose-h} we obtain
\begin{align*}
\| Y \|_{\alpha;h} &\leq c_4 \| X \|_{\alpha;h} + c_4 \big( \| X \|_{\alpha;h} + \| \bbx \|_{2 \alpha; h} + 1 \big) h^{\beta} + c_4 \| Y \|_{\alpha;h}^2 h^{\beta}
\\ &\leq c_4 \| X \|_{\alpha;h} + c_5 \big( \| X \|_{\alpha;h} + \| \bbx \|_{2 \alpha; h} + 1 \big)^{1/2} + c_4 \| Y \|_{\alpha;h}^2 h^{\beta}
\end{align*}
with constants $c_4,c_5 > 0$, only depending on $\alpha$, $T$ and $\| f \|_{C_b^{2\alpha,2}}$. Hence, multiplying this inequality with $c_4 h^{\beta}$ we have
\begin{align*}
c_4 \| Y \|_{\alpha;h} h^{\beta} \leq c_6 ( \interleave \mathbf{X} \interleave_{\alpha} + 1 ) h^{\beta} + ( c_4 \| Y \|_{\alpha;h} h^{\beta} )^2,
\end{align*}
where $c_6 := c_4 + c_5$. Now, we set
\begin{align*}
\psi_h := c_4 \| Y \|_{\alpha;h} h^{\beta} \quad \text{and} \quad \lambda_h := c_6 (\interleave \mathbf{X} \interleave_{\alpha} + 1) h^{\beta}.
\end{align*}
Then we have
\begin{align*}
\psi_h \leq \lambda_h + \psi_h^2.
\end{align*}
Now, literally the same argumentation as in the proof of \cite[Prop. 8.2]{Friz-Hairer} shows the existence of a constant $N > 0$, depending on $\alpha$, $T$, $\mathbf{X}$, $\| f_0 \|_{\Lip}$ and $\| f \|_{C_b^{2\alpha,2}}$, such that $\| Y \|_{\alpha} \leq N$. Consequently, setting $K := |\xi| + N T^{\alpha}$, we obtain $\| Y \|_{\infty} \leq |Y_0| + \| Y \|_{\alpha} T^{\alpha} \leq K$, completing the proof.
\end{proof}

The following auxiliary result shows how two local solutions of the RDE \eqref{RDE} can be concatenated.

\begin{lemma}\label{lemma-flow-property}
Let $\bfx = (X,\bbx) \in \scrc^{\alpha}([0,T],V)$ be a rough path for some index $\alpha \in (\frac{1}{3},\frac{1}{2}]$, let $f_0 : [0,T] \times W \to W$ be continuous, and let $f : [0,T] \times W \to L(V,W)$ be of class $C_b^{2\alpha,2}$. Moreover, let $0 \leq q \leq r \leq s \leq T$ and $\xi \in W$ be arbitrary. Let $(Y(q,r,\xi),Y(q,r,\xi)') \in \scrd_X^{2\alpha}([q,r],W)$ be a solution to the equation $Y(q,r,\xi)' = f(Y(q,r,\xi))$ and
\begin{align}\label{flow-1}
Y(q,r,\xi)_t = \xi + \int_q^t f_0(u,Y(q,r,\xi)_u) \, du + \int_q^t f(u,Y(q,r,\xi)_u) \, d \mathbf{X}_u \quad t \in [q,r].
\end{align}
Furthermore, we set $\eta := Y(q,r,\xi)_r$ and let $(Y(r,s,\eta),Y(r,s,\eta)') \in \scrd_X^{2\alpha}([r,s],W)$ be a solution to the equation $Y(r,s,\eta)' = f(Y(r,s,\eta))$ and
\begin{align}\label{flow-2}
Y(r,s,\eta)_t = \eta + \int_r^t f_0(u,Y(r,s,\eta)_u) \, du + \int_r^t f(u,Y(r,s,\eta)_u) \, d \mathbf{X}_u, \quad t \in [r,s].
\end{align}
We define the concatenated path $Y(q,s,\xi) : [q,s] \to W$ as
\begin{align}\label{pasted-path-1}
Y(q,s,\xi)_t :=
\begin{cases}
Y(q,r,\xi)_t, & t \in [q,r],
\\ Y(r,s,\eta)_t, & t \in [r,s],
\end{cases}
\end{align}
and we define the concatenated path $Y(q,s,\xi)' : [q,s] \to L(V,W)$ as
\begin{align}\label{pasted-path-2}
Y(q,s,\xi)_t' :=
\begin{cases}
Y(q,r,\xi)_t', & t \in [q,r],
\\ Y(r,s,\eta)_t', & t \in [r,s].
\end{cases}
\end{align}
Then we have $(Y(q,s,\xi),Y(q,s,\xi)') \in \scrd_X^{2\alpha}([q,r],W)$, and this path is a solution to $Y(q,s,\xi)' = f(Y(q,s,\xi))$ and
\begin{align}\label{RPDE-flow}
Y(q,s,\xi)_t &= \xi + \int_q^t f_0(u,Y(q,s,\xi)_u) \, du + \int_q^t f(u,Y(q,s,\xi)_u) \, d\mathbf{X}_u, \quad t \in [q,s].
\end{align}
\end{lemma}

\begin{proof}
It is straightforward to check that $(Y(q,s,\xi),Y(q,s,\xi)') \in \scrd_X^{2\alpha}([q,r],W)$ and $Y(q,s,\xi)' = f(Y(q,s,\xi))$. Furthermore, taking into account \eqref{flow-1} it is evident that equation \eqref{RPDE-flow} is satisfied for all $t \in [q,r]$. Moreover, recalling that $\eta = Y(q,r,\xi)_r = Y(r,s,\eta)_r$, by \eqref{pasted-path-1}, \eqref{flow-2} and \eqref{flow-1} for each $t \in [r,s]$ we obtain
\begin{align*}
Y(q,s,\xi)_t &= Y(r,s,\eta)_t = Y(q,r,\xi)_r + \int_r^t f_0(u,Y(r,s,\eta)_u) du
\\ &\quad + \int_r^t f(u,Y(r,s,\eta)_u) d \mathbf{X}_u
\\ &= \xi + \int_q^r f_0(u,Y(q,r,\xi)_u) du + \int_q^r f(u,Y(q,r,\xi)_u) d \mathbf{X}_u
\\ &\quad + \int_r^t f_0(u,Y(r,s,\eta)_u) du + \int_r^t f(u,Y(r,s,\eta)_u) d \mathbf{X}_u
\\ &= \xi + \int_q^t f_0(u,Y(q,s,\xi)_u) du + \int_q^t f(u,Y(q,s,\xi)_u) d\mathbf{X}_u,
\end{align*}
where in the last step we have used \eqref{pasted-path-1} again.
\end{proof}

\subsection{Existence and uniqueness of global solutions}

In this section we present a result concerning existence and uniqueness of global solutions to the RDE \eqref{RDE}.

\begin{theorem}\label{thm-RDGL-main}
Let $\bfx = (X,\bbx) \in \scrc^{\beta}([0,T],V)$ be a rough path for some index $\beta \in (\frac{1}{3},\frac{1}{2}]$, and let $f_0 \in \Lip([0,T] \times W,W)$ and $f \in C_b^{2\beta,3}([0,T] \times W, L(V,W))$ be mappings. Then for every $\xi \in W$ there exists a unique solution $(Y,Y') \in \scrd_X^{2 \beta}([0,T],W)$ to the RDE \eqref{RDE} with $Y_0 = \xi$.
\end{theorem}

\begin{proof}
We choose $\alpha \in (\frac{1}{3}, \beta)$ such that $\beta \leq \frac{3}{2} \alpha$. Since $\alpha < \beta$, we have $\bfx \in \scrc^{\alpha}([0,T],V)$. By Proposition \ref{prop-estimate} for every solution $(Y,Y') \in \scrd_X^{2\alpha}([0,T],W)$ to the RDE \eqref{RDE} with $Y_0 = \xi$ we have
\begin{align}\label{estimate-solution}
| Y_t | \leq K \quad \text{for all $t \in [0,T]$,}
\end{align}
where the constant $K > 0$ depends on $\xi$, $\alpha$, $T$, $\mathbf{X}$, $\| f_0 \|_{\Lip}$ and $\| f \|_{C_b^{2\alpha,3}}$. For $t \in [0,T]$ we define the mapping $\Phi_t : \scrd_X^{2\alpha}([0,t],W) \to \scrd_X^{2\alpha}([0,t],W)$ according to \eqref{def-Phi}. By Propositions \ref{prop-fixed-point-1}, \ref{prop-fixed-point-2} and Lemma \ref{lemma-alpha-beta}, for every $t \in [0,T]$ with $t \leq 1$ and all $(Y,Y'), (Z,Z') \in \bbb_t$ we have
\begin{align*}
&\| \Phi_t(Y,Y') \|_{X,2\alpha;[0,t]} \leq C ( 1 + |\xi| ) t^{1-2\alpha}
\\ &\quad + C \big( \| X \|_{\beta,[0,t]} t^{\beta-\alpha} + \| \bbx \|_{2\beta,[0,t]} t^{2(\beta-\alpha)} + t^{\alpha} \big) \quad \text{and}
\\ &\| \Phi(Y,Y') - \Phi(Z,Z') \|_{X,2\alpha;[0,t]} \leq C \interleave (Y,Y') - (Z,Z') \interleave_{X,2\alpha;[0,t]}
\\ &\quad ( \| X \|_{\beta,[0,t]} t^{\beta-\alpha} + \| \bbx \|_{2\beta,[0,t]} t^{2(\beta-\alpha)} + t^{\alpha} + t^{1-2\alpha} ),
\end{align*}
where the constant $C > 0$ depends on $\alpha$, $\mathbf{X}$, $\| f_0 \|_{\Lip}$ and $\| f \|_{C_b^{2\alpha,3}}$. We choose $t \in (0,1]$ such that $T = nt$ for some $n \in \bbn$ as well as
\begin{align*}
&C ( 1 + K ) t^{1-2\alpha} + C \big( \| X \|_{\beta,[0,t]} t^{\beta-\alpha} + \| \bbx \|_{2\beta,[0,t]} t^{2(\beta-\alpha)} + t^{\alpha} \big) \leq 1 \quad \text{and}
\\ &C \big( \| X \|_{\beta,[0,t]} t^{\beta-\alpha} + \| \bbx \|_{2\beta,[0,t]} t^{2(\beta-\alpha)} + t^{\alpha} + t^{1-2\alpha} \big) \leq \frac{1}{2}.
\end{align*}
Then $\Phi_t : \bbb_t \to \bbb_t$ is a contraction, and by the Banach fixed point theorem there is a unique solution $(Y(1),Y(1)') \in \scrd_X^{2\alpha}([0,t],W)$ to the RDE \eqref{RDE} with $Y(1)_0 = \xi$. Taking into account \eqref{estimate-solution}, we can inductively apply this argument to the intervals $[(k-1)t,kt]$ for all $k=2,\ldots,n$ to obtain a unique solution $(Y(k),Y(k)') \in \scrd_X^{2\alpha}([(k-1)t,kt])$ to the RDE \eqref{RDE} with $Y(k)_{(k-1)t} = Y(k-1)_{(k-1)t}$. Using Lemma \ref{lemma-flow-property} we can concatenate these solutions and deduce the existence of a unique solution $(Y,Y') \in \scrd_X^{2\alpha}([0,T],W)$. By Lemma \ref{lemma-spaces-alpha-beta} we even have $(Y,Y') \in \scrd_X^{2\beta}([0,T])$, showing that $(Y,Y')$ is the unique solution to the RDE \eqref{RDE} with $Y_0 = \xi$, which concludes the proof.
\end{proof}

Under the conditions of Theorem \ref{thm-RDGL-main} we can consider the \emph{It\^{o}-Lyons map}
\begin{align}\label{IL-map}
(\xi,\mathbf{X}) \mapsto (Y,Y'),
\end{align}
which assigns the unique solution to the RDE \eqref{RDE} with $Y_0 = \xi$. In the following result we denote by $d_{\beta}$ the metric defined according to \eqref{metric-rough-paths}.

\begin{proposition}\label{prop-Ito-Lyons}
Let $\beta \in (\frac{1}{3},\frac{1}{2}]$ be an index, and let $f_0 \in \Lip([0,T] \times W,W)$ and $f \in C_b^{2\beta,3}([0,T] \times W, L(V,W))$ be mappings. Then the It\^{o}-Lyons map
\begin{align*}
W \times \big( \scrc^{\beta}([0,T],V), d_{\beta} \big) \to \calc^{\beta}([0,T],W) \times \calc^{\beta}([0,T],L(V,W))
\end{align*}
given by \eqref{IL-map} is locally Lipschitz continuous.
\end{proposition}

\begin{proof}
Let $\bfx = (X,\bbx) \in \scrc^{\beta}([0,T],V)$ and $\tilde{\bfx} = (\tilde{X},\tilde{\bbx}) \in \scrc^{\beta}([0,T],V)$ be two rough paths. For $(Y,Y') \in \scrd_X^{2\alpha}([0,T],W)$ and $(\tilde{Y},\tilde{Y}') \in \scrd_{\tilde{X}}^{2\alpha}([0,T],W)$ we set
\begin{align*}
d_{X,\tilde{X},2\beta} \big( (Y,Y'), (\tilde{Y}, \tilde{Y}') \big) := \| Y' - \tilde{Y}' \|_{\beta} + \| R^Y - R^{\tilde{Y}} \|_{2 \beta},
\end{align*}
which obviously gives us
\begin{align*}
\| Y' - \tilde{Y}' \|_{\beta} \leq d_{X,\tilde{X},2\beta} \big( (Y,Y'), (\tilde{Y}, \tilde{Y}') \big).
\end{align*}
Therefore, an analogous version of \cite[Thm. 8.5]{Friz-Hairer} for RDEs of the type \eqref{RDE} provides the desired continuity of the It\^{o}-Lyons map.
\end{proof}

Now we consider time-homogeneous RDEs of the form
\begin{align}\label{RDE-th}
\left\{
\begin{array}{rcl}
dY_t & = & f_0(Y_t) dt + f(Y_t) d \bfx_t
\\ Y_0 & = & \xi.
\end{array}
\right.
\end{align}
As an immediate consequence of Theorem \ref{thm-RDGL-main} we obtain the following result.

\begin{corollary}
Let $\bfx = (X,\bbx) \in \scrc^{\beta}([0,T],V)$ be a rough path for some index $\beta \in (\frac{1}{3},\frac{1}{2}]$, and let $f_0 : W \to W$ be Lipschitz continuous and $f \in C_b^{3}(W, L(V,W))$. Then for every $\xi \in W$ there exists a unique solution $(Y,Y') \in \scrd_X^{2 \beta}([0,T],W)$ to the time-homogeneous RDE \eqref{RDE-th} with $Y_0 = \xi$.
\end{corollary}

\section{Brownian motion as a rough path}\label{sec-BB-rough-path}

So far, rough path theory and the results about RDEs have been completely deterministic. Now, we will demonstrate that typical sample paths of a Brownian motion provide examples of rough paths. For this purpose, we will consider two choices of the second order processes, which give rise to the It\^{o}-enhanced Brownian motion and Stratonovich-enhanced Brownian motion. Throughout this section, let $(\Omega,\calf,(\calf_t)_{t \in \bbr_+},\bbp)$ be a filtered probability space satisfying the usual conditions.

\subsection{The Kolmogorov-Chentsov theorem for rough paths}

In this section we briefly recall the Kolmogorov-Chentsov theorem for rough paths. It provides conditions which ensure that a stochastic process satisfying Chen's relation admits a modification which is a rough path.

\begin{theorem}\label{thm-KC}\cite[Thm. 3.1]{Friz-Hairer}
Let $V$ be a Banach space, and let $\bfx = (X,\bbx)$ be a stochastic process $\bfx = (\bfx_t)_{t \in [0,T]}$ such that for each $\omega \in \Omega$ we have paths $X(\omega) : [0,T] \to V$ and $\bbx(\omega) : [0,T]^2 \to V \otimes V$ satisfying Chen's relation \eqref{Chen-relation}. Let $q \geq 2$ and $\beta > 0$ be constants such that $\beta - \frac{1}{q} > \frac{1}{3}$. We assume there is a constant $C > 0$ such that
\begin{align*}
| X_{s,t} |_{L^q} &\leq C |t-s|^{\beta} \quad \text{for all $s,t \in [0,T]$,}
\\ | \bbx_{s,t} |_{L^{q/2}} &\leq C |t-s|^{2 \beta} \quad \text{for all $s,t \in [0,T]$.}
\end{align*}
Then there exists a version of $\bfx$ which belongs to $\scrc^{\alpha}([0,T],V)$ for every $\alpha \in (\frac{1}{3}, \frac{1}{2}]$ with $\alpha < \beta - \frac{1}{q}$.
\end{theorem}

\subsection{It\^{o}-enhanced Brownian motion}\label{sec-BB-Ito-enhanced}

In this section we construct the It\^{o}-enhanced Brownian motion. Let $B$ be a Brownian motion with values in $\bbr^d$. Recall from Proposition \ref{prop-tensor-product-Rn} that by identification we have $\bbr^d \otimes \bbr^d = \bbr^{d \times d}$.

\begin{definition}
We introduce the following notions:
\begin{enumerate}
\item We call the $\bbr^{d \times d}$-valued process $\bbb = \bbb^{\text{\rm It\^o}}$ defined by the It\^{o} integrals
\begin{align}\label{Ito-enhancement}
\bbb_{s,t} := \int_s^t B_{s,r} \otimes dB_r, \quad s,t \in [0,T]
\end{align}
the \emph{It\^{o}-enhancement} of Brownian motion.

\item We call the process $\bfb = \bfb^{\text{\rm It\^o}}$ defined as $\bfb := (B,\bbb)$ the \emph{It\^{o}-enhanced Brownian motion}.
\end{enumerate}
\end{definition}

Now, our goal is to prove that a typical sample path of $\bfb^{\text{\rm It\^o}}$ is a rough path.

\begin{exercise}\label{exercise-BB-Ito}
Show that the It\^{o}-enhanced Brownian motion $\bfb = (B,\bbb)$ satisfies Chen's relation \eqref{Chen-relation}.
\end{exercise}

\begin{solution}
Since $B_0 = 0$, in view of Exercise \ref{exercise-Chen-sufficient} we only need to show that
\begin{align*}
\bbb_{s,t} = \bbb_{0,t} - \bbb_{0,s} - B_s \otimes B_{s,t} \quad \text{for all $s,t \in [0,T]$,}
\end{align*}
and this is established by the calculation
\begin{align*}
\bbb_{s,t} &= \int_s^t B_{s,r} \otimes dB_r = \int_s^t ( B_r - B_s ) \otimes dB_r
\\ &= \int_s^t B_r \otimes dB_r - B_s \otimes (B_t - B_s)
\\ &= \int_0^t B_r \otimes dB_r - \int_0^s B_r \otimes dB_r - B_s \otimes (B_t - B_s)
\\ &= \bbb_{0,t} - \bbb_{0,s} - B_s \otimes B_{s,t}.
\end{align*}
\end{solution}

We will require the following auxiliary result, which determines the representing matrix of the linear operator
\begin{align*}
\bbr^d \to \bbr^{m \times d}, \quad w \mapsto v \otimes w
\end{align*}
for any fixed $v \in \bbr^m$.

\begin{lemma}\label{lemma-matrix-for-tensor}
Let $A : \bbr^m \to \bbr^{(m \times d) \times d}$ be the linear operator given by
\begin{align*}
A_{ij,k}(v) =
\begin{cases}
v_i, & \text{if $j=k$,}
\\ 0, & \text{otherwise,}
\end{cases}
\quad \text{for all $v \in \bbr^m$,}
\end{align*}
for all $i=1,\ldots,m$ and $j,k = 1,\ldots,d$. Then we have
\begin{align*}
v \otimes w = A(v) \cdot w \quad \text{for all $v \in \bbr^m$ and $w \in \bbr^d$.}
\end{align*}
\end{lemma}

\begin{proof}
Let $v \in \bbr^m$ and $w \in \bbr^d$ be arbitrary. Then for all $i \in \{ 1,\ldots,m \}$ and $j \in \{ 1,\ldots,d \}$ we obtain
\begin{align*}
( A(v) \cdot w )_{ij} = \sum_{k=1}^d A_{ij,k}(v) \cdot w_k = A_{ij,j}(v) \cdot w_j = v_i \cdot w_j,
\end{align*}
completing the proof.
\end{proof}

In view of the following auxiliary result, note that for a vector $x \in \bbr^d$ we denote by
\begin{align*}
|x| := \bigg( \sum_{i=1}^d |x_i|^2 \bigg)^{1/2}
\end{align*}
the \emph{Euclidean norm} of $x$, and that for a matrix $A \in \bbr^{d \times m}$ we denote by
\begin{align*}
|A| := \bigg( \sum_{i=1}^d \sum_{j=1}^m |A_{ij}|^2 \bigg)^{1/2}
\end{align*}
the \emph{Frobenius norm} of $A$.

\begin{lemma}\label{lemma-Burkholder}
For each $p \in [2,\infty)$ there is a constant $C_p > 0$ such that for all $s,t \in [0,T]$ with $s < t$ the following statements are true:
\begin{enumerate}
\item For each $\bbr^{m \times d}$-valued predictable process $\Phi$ such that
\begin{align*}
\int_s^t | \Phi_r |^2 dr < \infty \quad \text{$\bbp$-almost surely}
\end{align*}
we have
\begin{align*}
\bbe \Bigg[ \sup_{u \in [s,t]} \bigg| \int_s^u \Phi_r dB_r \bigg|^p \Bigg] \leq C_p \cdot \bbe \Bigg[ \bigg( \int_s^t | \Phi_r |^2 dr \bigg)^{p/2} \Bigg].
\end{align*}
\item For each $\bbr^m$-valued predictable process $\Psi$ such that
\begin{align*}
\int_s^t | \Psi_r |^2 dr < \infty \quad \text{$\bbp$-almost surely}
\end{align*}
we have
\begin{align*}
\bbe \Bigg[ \sup_{u \in [s,t]} \bigg| \int_s^u \Psi_r \otimes dB_r \bigg|^p \Bigg] \leq C_p \cdot \bbe \Bigg[ \bigg( \int_s^t | \Psi_r |^2 dr \bigg)^{p/2} \Bigg].
\end{align*}
\end{enumerate}
\end{lemma}

\begin{proof}
Taking into account Lemma \ref{lemma-matrix-for-tensor}, both statements are a consequence of the multi-dimensional Burkholder-Davis-Gundy inequality; see, for example \cite[Thm. 1.1]{Marinelli-BDG}, where this inequality is even provided for processes with values in Hilbert spaces.
\end{proof}

Now, we are in the position to show that the It\^{o}-enhanced Brownian motion $\mathbf{B} = (B, \bbb)$ satisfies the conditions of the Kolmogorov-Chentsov theorem.

\begin{lemma}\label{lemma-Q-BM-1}
For each $p \geq 2$ there is a constant $C_p > 0$ such that
\begin{align*}
\bbe \big[ |B_{s,t}|^p \big] \leq C_p |t-s|^{p/2}, \quad s,t \in [0,T].
\end{align*}
\end{lemma}

\begin{proof}
Let $s,t \in [0,T]$ be arbitrary. Denoting by $\Id \in \bbr^{d \times d}$ the identity matrix, by Lemma \ref{lemma-Burkholder} we obtain
\begin{align*}
\bbe \big[ |B_{s,t}|^p \big] = \bbe \Bigg[ \bigg| \int_s^t \Id \, dB_r \bigg|^p \Bigg] \leq C_p \cdot \bbe \Bigg[ \bigg( \int_s^t |\Id|^2 dr \bigg)^{p/2} \Bigg] \leq C_p |t-s|^{p/2}
\end{align*}
with a constant $C_p > 0$.
\end{proof}

\begin{lemma}\label{lemma-Q-BM-2}
For each $p \geq 2$ there is a constant $C_p > 0$ such that
\begin{align*}
\bbe \big[ |\bbb_{s,t}|^p \big] \leq C_p |t-s|^p, \quad s,t \in [0,T].
\end{align*}
\end{lemma}

\begin{proof}
By Lemma \ref{lemma-Burkholder} and H\"{o}lder's inequality we obtain
\begin{align*}
\bbe \big[ |\bbb_{s,t}|^p \big] &= \bbe \Bigg[ \bigg| \int_s^t B_{s,r} \otimes dB_r \bigg|^p \Bigg] \leq c_p \cdot \bbe \Bigg[ \bigg( \int_s^t | B_{s,r} |^2 dr \bigg)^{p/2} \Bigg]
\\ &\leq c_p \cdot \bbe \Bigg[ \bigg| \int_s^t | B_{s,r} |^p dr \bigg| \Bigg] \cdot |t-s|^{\frac{p}{2} - 1}
\\ &\leq c_p \cdot \bbe \bigg[ \sup_{r \in [s,t]} |B_{s,r}|^p \bigg] \cdot |t-s|^{\frac{p}{2}} \leq c_p^2 |t-s|^p
\end{align*}
with a constant $c_p > 0$.
\end{proof}

After these preparations, we can show that typical sample paths of the It\^{o}-enhanced Brownian motion are rough paths.

\begin{proposition}\label{prop-BB-Ito-rough-path}
For each $\alpha \in (\frac{1}{3},\frac{1}{2})$ we have $\bbp$-almost surely
$$\mathbf{B}^{\text{\rm It\^o}} = ( B, \bbb^{\text{\rm It\^o}} ) \in \scrc^{\alpha}([0,T],\bbr^d).$$
\end{proposition}

\begin{proof}
Using Exercise \ref{exercise-BB-Ito}, it follows that the process $\mathbf{B} = \mathbf{B}^{\text{\rm It\^o}}$ satisfies Chen's relation \eqref{Chen-relation}. Now, let $\alpha \in (\frac{1}{3},\frac{1}{2})$ be arbitrary. We choose $\beta = \frac{1}{2}$, and let $q > 6$ be large enough such that $\beta - \frac{1}{q} > \alpha$. Obviously, then we also have $\beta - \frac{1}{q} > \frac{1}{3}$. Furthermore, by Lemmas \ref{lemma-Q-BM-1} and \ref{lemma-Q-BM-2} we obtain
\begin{align*}
| B_{s,t} |_{L^q} &= \bbe \big[ |B_{s,t}|^q \big]^{1/q} \leq C_q^{1/q} |t-s|^{1/2} = C_q^{1/q} |t-s|^{\beta}, \quad s,t \in [0,T],
\\ | \bbb_{s,t} |_{L^{q/2}} &= \bbe \big[ |\bbb_{s,t}|^{q/2} \big]^{2/q} \leq C_q^{2/q} |t-s| = C_q^{2/q} |t-s|^{2 \beta}, \quad s,t \in [0,T],
\end{align*}
with a constant $C_q > 0$. Consequently, applying the Kolmogorov Chentsov theorem (Theorem \ref{thm-KC}) completes the proof.
\end{proof}

Now, we determine the bracket $[ \mathbf{B} ]$ of the It\^{o}-enhanced Brownian motion. Recall from Exercise \ref{exercise-bracket} that it suffices to determine the values of the one-parameter bracket.

\begin{exercise}\label{exercise-bracket-Ito}
Show that $[ \mathbf{B} ]_t = t \cdot \Id$ for all $t \in [0,T]$.
\end{exercise}

\begin{solution}
Let $t \in [0,T]$ be arbitrary. Then we have
\begin{align*}
\Sym(\bbb_{0,t}) = \frac{1}{2} \big( \bbb_{0,t} + \bbb_{0,t}^{\top} \big),
\end{align*}
and hence, using integration by parts, for all $i,j \in \{ 1,\ldots,d \}$ we obtain
\begin{align*}
\Sym(\bbb_{0,t})_{ij} &= \frac{1}{2} \bigg( \int_0^t B_r^i d B_r^j + \int_0^t B_r^j dB_r^i \bigg)
\\ &= \frac{1}{2} \big( B_t^i B_t^j - [B^i,B^j]_t \big),
\end{align*}
showing that
\begin{align*}
\Sym(\bbb_{0,t}) = \frac{1}{2} \big( B_t \otimes B_t - t \cdot \Id \big) \quad \text{for all $t \in [0,T]$.}
\end{align*}
Therefore, we deduce that
\begin{align*}
[ \mathbf{B} ]_t = B_t \otimes B_t - 2 \, \Sym(\bbb_{0,t}) = t \cdot \Id.
\end{align*}
\end{solution}

As a consequence of Lemma \ref{lemma-geometric}, the It\^{o}-enhanced Brownian motion $\mathbf{B}^{\text{\rm It\^o}}$ is a rough path which is \emph{not} weakly geometric.

\subsection{Stratonovich-enhanced Brownian motion}\label{sec-BB-Strat-enhanced}

In this section we construct the Stratonovich-enhanced Brownian motion. As in the previous section, let $B$ be a Brownian motion with values in $\bbr^d$. In view of the following definition, recall that for two continuous semimartingales $X$ and $Y$ the \emph{Stratonovich integral} is defined as
\begin{align*}
\int_0^t X_s \circ dY_s := \int_0^t X_s dY_s + \frac{1}{2} [X,Y]_t, \quad t \in [0,T].
\end{align*}

\begin{definition}
We introduce the following notions:
\begin{enumerate}
\item We call the $\bbr^{d \times d}$-valued process $\bbb = \bbb^{\rm Strat}$ defined by the Stratonovich integrals
\begin{align*}
\bbb_{s,t} := \int_s^t B_{s,r} \otimes \circ dB_r, \quad s,t \in [0,T]
\end{align*}
the \emph{Stratonovich-enhancement} of Brownian motion.

\item We call the process $\bfb = \bfb^{\rm Strat}$ defined as $\bfb := (B,\bbb)$ the \emph{Stratonovich-enhanced Brownian motion}.
\end{enumerate}
\end{definition}

Now, our goal is to prove that a typical sample path of $\bfb^{\rm Strat}$ is a rough path. For this purpose, we would like to express $\bbb^{\rm Strat}$ by means of $\bbb^{\text{\rm It\^o}}$. Let us define $F : [0,T] \to \bbr^{d \times d}$ as
\begin{align*}
F_t^{ij} := \frac{1}{2} [ B^i, B^j ]_t, \quad t \in [0,T]
\end{align*}
for all $i,j = 1,\ldots,d$. Then we have
\begin{align*}
F_t = \frac{t}{2} \Id, \quad t \in [0,T].
\end{align*}

\begin{exercise}
Show that
\begin{align}\label{Strat-Ito-relation}
\bbb_{s,t}^{\rm Strat} = \bbb_{s,t}^{\text{\rm It\^o}} + F_{s,t} \quad \text{for all $s,t \in [0,T]$.}
\end{align}
\end{exercise}

\begin{solution}
Note that
\begin{align*}
\bbb_{s,t}^{\rm Strat} = \bbb_{s,t}^{\text{\rm It\^o}} + \frac{1}{2} G_{s,t} \quad \text{for all $s,t \in [0,T]$,}
\end{align*}
where $G : [0,T]^2 \to \bbr^{d \times d}$ is given by
\begin{align*}
G_{s,t}^{ij} = [ B_{s,\cdot}^i, B^j ]_t - [ B_{s,\cdot}^i, B^j ]_s, \quad s,t \in [0,T].
\end{align*}
Noting that $G_{s,t} = (t-s) \Id$ for all $s,t \in [0,T]$ provides the stated formula \eqref{Strat-Ito-relation}.
\end{solution}

\begin{exercise}\label{exercise-BB-Strat}
Show that the Stratonovich-enhanced Brownian motion $\bfb^{\rm Strat} = (B,\bbb^{\rm Strat})$ satisfies Chen's relation \eqref{Chen-relation}.
\end{exercise}

\begin{solution}
By Exercise \ref{exercise-BB-Ito} the It\^{o}-enhanced Brownian motion $\bfb^{\text{\rm It\^o}} = (B,\bbb^{\text{\rm It\^o}})$ satisfies Chen's relation \eqref{Chen-relation}. Taking into account \eqref{Strat-Ito-relation}, by Exercise \ref{exercise-F-second-order} we deduce that the Stratonovich-enhanced Brownian motion $\bfb^{\rm Strat} = (B,\bbb^{\rm Strat})$ satisfies Chen's relation \eqref{Chen-relation} as well.
\end{solution}

Now, we are ready to show that typical sample paths of the Stratonovich-enhanced Brownian motion are rough paths.

\begin{proposition}\label{prop-BB-Strat-rough-path}
For each $\alpha \in (\frac{1}{3},\frac{1}{2})$ we have $\bbp$-almost surely
$$\mathbf{B}^{\rm Strat} = ( B, \bbb^{\rm Strat} ) \in \scrc^{\alpha}([0,T],\bbr^d).$$
\end{proposition}

\begin{proof}
In view of \eqref{Strat-Ito-relation}, this is an immediate consequence of Proposition \ref{prop-BB-Ito-rough-path} and Corollary \ref{cor-F-second-order}.
\end{proof}

Actually, the Stratonovich-enhanced Brownian motion is even a weakly geometric rough path. In order to see this, let us compute the bracket.

\begin{exercise}\label{exercise-Strat-bracket}
Show that $[ \mathbf{B}^{\rm Strat} ]_t = 0$ for all $t \in [0,T]$.
\end{exercise}

\begin{solution}
Using \eqref{Strat-Ito-relation}, by Exercises \ref{exercise-bracket-F} and \ref{exercise-bracket-Ito} we have
\begin{align*}
[ \mathbf{B}^{\rm Strat} ]_t = [ \mathbf{B}^{\text{\rm It\^o}} ]_t - 2 \, \Sym(F_t) = t \cdot \Id - t \cdot \Id = 0, \quad t \in [0,T].
\end{align*}
\end{solution}

\begin{proposition}\label{prop-BB-Strat-rough-path-g}
For each $\alpha \in (\frac{1}{3},\frac{1}{2})$ we have $\bbp$-almost surely
$$\mathbf{B}^{\rm Strat} = ( B, \bbb^{\rm Strat} ) \in \scrc_g^{\alpha}([0,T],\bbr^d).$$
\end{proposition}

\begin{proof}
In view of Proposition \ref{prop-BB-Strat-rough-path} and Exercise \ref{exercise-Strat-bracket}, this is an immediate consequence of Lemma \ref{lemma-geometric}.
\end{proof}

\subsection{Coincidence of the It\^{o} integrals}

In this section we briefly recall a result showing that the rough It\^{o} integral and the stochastic It\^{o} integral coincide. Let $B$ be an $\bbr^d$-valued Brownian motion, and let $\mathbf{B}^{\text{\rm It\^o}} = (B,\bbb^{\text{\rm It\^o}})$ be the It\^{o}-enhanced Brownian motion, as introduced in Section \ref{sec-BB-Ito-enhanced}. Furthermore, let $\alpha \in (\frac{1}{3},\frac{1}{2})$ be an arbitrary index. According to Proposition \ref{prop-BB-Ito-rough-path} there is a $\bbp$-nullset $N_1$ such that $\mathbf{B}^{\text{\rm It\^o}} \in \scrc^{\alpha}([0,T],\bbr^d)$ on $N_1^c$. In view of the following result, recall the identifications from Remark \ref{rem-identifications-fin-dim}.

\begin{proposition}\label{prop-integrals-Ito-coincide}
Let $Y$ be a continuous $\bbr^{m \times d}$-valued process, and let $Y'$ be a continuous $\bbr^{m \times d \times d}$-valued processes such that $(Y,Y') \in \scrd_X^{2 \alpha}([0,T],\bbr^{m \times d})$ on $N_1^c$. Then the following statements are true:
\begin{enumerate}
\item The $\bbr^m$-valued rough integral
\begin{align*}
\int_0^T Y_s \, d \mathbf{B}^{\text{\rm It\^o}}_s = \lim_{|\Pi| \to 0} \sum_{[u,v] \in \Pi} ( Y_u B_{u,v} + Y_u' \bbb_{u,v}^{\text{\rm It\^o}} )
\end{align*}
exists on $N_1^c$.

\item If $Y$ and $Y'$ are adapted and bounded, then there is a $\bbp$-nullset $N$ with $N_1 \subset N$ such that
\begin{align*}
\int_0^T Y_s \, d \mathbf{B}^{\text{\rm It\^o}}_s = \int_0^T Y_s \, dB_s \quad \text{on $N^c$.}
\end{align*}
\end{enumerate}
\end{proposition}

\begin{proof}
This is a consequence of \cite[Prop. 5.1]{Friz-Hairer}.
\end{proof}

\subsection{Coincidence of the Stratonovich integrals}

Now, we briefly recall a result showing that the rough Stratonovich integral and the stochastic Stratonovich integral coincide. Let $B$ be an $\bbr^d$-valued Brownian motion, and let $\mathbf{B}^{\rm Strat} = (B,\bbb^{\rm Strat})$ be the Stratonovich-enhanced Brownian motion, as introduced in Section \ref{sec-BB-Strat-enhanced}. Furthermore, let $\alpha \in (\frac{1}{3},\frac{1}{2})$ be an arbitrary index. According to Proposition \ref{prop-BB-Strat-rough-path-g} there is a $\bbp$-nullset $N_1$ such that $\mathbf{B}^{\rm Strat} \in \scrc_g^{\alpha}([0,T],\bbr^d)$ on $N_1^c$. In view of the following result, recall the identifications from Remark \ref{rem-identifications-fin-dim}.

\begin{proposition}\label{prop-integrals-Stratonovich-coincide}
Let $Y$ be a continuous $\bbr^{m \times d}$-valued process, and let $Y'$ be a continuous $\bbr^{m \times d \times d}$-valued processes such that $(Y,Y') \in \scrd_X^{2 \alpha}([0,T],\bbr^{m \times d})$ on $N_1^c$. Then the following statements are true:
\begin{enumerate}
\item The $\bbr^m$-valued rough integral
\begin{align*}
\int_0^T Y_s \, d \mathbf{B}^{\rm Strat}_s = \lim_{|\Pi| \to 0} \sum_{[u,v] \in \Pi} ( Y_u B_{u,v} + Y_u' \bbb_{u,v}^{\rm Strat} )
\end{align*}
exists on $N_1^c$.

\item If $Y$ and $Y'$ are adapted and bounded, then there is a $\bbp$-nullset $N$ with $N_1 \subset N$ such that
\begin{align*}
\int_0^T Y_s \, d \mathbf{B}^{\rm Strat}_s = \int_0^T Y_s \circ dB_s \quad \text{on $N^c$.}
\end{align*}
\end{enumerate}
\end{proposition}

\begin{proof}
This is a consequence of \cite[Cor. 5.2]{Friz-Hairer}.
\end{proof}

\section{Stochastic differential equations}\label{sec-SDEs}

In this section we will apply our previous findings in order to establish existence and uniqueness results for stochastic differential equations (SDEs). In Section \ref{sec-Ito-equations} we will treat It\^{o} SDEs, and in Section \ref{sec-Strat-equations} we will treat Stratonovich SDEs. Throughout this section, let $(\Omega,\calf,(\calf_t)_{t \in \bbr_+},\bbp)$ be a filtered probability space satisfying the usual conditions, and let $B$ be a Brownian motion with values in $\bbr^d$.

\subsection{It\^{o} equations}\label{sec-Ito-equations}

In this section we establish an existence and uniqueness result for It\^{o} SDEs. Let $\bfb^{\text{\rm It\^o}} = (B,\bbb^{\text{\rm It\^o}})$ be the It\^{o}-enhanced Brownian motion, as introduced in Section \ref{sec-BB-Ito-enhanced}. We consider the $\bbr^m$-valued random It\^{o} RDE
\begin{align}\label{RDE-Ito-random}
\left\{
\begin{array}{rcl}
dY_t & = & f_0(t,Y_t) dt + f(t,Y_t) d \mathbf{B}_t^{\text{\rm It\^o}}
\\ Y_0 & = & \xi
\end{array}
\right.
\end{align}
as well as the $\bbr^m$-valued It\^{o} SDE
\begin{align}\label{SDE-Ito}
\left\{
\begin{array}{rcl}
dY_t & = & f_0(t,Y_t) dt + f(t,Y_t) d B_t
\\ Y_0 & = & \xi
\end{array}
\right.
\end{align}
with mappings $f_0 : [0,T] \times \bbr^m \to \bbr^m$ and $f : [0,T] \times \bbr^m \to \bbr^{m \times d}$.

\begin{definition}
Let $\xi \in \bbr^m$ be arbitrary. An $\bbr^m$-valued continuous, adapted process $Y = (Y_t)_{t \in [0,T_0]}$ for some $T_0 \in (0,T]$ is called a \emph{local solution} to the It\^{o} SDE \eqref{SDE-Ito} with $Y_0 = \xi$ if
\begin{align*}
Y_t = \xi + \int_0^t f_0(s,Y_s) ds + \int_0^t f(s,Y_s) dB_s, \quad t \in [0,T_0].
\end{align*}
If we can choose $T_0 = T$, then we also call $Y$ a \emph{(global) solution} to the It\^{o} SDE \eqref{SDE-Ito} with $Y_0 = \xi$.
\end{definition}

Let $\beta \in (\frac{1}{3},\frac{1}{2})$ be an index. By Proposition \ref{prop-BB-Ito-rough-path} there is a $\bbp$-nullset $N$ such that $\mathbf{B}^{\text{\rm It\^o}} \in \scrc^{\beta}([0,T],\bbr^d)$ on $N^c$. In the upcoming result, local and global solutions to the random RDE \eqref{RDE-Ito-random} are defined according to Definition \ref{def-solution} for each $\omega \in N^c$.

\begin{theorem}\label{thm-Ito-SDEs}
Suppose that $f_0 \in \Lip([0,T] \times \bbr^m,\bbr^m)$ and $f \in C_b^{2\beta,3}([0,T] \times \bbr^m, \bbr^{m \times d})$. Then for every $\xi \in \bbr^m$ the following statements are true:
\begin{enumerate}
\item There exists a unique solution $(Y,Y') \in \scrd_X^{2 \beta}([0,T],W)$ to the random It\^{o} RDE \eqref{RDE-Ito-random} with $Y_0 = \xi$ on $N^c$.

\item There is a $\bbp$-nullset $N_1$ with $N \subset N_1$ such that the associated stochastic process $Y$ restricted to $N_1^c$ is also a solution to the It\^{o} SDE \eqref{SDE-Ito}.
\end{enumerate}
\end{theorem}

\begin{proof}
The first statement is an immediate consequence of Theorem \ref{thm-RDGL-main}. For the proof of the second statement, note that $\calf_t^B = \calf_t^{\mathbf{B}^{\text{\rm It\^o}}} \subset \calf_t$ for all $t \in [0,T]$. Furthermore, we have
\begin{align*}
(Y,Y') = \Phi(\xi,\mathbf{B}^{\text{\rm It\^o}}) \quad \text{on $N^c$,}
\end{align*}
where $\Phi$ denotes the It\^{o}-Lyons map defined according to \eqref{IL-map}. By the continuity of the It\^{o}-Lyons maps (see Proposition \ref{prop-Ito-Lyons}) it follows that the stochastic process $(Y,Y')$ is adapted. Therefore, applying Proposition \ref{prop-integrals-Ito-coincide} concludes the proof.
\end{proof}

\subsection{Stratonovich equations}\label{sec-Strat-equations}

Now, we establish an existence and uniqueness result for Stratonovich SDEs. Let $B$ be an $\bbr^d$-valued Brownian motion, and let $\mathbf{B}^{\rm Strat} = (B,\bbb^{\rm Strat})$ be the Stratonovich-enhanced Brownian motion, as introduced in Section \ref{sec-BB-Strat-enhanced}. We consider the $\bbr^m$-valued random Stratonovich RDE
\begin{align}\label{RDE-Strat-random}
\left\{
\begin{array}{rcl}
dY_t & = & f_0(t,Y_t) dt + f(t,Y_t) d \mathbf{B}_t^{\rm Strat}
\\ Y_0 & = & \xi
\end{array}
\right.
\end{align}
as well as the $\bbr^m$-valued Stratonovich SDE
\begin{align}\label{SDE-Strat}
\left\{
\begin{array}{rcl}
dY_t & = & f_0(t,Y_t) dt + f(t,Y_t) \circ d B_t
\\ Y_0 & = & \xi
\end{array}
\right.
\end{align}
with mappings $f_0 : [0,T] \times \bbr^m \to \bbr^m$ and $f : [0,T] \times \bbr^m \to \bbr^{m \times d}$.

\begin{definition}
Let $\xi \in \bbr^m$ be arbitrary. An $\bbr^m$-valued continuous, adapted process $Y = (Y_t)_{t \in [0,T_0]}$ for some $T_0 \in (0,T]$ is called a \emph{local solution} to the Stratonovich SDE \eqref{SDE-Strat} with $Y_0 = \xi$ if
\begin{align*}
Y_t = \xi + \int_0^t f_0(s,Y_s) ds + \int_0^t f(s,Y_s) \circ dB_s, \quad t \in [0,T_0].
\end{align*}
If we can choose $T_0 = T$, then we also call $Y$ a \emph{(global) solution} to the Stratonovich SDE \eqref{SDE-Strat} with $Y_0 = \xi$.
\end{definition}

Let $\beta \in (\frac{1}{3},\frac{1}{2})$ be an index. By Proposition \ref{prop-BB-Strat-rough-path-g} there is a $\bbp$-nullset $N$ such that $\mathbf{B}^{\rm Strat} \in \scrc_g^{\beta}([0,T],\bbr^d)$ on $N^c$. In the upcoming result, local and global solutions to the random RDE \eqref{RDE-Strat-random} are defined according to Definition \ref{def-solution} for each $\omega \in N^c$. Its proof is analogous to that of Theorem \ref{thm-Ito-SDEs}, and therefore omitted. The essential difference is that at the end of the proof we use Proposition \ref{prop-integrals-Stratonovich-coincide}.

\begin{theorem}\label{thm-Strat-SDEs}
Suppose that $f_0 \in \Lip([0,T] \times \bbr^m,\bbr^m)$ and $f \in C_b^{2\beta,3}([0,T] \times \bbr^m, \bbr^{m \times d})$. Then for every $\xi \in \bbr^m$ the following statements are true:
\begin{enumerate}
\item There exists a unique solution $(Y,Y') \in \scrd_X^{2 \beta}([0,T],W)$ to the random Stratonovich RDE \eqref{RDE-Strat-random} with $Y_0 = \xi$ on $N^c$.

\item There is a $\bbp$-nullset $N_1$ with $N \subset N_1$ such that the associated stochastic process $Y$ restricted to $N_1^c$ is also a solution to the Stratonovich SDE \eqref{SDE-Strat}.
\end{enumerate}
\end{theorem}

\section{Rough partial differential equations}\label{sec-RPDEs}

In this section we describe how our previous results about RDEs can be extended to semilinear rough partial differential equations (RPDEs), and which new challenges this extension entails.

As in Section \ref{sec-RDEs}, let $V,W$ be Banach spaces, let $T \in \bbr_+$ be a finite time horizon, and let $\bfx = (X,\bbx) \in \scrc^{\alpha}([0,T],V)$ be a rough path for some index $\alpha \in (\frac{1}{3},\frac{1}{2}]$. We consider time-inhomogeneous RPDEs of the form
\begin{align}\label{RPDE}
\left\{
\begin{array}{rcl}
dY_t & = & (A Y_t + f_0(t,Y_t)) dt + f(t,Y_t) d \bfx_t
\\ Y_0 & = & \xi
\end{array}
\right.
\end{align}
with appropriate mappings $f_0 : [0,T] \times W \to W$ and $f : [0,T] \times W \to L(V,W)$. Moreover, in contrast to the RDE \eqref{RDE}, we have an additional linear operator $A$, which is assumed to be the generator of a strongly continuous semigroup.

\subsection{Strongly continuous semigroups}\label{sec-semigroups}

At this juncture, let us briefly provide the required background about strongly continuous semigroups. Further details can be found, for example, in \cite{Engel-Nagel} or \cite{Pazy}. Recall that $W$ denotes a Banach space. A family $(S_t)_{t \geq 0}$ of continuous linear operators $S_t \in L(W)$ is called a \emph{strongly continuous semigroup} (or \emph{$C_0$-semigroup}) on $W$ if the following conditions are fulfilled:
\begin{enumerate}
\item $S_0 = \Id$.

\item $S_{s+t} = S_s S_t$ for all $s,t \geq 0$.

\item $\lim_{t \to 0} S_t y = y$ for all $y \in W$.
\end{enumerate}
Now, let $(S_t)_{t \geq 0}$ be a $C_0$-semigroup on $W$. It is well-known (see, e.g. \cite[Thm. 1.2.2]{Pazy}) that there are constants $M \geq 1$ and $\omega \in \bbr$ such that
\begin{align}\label{est-semigroup}
|S_t| \leq M e^{\omega t} \quad \text{for all $t \geq 0$.}
\end{align}
If we can choose $M=1$ and $\omega = 0$ in \eqref{est-semigroup}, that is $|S_t| \leq 1$ for all $t \geq 0$, then we call the $C_0$-semigroup $(S_t)_{t \geq 0}$ a \emph{semigroup of contractions}. The \emph{infinitesimal generator} $A : W \supset D(A) \to W$ is the  operator
\begin{align*}
Ay := \lim_{h \to 0} \frac{S_h y - y}{h}
\end{align*}
defined on the \emph{domain}
\begin{align*}
D(A) := \bigg\{ y \in W : \lim_{h \to 0} \frac{S_h y - y}{h} \text{ exists} \bigg\}.
\end{align*}
So far, this is all we need to know about $C_0$-semigroups, but further results will be presented later in Section \ref{subsec-semigroups}.

\subsection{Mild solutions}

Now we can introduce the concept of a mild solution. Let $A$ be the generator of a $C_0$-semigroup $(S_t)_{t \geq 0}$ on $W$.

\begin{definition}\label{def-mild-solution}
Let $\xi \in W$ be arbitrary. A path $(Y,Y') \in \scrd_X^{2 \alpha}([0,T_0],W)$ for some $T_0 \in (0,T]$ is called a \emph{local mild solution} to the RPDE \eqref{RPDE} with $Y_0 = \xi$ if $Y' = f(Y)$ and
\begin{align}\label{mild}
Y_t = S_t \xi + \int_0^t S_{t-s} f_0(s,Y_s) ds + \int_0^t S_{t-s} f(s,Y_s) d\mathbf{X}_s, \quad t \in [0,T_0].
\end{align}
If we can choose $T_0 = T$, then we also call $(Y,Y')$ a \emph{(global) mild solution} to the RPDE \eqref{RPDE} with $Y_0 = \xi$.
\end{definition}

The equation \eqref{mild} is also called the \emph{variation of constants formula}. The principle challenge when dealing with mild solutions is an appropriate definition of the rough convolution
\begin{align}\label{conv-intro}
\int_0^t S_{t-s} Y_s \, d \mathbf{X}_s, \quad t \in [0,T]
\end{align}
for controlled rough paths $(Y,Y')$. In order to demonstrate this point, let us fix a controlled rough path
\begin{align}\label{controlled-in-W}
(Y,Y') \in \scrd_X^{2 \alpha}([0,T],L(V,W)).
\end{align}
Then a self-evident idea to define the rough convolution \eqref{conv-intro} is as follows. We fix an arbitrary $t \in [0,T]$ and define the paths $\Upsilon = \Upsilon(t) : [0,t] \to L(V,W)$ and $\Upsilon' = \Upsilon(t)' : [0,t] \to L(V,L(V,W))$ as
\begin{align}\label{Z-def-1}
\Upsilon_s &:= S_{t-s} Y_s, \quad s \in [0,t],
\\ \label{Z-def-2} \Upsilon_s' &:= S_{t-s} Y_s', \quad s \in [0,t].
\end{align}
Note that in the second definition \eqref{Z-def-2} we may regard $S_{t-s}$ as a linear operator from $L(L(V,W))$, due to the following auxiliary result.

\begin{lemma}\label{lemma-embedding-operators}
Let $E$ be another Banach space, and let $\varphi \in L(E,W)$ be arbitrary. Setting
\begin{align*}
L(V,E) \to L(V,W), \quad S \mapsto \varphi S,
\end{align*}
we may regard $\varphi$ as an element from $L(L(V,E),L(V,W))$. Furthermore, we have
\begin{align*}
| \varphi |_{L(L(V,E),L(V,W))} \leq | \varphi |_{L(E,W)}.
\end{align*}
\end{lemma}

\begin{proof}
We have
\begin{align*}
| \varphi |_{L(L(V,E),L(V,W))} &= \sup_{| S | \leq 1} | \varphi S |_{L(V,W)}
\\ &\leq \sup_{| S | \leq 1} | \varphi |_{L(E,W)} | S |_{L(V,E)} \leq | \varphi |_{L(E,W)},
\end{align*}
completing the proof.
\end{proof}

Now, we would like to define the rough convolution as
\begin{align}\label{convolution-try}
\int_0^t S_{t-s} Y_s \, d \mathbf{X}_s := \int_0^t \Upsilon_s \, d \mathbf{X}_s
\end{align}
by using the Gubinelli integral \eqref{candidate}.

\begin{remark}\label{rem-problem}
The problem with this approach is that in general we do \emph{not} have
$$(\Upsilon,\Upsilon') \in \scrd_X^{2 \alpha}([0,t],L(V,W)).$$
Indeed, in general we do \emph{not} even have
$$\Upsilon \in \calc^{\alpha}([0,t],L(V,W)).$$
In order to demonstrate this point, let $s,r \in [0,t]$ be arbitrary. Then we have
\begin{align*}
|\Upsilon_s - \Upsilon_r|_{L(V,W)} &= | S_{t-s} Y_s - S_{t-r} Y_r |_{L(V,W)}
\\ &\leq | S_{t-s} (Y_s - Y_r) |_{L(V,W)} + | (S_{t-s} - S_{t-r}) Y_r |_{L(V,W)}.
\end{align*}
Using \eqref{est-semigroup}, the first term can nicely be estimated as
\begin{align*}
| S_{t-s} (Y_s - Y_r) |_{L(V,W)} \leq |S_{t-s}|_{L(W)} |Y_s - Y_r|_{L(V,W)} \leq M e^{\omega T} \| Y \|_{\alpha} |s-r|^{\alpha}.
\end{align*}
However, for the second term
\begin{align*}
| (S_{t-s} - S_{t-r}) Y_r |_{L(V,W)}
\end{align*}
we cannot achieve such an estimate without additional conditions.
\end{remark}

\subsection{Rough convolutions by means of a mild sewing lemma}

In the previous section we have seen that defining rough convolutions by using the Gubinelli integral does not work without additional assumptions. In the literature (see, e.g. \cite{Gubinelli-Tindel, Hairer, Gerasimovics, Hesse-Neamtu-local, Hesse-Neamtu-global}), this issue is often resolved by establishing a mild version of the sewing lemma, which can then be used in order to define rough convolution integrals. In the sequel, we will outline these ideas.

\begin{definition}
Let $\alpha, \beta > 0$ be arbitrary. We denote by $\hat{\calc}_2^{\alpha,\beta}([0,T],W)$ the space of all functions $\Xi : \Delta_T^2 \to W$ such that $\Xi_{t,t} = 0$ for all $t \in [0,T]$ and
\begin{align*}
\| \Xi \|_{\alpha,\beta} := \| \Xi \|_{\alpha} + \| \hat{\delta} \Xi \|_{\beta} < \infty.
\end{align*}
Here $\| \Xi \|_{\alpha}$ denotes the H\"older norm \eqref{Hoelder-norm-sewing-1}. Furthermore, the function $\hat{\delta} \Xi : \Delta_T^3 \to W$ is defined as
\begin{align*}
\hat{\delta} \Xi_{s,u,t} := \Xi_{s,t} - S_{t-u} \Xi_{s,u} - \Xi_{u,t}, \quad (s,u,t) \in \Delta_T^3,
\end{align*}
and its H\"older norm is defined as
\begin{align*}
\| \hat{\delta} \Xi \|_{\beta} := \sup_{\genfrac{}{}{0pt}{}{(s,u,t) \in \Delta_T^3}{s < u < t}} \frac{|\hat{\delta} \Xi_{s,u,t}|}{|t-s|^{\beta}}.
\end{align*}
\end{definition}

\begin{remark}
Note that the only difference compared to Definition \ref{def-space-for-sewing} is given by the definitions of $\delta \Xi$ and $\hat{\delta} \Xi$. In the particular case $S_t = \Id$ for all $t \geq 0$ (which just means $A=0$) the two spaces $\calc_2^{\alpha,\beta}([0,T],W)$ and $\hat{\calc}_2^{\alpha,\beta}([0,T],W)$ coincide.
\end{remark}

After introducing the space $\hat{\calc}_2^{\alpha,\beta}([0,T],W)$ we wish to establish a \emph{mild} version of the \emph{sewing lemma} (Lemma \ref{lemma-sewing}). More precisely, under suitable conditions we wish to show that there exists a unique mapping $\cali : \hat{\calc}_2^{\alpha,\beta}([0,T],W) \to \calc^{\alpha}([0,T],W)$, which is characterized as follows:
\begin{enumerate}
\item We have $(\cali \Xi)_0 = 0$ for all $\Xi \in \hat{\calc}_2^{\alpha,\beta}([0,T],W)$.

\item For every $\Xi \in \hat{\calc}_2^{\alpha,\beta}([0,T],W)$ there exists a constant $C > 0$ such that
\begin{align*}
|(\cali \Xi)_{s,t} - \Xi_{s,t}| \leq C|t-s|^{\beta} \quad \text{for all $(s,t) \in \Delta_T^2$.}
\end{align*}
\end{enumerate}
Moreover, we would like $\cali \in L(\hat{\calc}_2^{\alpha,\beta}([0,T],W), \calc^{\alpha}([0,T],W))$, and that $\cali$ is given by
\begin{align*}
(\cali \Xi)_{s,t} = \lim_{|\Pi| \to 0} \sum_{[u,v] \in \Pi} S_{t-v} \Xi_{u,v} \quad \text{for all $\Xi \in \hat{\calc}_2^{\alpha,\beta}([0,T],W)$.}
\end{align*}
It is possible to establish such a mild sewing lemma, but this requires additional assumptions. Typically, it is assumed that the semigroup $(S_t)_{t \geq 0}$ is analytic, and one works on scales of Hilbert spaces or Banach spaces. Apart from the aforementioned references, we refer the reader to Exercise 4.16 in \cite{Friz-Hairer}.

In the next step, we define the \emph{mild Gubinelli derivative} as follows. Let $\bar{W}$ be another Banach space. We denote by $\hat{\calc}^{\alpha}([0,T],\bar{W})$ the space of all functions $Y : [0,T] \to \bar{W}$ such that
\begin{align*}
\| Y \|_{\alpha}^{\wedge} := \sup_{\genfrac{}{}{0pt}{}{s,t \in [0,T]}{s \neq t}} \frac{|\hat{\delta} Y_{s,t}|}{|t-s|^{\alpha}} < \infty,
\end{align*}
where we use the notation
\begin{align*}
\hat{\delta} Y_{s,t} := Y_t - S_{t-s} Y_s.
\end{align*}
Obviously, in case $A=0$ this is just the classical H\"{o}lder space $\calc^{\alpha}([0,T],\bar{W})$.

\begin{definition}
We introduce the following notions:
\begin{enumerate}
\item We say that a path $Y \in \hat{\calc}^{\alpha}([0,T],\bar{W})$ is \emph{mildly controlled} by $X$ if there exists $Y' \in \hat{\calc}^{\alpha}([0,T], L(V,\bar{W}))$ such that for the remainder term $R^Y : [0,T]^2 \to \bar{W}$ given by
\begin{align}\label{remainder-term}
R_{s,t}^Y := \hat{\delta} Y_{s,t} - S_{t-s} Y_s' X_{s,t} \quad \text{for all $s,t \in [0,T]$.}
\end{align}
we have $R^Y \in \calc_2^{2 \alpha}([0,T]^2,\bar{W})$, where $\hat{\delta} Y_{s,t} = Y_t - S_{t-s} Y_s$.

\item This defines the space of \emph{mildly controlled rough paths}, and we write
\begin{align*}
(Y,Y') \in \scrd_{S,X}^{2 \alpha}([0,T], \bar{W}).
\end{align*}
\item We call $Y'$ a \emph{mild Gubinelli derivative} of $Y$ (with respect to $X$).
\end{enumerate}
\end{definition}

\begin{remark}
Note the similarity to Definition \ref{def-Gubinelli-derivative}. The difference comes from the respective remainder terms $R^Y$. In the particular case $A=0$ the two spaces $\scrd_{X}^{2 \alpha}([0,T], \bar{W})$ and $\scrd_{S,X}^{2 \alpha}([0,T], \bar{W})$ coincide.
\end{remark}

\begin{definition}
We endow the space $\scrd_{S,X}^{2 \alpha}([0,T], W)$ with the seminorm
\begin{align*}
\| Y,Y' \|_{X,2\alpha}^{\wedge} := \| Y' \|_{\alpha}^{\wedge} + \| R^Y \|_{2 \alpha}.
\end{align*}
\end{definition}

\begin{remark}
Obviously, in case $A=0$ this coincides with the seminorm \eqref{seminorm-Gubinelli} introduced in Definition \ref{def-seminorm-Gubinelli}.
\end{remark}

Now, the idea is to consider the candidate
\begin{align}\label{candidate-rough}
\int_0^t S_{t-s} Y_s \, d \bfx_s := \lim_{|\Pi| \to 0} \sum_{[u,v] \in \Pi} S_{t-u} ( Y_u X_{u,v} + Y_u' \bbx_{u,v} )
\end{align}
for the \emph{rough convolution integral}, which apparently coincides with the Gubinelli integral \eqref{candidate} in case $A = 0$. More precisely, we wish to use the mild sewing lemma in order to establish the following types of results, which are similar to Theorem \ref{thm-Gubinelli} and Proposition \ref{prop-conv-rough}:
\begin{itemize}
\item For each mildly controlled rough path $(Y,Y') \in \scrd_{S,X}^{2 \alpha}([0,T], L(V,W))$ the $W$-valued rough convolution \eqref{candidate-rough} exists.

\item We obtain an estimate similar to \eqref{estimate-third-order}.

\item The rough convolution operator
\begin{align*}
I : \scrd_{S,X}^{2 \alpha}([0,T],L(V,W)) \to \scrd_{S,X}^{2 \alpha}([0,T],W)
\end{align*}
given by
\begin{align*}
(Y,Y') \mapsto \bigg( \int_0^{\cdot} S_{\cdot-s} Y_s \, d \bfx_s, Y \bigg)
\end{align*}
is a continuous linear operator.
\end{itemize}
It is possible to establish such results, but of course, as we use the mild sewing lemma, the additional assumptions outlined above are also required here. Apart from the aforementioned references, we refer the reader to Exercise 4.17 in \cite{Friz-Hairer}.

Once the rough convolution \eqref{candidate-rough} has been established, we can consider equation \eqref{mild} as a fixed point problem. Assuming suitable conditions on the coefficients $f_0$ and $f$, similar arguments as in Section \ref{sec-RDEs} provide existence and uniqueness results for mild solutions of the RPDE \eqref{RPDE}. We refer the reader to the aforementioned references \cite{Gubinelli-Tindel, Hairer, Gerasimovics, Hesse-Neamtu-local, Hesse-Neamtu-global}, where various situations have been investigated, for further details.

\section{Another approach to rough partial differential equations}\label{sec-RPDEs-Tappe}

In this section we present an alternative approach to solving RPDEs, which has been implemented in \cite{Tappe}. While the approach presented in the previous Section \ref{sec-RPDEs} requires assumptions on the semigroup $(S_t)_{t \geq 0}$ like analyticity, the results from this section hold true for arbitrary $C_0$-semigroups $(S_t)_{t \geq 0}$, but on the other hand we have to impose additional conditions on the coefficients $f_0$ and $f$.

The general mathematical framework is the same as in Section \ref{sec-RPDEs}. In particular, we consider RPDEs of the form \eqref{RPDE} with a driving signal $\bfx = (X,\bbx) \in \scrc^{\alpha}([0,T],V)$ for some index $\alpha \in (\frac{1}{3},\frac{1}{2}]$, and appropriate mappings $f_0 : [0,T] \times W \to W$ and $f : [0,T] \times W \to L(V,W)$. Moreover, the operator $A$ is the generator of a $C_0$-semigroup $(S_t)_{t \geq 0}$ on $W$.

The essential idea of the approach presented here is to define the rough convolution for suitable controlled rough paths $(Y,Y')$ according to \eqref{convolution-try} by using the Gubinelli integral \eqref{candidate}. In order to overcome the problems described in Remark \ref{rem-problem}, we will consider the first and second order domains of the generator $A$.

\subsection{Further results about strongly continuous semigroups}\label{subsec-semigroups}

At this point, let us collect further results about strongly continuous semigroups.  Let $(S_t)_{t \geq 0}$ be a $C_0$-semigroup on the Banach space $W$ with infinitesimal generator $A$. The following result is well-known; see, e.g. \cite[Thm. 1.2.4]{Pazy}.

\begin{lemma}\label{lemma-hg-rules}
The following statements are true:
\begin{enumerate}
\item For every $y \in D(A)$ the mapping
\begin{align*}
\mathbb{R}_+ \rightarrow W, \quad t \mapsto S_t y
\end{align*}
belongs to class $C^1(\mathbb{R}_+;W)$.

\item For all $t \geq 0$ and $y \in D(A)$ we have $S_t y \in D(A)$ and
\begin{align*}
\frac{d}{dt} S_t y = A S_t y = S_t A y.
\end{align*}

\item For all $t \geq 0$ and $y \in W$ we have $\int_0^t S_s y \, ds \in D(A)$ and
\begin{align*}
A \bigg( \int_0^t S_s y \, ds \bigg) = S_t y - y.
\end{align*}

\item For all $t \geq 0$ and $y \in D(A)$ we have
\begin{align*}
\int_0^t S_s A y \, ds = S_t y - y.
\end{align*}
\end{enumerate}
\end{lemma}

The \emph{domain} $D(A)$ endowed with the graph norm
\begin{align*}
|y|_{D(A)} := |y| + |Ay| \quad \text{for all $y \in D(A)$}
\end{align*}
is a Banach space. Moreover, the \emph{second order domain}
\begin{align*}
D(A^2) := \{ y \in D(A) :A y \in D(A) \}
\end{align*}
endowed with the graph norm
\begin{align*}
|y|_{D(A^2)} := |y| + |Ay| + |A^2 y| \quad \text{for all $y \in D(A^2)$}
\end{align*}
is a Banach space as well. Note that $D(A^2) \hookrightarrow D(A) \hookrightarrow W$ with continuous embedding. More precisely, we have the following result, which is easily verified.

\begin{lemma}\label{lemma-domains-embedding}
The following statements are true:
\begin{enumerate}
\item We have $D(A) \subset W$ and $| x | \leq |x|_{D(A)}$ for all $x \in D(A)$.

\item We have $D(A^2) \subset D(A)$ and $| x |_{D(A)} \leq |x|_{D(A^2)}$ for all $x \in D(A^2)$.
\end{enumerate}
\end{lemma}

The following auxiliary result is also straightforward to prove.

\begin{lemma}\label{lemma-restricted-semigroup}
The restriction $(S_t|_{D(A)})_{t \geq 0}$ is a $C_0$-semigroup on $(D(A),|\cdot|_{D(A)})$ with generator $A$ on the domain $D(A^2)$.
\end{lemma}

For what follows, we agree on the notation
\begin{align*}
S_{s,t} := S_t - S_s \quad \text{for all $s,t \in \bbr_+$.}
\end{align*}
Furthermore, we fix some $T \in \bbr_+$.

\begin{proposition}\label{prop-orbit-map}
For all $y \in D(A)$ and all $s,t \in [0,T]$ we have
\begin{align*}
| S_{s,t} y | \leq M e^{\omega T} | y |_{D(A)} \, |t-s|.
\end{align*}
\end{proposition}

\begin{proof}
By Lemma \ref{lemma-hg-rules} and the growth estimate \eqref{est-semigroup} we obtain
\begin{align*}
| S_{s,t} y | &= | S_{t} y - S_{s} y | = \bigg| \int_{s}^{t} S_u Ay \, du \bigg| \leq \int_{s}^{t} | S_u Ay | \, du
\\ &\leq \int_{s}^{t} | S_u | \, | Ay | \, du \leq Me^{\omega T} | y |_{D(A)} |t - s|,
\end{align*}
completing the proof.
\end{proof}

\begin{corollary}\label{cor-orbit-map-1}
For all $s,t \in [0,T]$ we have
\begin{align*}
| S_{s,t} |_{L(D(A),W)} \leq M e^{\omega T} |t-s|.
\end{align*}
\end{corollary}

\begin{corollary}\label{cor-orbit-map-2}
Let $V$ be another Banach space. Then for all $Y \in L(V,D(A))$ and all $s,t \in [0,T]$ we have
\begin{align*}
| S_{s,t} Y |_{L(V,W)} \leq M e^{\omega T} | Y |_{L(V,D(A))} \, |t-s|.
\end{align*}
\end{corollary}

The following result will be crucial for the analysis of rough convolutions later on. We refer to \cite[Prop. 2.4]{Tappe} for its proof.

\begin{proposition}\label{prop-est-quad}
Let $y \in D(A^2)$ be arbitrary. Then we have
\begin{align}\label{est-quad}
| S_{s-r,t-r}y - S_{s-q,t-q}y | \leq M e^{2 \omega T} |y|_{D(A^2)} |t-s| \, |r-q|
\end{align}
for all $s,t \in [0,T]$ with $s \leq t$ and all $q,r \in [0,s]$.
\end{proposition}

\subsection{Rough convolutions}

After these preparations, we resume the idea to define the rough convolution according to \eqref{convolution-try} by using the Gubinelli integral \eqref{candidate}. For this purpose, for any fixed $t \in [0,T]$ we consider the paths $\Upsilon = \Upsilon(t) : [0,t] \to L(V,W)$ and $\Upsilon' = \Upsilon(t)' : [0,t] \to L(V,L(V,W))$ given by \eqref{Z-def-1} and \eqref{Z-def-2} for a suitable controlled rough path $(Y,Y')$. However, rather than \eqref{controlled-in-W} let us make the stronger assumption that the controlled rough path satisfies
\begin{align}\label{controlled-in-D-A}
(Y,Y') \in \scrd_X^{2 \alpha}([0,T],L(V,D(A))) \hookrightarrow \scrd_X^{2 \alpha}([0,T],L(V,W)),
\end{align}
where we recall Lemma \ref{lemma-domains-embedding}. This allows us to extend the ideas from Remark \ref{rem-problem} as follows.

\begin{lemma}\label{lemma-semigroup-term-1}
Suppose that \eqref{controlled-in-D-A} holds true. Then we have
$$\Upsilon \in \calc^{\alpha}([0,t],L(V,W)).$$
\end{lemma}

\begin{proof}
Note that \eqref{controlled-in-D-A} and Lemma \ref{lemma-domains-embedding} imply
$$Y \in \calc^{\alpha}([0,T],L(V,D(A))) \hookrightarrow \calc^{\alpha}([0,T],L(V,W)).$$
Let $s,r \in [0,t]$ be arbitrary. Then we have
\begin{align*}
|\Upsilon_s - \Upsilon_r|_{L(V,W)} &= | S_{t-s} Y_s - S_{t-r} Y_r |_{L(V,W)}
\\ &\leq | S_{t-s} (Y_s - Y_r) |_{L(V,W)} + | S_{t-r,t-s} Y_r |_{L(V,W)}.
\end{align*}
Using \eqref{est-semigroup}, the first term can be estimated as
\begin{align*}
| S_{t-s} (Y_s - Y_r) |_{L(V,W)} \leq | S_{t-s} |_{L(W)} |Y_s - Y_r|_{L(V,W)} \leq M e^{\omega T} \| Y \|_{\alpha} |s-r|^{\alpha}.
\end{align*}
So far, these are just the calculations that we have seen in Remark \ref{rem-problem}. Moreover, by Corollary \ref{cor-orbit-map-2} the second term is estimated as
\begin{align*}
| S_{t-r,t-s} Y_r |_{L(V,W)} \leq M e^{\omega T} | Y_r |_{L(V,D(A))} |s-r| \leq M e^{\omega T} \| Y \|_{\infty} |s-r|,
\end{align*}
completing the proof.
\end{proof}

\begin{lemma}\label{lemma-semigroup-term-2}
Suppose that \eqref{controlled-in-D-A} holds true. Then we have
$$\Upsilon' \in \calc^{\alpha}([0,t],L(V,L(V,W))).$$
\end{lemma}

\begin{proof}
Note that \eqref{controlled-in-D-A} and Lemma \ref{lemma-domains-embedding} imply
$$Y' \in \calc^{\alpha}([0,t],L(V,L(V,D(A)))) \hookrightarrow \calc^{\alpha}([0,t],L(V,L(V,W))).$$ Let $s,r \in [0,t]$ be arbitrary. Then we have
\begin{align*}
|\Upsilon_s' - \Upsilon_r'|_{L(V,L(V,W))} &= | S_{t-s} Y_s' - S_{t-r} Y_r' |_{L(V,L(V,W))}
\\ &\leq | S_{t-s} (Y_s' - Y_r') |_{L(V,L(V,W))} + | S_{t-r,t-s} Y_r' |_{L(V,L(V,W))}.
\end{align*}
Recall that by Lemma \ref{lemma-embedding-operators} we may regard $S_{t-s}$ as an element from $L(L(V,W))$. Thus,
using Lemma \ref{lemma-embedding-operators} and \eqref{est-semigroup}, the first term can be estimated as
\begin{align*}
| S_{t-s} (Y_s' - Y_r') |_{L(V,L(V,W))} &\leq | S_{t-s} |_{L(L(V,W))} |Y_s' - Y_r'|_{L(V,L(V,W))}
\\ &\leq | S_{t-s} |_{L(W)} |Y_s' - Y_r'|_{L(V,L(V,W))}
\\ &\leq M e^{\omega T} \| Y' \|_{\alpha} |s-r|^{\alpha}.
\end{align*}
Moreover, by Lemma \ref{lemma-embedding-operators} and Corollary \ref{cor-orbit-map-1} the second term is estimated as
\begin{align*}
| S_{t-r,t-s} Y_r' |_{L(V,L(V,W))} &\leq | S_{t-r,t-s} |_{L(L(V,D(A)), L(V,W) )} | Y_r' |_{L(V,L(V,D(A)))}
\\ &\leq | S_{t-r,t-s} |_{L(D(A), W )} | Y_r' |_{L(V,L(V,D(A)))}
\\ &\leq M e^{\omega T} | Y' |_{\infty} |s-r|,
\end{align*}
completing the proof.
\end{proof}

Now, we are in the position to show that $(\Upsilon,\Upsilon')$ actually is a controlled rough path.

\begin{proposition}\label{prop-convolution-well-defined}
Suppose that \eqref{controlled-in-D-A} holds true. Then we have
\begin{align*}
(\Upsilon,\Upsilon') \in \scrd_X^{2\alpha}([0,t],L(V,W)).
\end{align*}
\end{proposition}

\begin{proof}
In view of Lemma \ref{lemma-semigroup-term-1} and Lemma \ref{lemma-semigroup-term-2} we only need to show that
$$R^Z \in \calc_2^{2 \alpha}([0,t]^2,L(V,W)).$$
Note that \eqref{controlled-in-D-A} implies
$$R^Y \in \calc_2^{2 \alpha}([0,t]^2,L(V,D(A))).$$
Let $s,r \in [0,t]$ be arbitrary. Then we have
\begin{align*}
R_{r,s}^\Upsilon &= \Upsilon_{r,s} - \Upsilon_r' X_{r,s} = \Upsilon_s - \Upsilon_r - \Upsilon_r' X_{r,s}
\\ &= S_{t-s} Y_s - S_{t-r} Y_r - S_{t-r} Y_r' X_{r,s}
\\ &= ( S_{t-s} - S_{t-r} ) Y_s + S_{t-r} ( Y_{r,s} - Y_r' X_{r,s} )
\\ &= S_{t-r,t-s} Y_s + S_{t-r} R_{r,s}^Y.
\end{align*}
Hence by Corollary \ref{cor-orbit-map-2} we obtain
\begin{align*}
| R_{r,s}^\Upsilon |_{L(V,W)} &\leq | S_{t-r,t-s} Y_s |_{L(V,W)} + | S_{t-r} R_{r,s}^Y |_{L(V,W)}
\\ &\leq M e^{\omega T} |Y_s|_{L(V,D(A))} |s-r| + M e^{\omega T} | R_{r,s}^Y |_{L(V,D(A))}
\\ &\leq M e^{\omega T} \| Y \|_{\infty} |s-r| + M e^{\omega T} \| R^Y \|_{2 \alpha} |s-r|^{2 \alpha},
\end{align*}
completing the proof.
\end{proof}

Consequently, for a controlled rough path $(Y,Y')$ satisfying \eqref{controlled-in-D-A} we can define the rough convolution according to \eqref{convolution-try} by using the Gubinelli integral \eqref{candidate}. In the next step, let us have a closer look at the terms appearing in the fixed point equation \eqref{mild}.

\begin{remark}
For any $\xi \in W$ we denote by $u_{\xi} : [0,T] \to W$ the orbit map $u_{\xi}(t) := S_t \xi$, $t \in [0,T]$. Note that the orbit map $u_{\xi}$ is always continuous, but it does not need to be H\"{o}lder continuous. More precisely, it may not be true that $u_{\xi} \in \calc^{\alpha}([0,T],W)$, and in particular, it may not be true that $(u_{\xi},0) \in \scrd_X^{2\alpha}([0,T],W)$. However, as the next result shows, choosing starting points from the domain $D(A)$, we obtain the desired property.
\end{remark}

\begin{proposition}
The following statements are true:
\begin{enumerate}
\item The orbit map operator
\begin{align*}
\Omega : \big( D(A), |\cdot|_{D(A)} \big) \to \big( \scrd_X^{2\alpha}([0,T],W), \| \cdot \|_{X,2\alpha} \big), \quad \xi \mapsto ( u_{\xi}, 0 )
\end{align*}
is a continuous linear operator between seminormed spaces with operator norm bounded by
\begin{align*}
\| \Omega \| \leq M e^{\omega T} T^{1-2 \alpha}.
\end{align*}
\item The orbit map operator
\begin{align*}
\Omega : \big( D(A), |\cdot|_{D(A)} \big) \to \big( \scrd_X^{2\alpha}([0,T],W), \interleave \cdot \interleave_{X,2\alpha} \big), \quad \xi \mapsto ( u_{\xi}, 0 )
\end{align*}
is a continuous linear operator between Banach spaces with operator norm bounded by
\begin{align*}
\| \Omega \| \leq M e^{\omega T} T^{1-2 \alpha} + 1.
\end{align*}
\end{enumerate}
\end{proposition}

\begin{proof}
Let $\xi \in D(A)$ be arbitrary. By Proposition \ref{prop-orbit-map}, for all $s,t \in [0,T]$ we have
\begin{align*}
| u_{\xi}(t) - u_{\xi}(s) | = | S_{s,t} \xi | \leq M e^{\omega T} |\xi|_{D(A)} |t-s| \leq M e^{\omega T} |\xi|_{D(A)} T^{1-2\alpha} |t-s|^{2\alpha},
\end{align*}
showing that $u_{\xi} \in \calc^{2\alpha}([0,T],W)$ with
\begin{align*}
\| u_{\xi} \|_{2 \alpha} \leq M e^{\omega T} |\xi|_{D(A)} T^{1-2\alpha}.
\end{align*}
Now, we set $(Z,Z') := \Omega(\xi) = (u_{\xi},0)$. Then we have $Z_0 = \xi$ and $Z_0' = 0$. Thus, by Exercise \ref{exercise-reg-derivative-zero} we obtain $(Z,Z') \in \scrd_X^{2\alpha}([0,T],W)$ and
\begin{align*}
\| Z,Z' \|_{X,2\alpha} = \| Z \|_{2\alpha} \leq M e^{\omega T} T^{1-2\alpha} |\xi|_{D(A)}.
\end{align*}
Moreover, we have
\begin{align*}
\interleave Z,Z' \interleave_{X,2\alpha} = |Z_0| + |Z_0'| + \| Z,Z' \|_{X,2\alpha} = |\xi| + \| Z \|_{2\alpha} \leq |\xi|_{D(A)} + \| Z \|_{2\alpha},
\end{align*}
completing the proof.
\end{proof}

\begin{lemma}\label{lemma-reg-conv-Hoelder-2}
Let $Y : [0,T] \to D(A)$ be measurable and bounded. We define the path $Z : [0,T] \to W$ as
\begin{align*}
Z_t := \int_0^t S_{t-s} Y_s \, ds, \quad t \in [0,T].
\end{align*}
Then we have
\begin{align*}
|Z_{s,t}| \leq (1+T) M e^{\omega T} \| Y \|_{\infty} |t-s| \quad \text{for all $s,t \in [0,T]$.}
\end{align*}
\end{lemma}

\begin{proof}
Let $s,t \in [0,T]$ with $s \leq t$ be arbitrary. Then we have
\begin{align*}
Z_{s,t} &= \int_0^t S_{t-r} Y_r \, dr - \int_0^s S_{s-r} Y_r \, dr
\\ &= \int_s^t S_{t-r} Y_r \, dr + \int_0^s (S_{t-r} - S_{s-r}) Y_r \, dr.
\end{align*}
Therefore, using the estimate \eqref{est-semigroup} and Proposition \ref{prop-orbit-map} we obtain
\begin{align*}
|Z_{s,t}| &\leq \bigg| \int_s^t S_{t-r} Y_r \, dr \bigg| + \bigg| \int_0^s (S_{t-r} - S_{s-r}) Y_r \, dr \bigg|
\\ &\leq \int_s^t |S_{t-s} Y_r| \, dr + \int_0^s |S_{t-r} Y_r - S_{s-r} Y_r| \, dr
\\ &\leq M e^{\omega T} \int_s^t |Y_r| \, dr + M e^{\omega T} \int_0^s |Y_r|_{D(A)} |t-s| \, dr.
\\ &\leq M e^{\omega T} \|Y\|_{\infty} |t-s| + M e^{\omega T} T \|Y\|_{\infty} |t-s|,
\end{align*}
completing the proof.
\end{proof}

\begin{proposition}\label{prop-reg-conv-Hoelder-semigroup}
The convolution operator
\begin{align*}
\mathscr{I} : \big( C([0,T],D(A)), \| \cdot \|_{\infty} \big) \to \big( \scrd_X^{2\alpha}([0,T],W), \interleave \cdot \interleave_{X,2\alpha} \big)
\end{align*}
given by
\begin{align*}
Y \mapsto \bigg( \int_0^{\cdot} S_{\cdot - s} Y_s \, ds, 0 \bigg)
\end{align*}
is a continuous linear operator between Banach spaces with operator norm bounded by
\begin{align*}
\| \mathscr{I} \| \leq M e^{\omega T} (1+T) T^{1-2\alpha}.
\end{align*}
\end{proposition}

\begin{proof}
Let $Y \in C([0,T],D(A))$ be arbitrary, and set $(Z,Z') := \mathscr{I}(Y)$. By Lemma \ref{lemma-reg-conv-Hoelder-2}, for all $s,t \in [0,T]$ we have
\begin{align*}
|Z_{s,t}| \leq (1+T) M e^{\omega T} \| Y \|_{\infty} |t-s| \leq (1+T) M e^{\omega T} T^{1-2\alpha} \| Y \|_{\infty} |t-s|^{2\alpha},
\end{align*}
showing that $Z \in \calc^{2\alpha}([0,T],W)$ with
\begin{align*}
\| Z \|_{2\alpha} \leq (1+T) M e^{\omega T} T^{1-2\alpha} \| Y \|_{\infty}.
\end{align*}
Note that $Z_0 = 0$ and $Z_0' = 0$. Therefore, by Exercise \ref{exercise-reg-derivative-zero} we obtain $(Z,Z') \in \scrd_X^{2\alpha}([0,T],W)$ and
\begin{align*}
\interleave Z,Z' \interleave_{X,2\alpha} = \| Z,Z' \|_{X,2\alpha} = \| Z \|_{2 \alpha} \leq M e^{\omega T} (1+T) T^{1-2\alpha} \| Y \|_{\infty},
\end{align*}
completing the proof.
\end{proof}

So far, we have seen that the first two terms appearing at the right-hand side of the fixed point equation \eqref{mild} preserve the required regularity. As it turns out, there is a problem with the third term in \eqref{mild}. Indeed, even though for controlled rough paths $(Y,Y')$ satisfying \eqref{controlled-in-D-A} the rough convolution \eqref{convolution-try} is well-defined, it may happen that for the path
\begin{align}\label{Z-rough-convolution}
(Z,Z') := \bigg( \int_0^{\cdot} S_{\cdot - s} Y_s \, d \bfx_s, Y \bigg)
\end{align}
we do \emph{not} have
\begin{align*}
(Z,Z') \in \scrd_X^{2 \alpha}([0,T],W).
\end{align*}
We can, however, overcome this difficulty by considering the second order domain $D(A^2)$ and controlled rough paths
\begin{align}\label{controlled-in-D-A-2}
(Y,Y') \in \scrd_X^{2 \alpha}([0,T],L(V,D(A^2))) \hookrightarrow \scrd_X^{2 \alpha}([0,T],L(V,D(A))),
\end{align}
where we recall Lemma \ref{lemma-domains-embedding}. The proof of the following result is rather complex, and we refer to \cite[Prop. 2.14]{Tappe} for further details. We just point out that a key result for its proof is the estimate \eqref{est-quad} from Proposition \ref{prop-est-quad}.

\begin{proposition}\label{prop-conv-rough-semigroup}
The rough convolution operator
\begin{align*}
\mathscr{J} : \big( \scrd_X^{2 \alpha}([0,T],L(V,D(A^2))), \interleave \cdot \interleave_{X,2\alpha} \big) \to \big( \scrd_X^{2 \alpha}([0,T],W), \interleave \cdot \interleave_{X,2\alpha} \big)
\end{align*}
given by
\begin{align*}
(Y,Y') \mapsto \bigg( \int_0^{\cdot} S_{\cdot - s} Y_s \, d \bfx_s, Y \bigg)
\end{align*}
is a continuous linear operator between Banach spaces with operator norm bounded by
\begin{align*}
\| \mathscr{J} \| \leq C (  \interleave \mathbf{X} \interleave_{\alpha} + T^{\alpha} ),
\end{align*}
where the constant $C > 0$ depends on $\alpha$, $T$, $\mathbf{X}$ and $M$, $\omega$, but does not depend on $T \leq 1$.
\end{proposition}

In particular, for a controlled rough path $(Y,Y')$ satisfying \eqref{controlled-in-D-A-2} the rough convolution $(Z,Z')$ given by \eqref{Z-rough-convolution} is also a controlled rough path, and hence, all terms appearing at the right-hand side of the fixed point equation \eqref{mild} preserve the required regularity.

\subsection{The fixed point argument}

After establishing the required results about rough convolutions, we can come back to the RPDE \eqref{RPDE}. Let $f_0 \in \Lip([0,T] \times W,D(A))$ and $f \in C_b^{2\alpha,3}([0,T] \times W, L(V,D(A^2)))$ be arbitrary. We also fix an initial condition $\xi \in D(A)$. For every $t \in [0,T]$ we define the complete metric space $\bbb_t \subset \scrd_X^{2\alpha}([0,t],W)$ as \eqref{B-fixed-point}, and the mapping $\Phi_t : \scrd_X^{2\alpha}([0,t],W) \to \scrd_X^{2\alpha}([0,t],W)$ as
\begin{align}\label{def-Phi-2}
\Phi_t(Y,Y') := (u_\xi, 0) + (\Gamma(Y),0) + (\Psi(Y),f(Y)),
\end{align}
where $u_{\xi} : [0,t] \to W$ denotes the orbit map $u_{\xi}(s) := S_s \xi$, $s \in [0,t]$, and the paths $\Gamma(Y), \Psi(Y) : [0,t] \to W$ are defined as
\begin{align*}
\Gamma(Y) &:= \int_0^{\cdot} S_{\cdot-s} f_0(s,Y_s) ds,
\\ \Psi(Y) &:= \int_0^{\cdot} S_{\cdot-s} f(s,Y_s) d \mathbf{X}_s.
\end{align*}
Then the mapping $\Phi_t$ is well-defined due to our findings from the previous section. Thus, we can proceed as in Section \ref{sec-RDEs} and generalize Theorem \ref{thm-RDGL-main-local} concerning the existence of local solutions as follows. Just as in the proof of Theorem \ref{thm-RDGL-main-local}, the essential idea for the proof of the upcoming Theorem \ref{thm-RPDE-main-local} is to find $T_0 \in (0,T]$ with $T_0 \leq 1$ small enough such that $\Phi_{T_0} : \bbb_{T_0} \to \bbb_{T_0}$ is a contraction. We refer to \cite[Thm. 2.2]{Tappe} for the precise proof and further details.

\begin{theorem}\label{thm-RPDE-main-local}
Let $\bfx = (X,\bbx) \in \scrc^{\beta}([0,T],V)$ be a rough path for some index $\beta \in (\frac{1}{3},\frac{1}{2}]$, and let $f_0 \in \Lip([0,T] \times W,D(A))$ and $f \in C_b^{2\beta,3}([0,T] \times W, L(V,D(A^2)))$ be mappings. Then for every $\xi \in D(A)$ there exist $T_0 \in (0,T]$ and a unique local mild solution $(Y,Y') \in \scrd_X^{2 \beta}([0,T_0],W)$ to the RPDE \eqref{RPDE} with $Y_0 = \xi$.
\end{theorem}

Now we can proceed further as in Section \ref{sec-RDEs} and derive the following generalization of Theorem \ref{thm-RDGL-main} concerning the existence of global solutions. We refer to \cite[Thm. 2.3]{Tappe} for the precise proof and further details. Just as in the proof of Theorem \ref{thm-RDGL-main}, the essential idea for the proof of the upcoming Theorem \ref{thm-RPDE-main} is to find $t \in (0,1]$ small enough such that $T = nt$ for some $n \in \bbn$, and to concatenate local solutions, which are obtained from the Banach fixed point theorem, on the intervals $[(k-1)t,kt]$ for $k=1,\ldots,n$.

\begin{theorem}\label{thm-RPDE-main}
Let $\bfx = (X,\bbx) \in \scrc^{\beta}([0,T],V)$ be a rough path for some index $\beta \in (\frac{1}{3},\frac{1}{2}]$, and let $f_0 \in \Lip([0,T] \times W,D(A))$ and $f \in C_b^{2\beta,3}([0,T] \times W, L(V,D(A^2)))$ be mappings such that
\begin{align}\label{f0-restriction}
&f_0|_{[0,T] \times D(A)} \in \Lip([0,T] \times D(A),D(A^2)),
\\ \label{f-restriction} &f|_{[0,T] \times D(A)} \in C_b^{2\beta,3}([0,T] \times D(A), L(V,D(A^3))).
\end{align}
Then for every $\xi \in D(A^2)$ there exists a unique mild solution $$(Y,Y') \in \scrd_X^{2 \beta}([0,T],W)$$ to the RPDE \eqref{RPDE} with $Y_0 = \xi$.
\end{theorem}

Here $D(A^3)$ denotes the \emph{third order domain}
\begin{align*}
D(A^3) := \{ y \in D(A^2) : A^2 y \in D(A) \},
\end{align*}
endowed with the graph norm
\begin{align*}
|y|_{D(A^3)} := |y| + |Ay| + |A^2 y| + |A^3 y| \quad \text{for all $y \in D(A^3)$.}
\end{align*}
Moreover, we obtain the following result concerning the It\^{o}-Lyons map.

\begin{proposition}\label{prop-Ito-Lyons-PDE}
Let $\beta \in (\frac{1}{3},\frac{1}{2}]$ be an index, and let $f_0 \in \Lip([0,T] \times W,D(A))$ and $f \in C_b^{2\beta,3}([0,T] \times W, L(V,D(A^2)))$ be mappings such that \eqref{f0-restriction} and \eqref{f-restriction} are fulfilled. Then the It\^{o}-Lyons map
\begin{align*}
D(A^2) \times \big( \scrc^{\beta}([0,T],V), d_{\beta} \big) \to \calc^{\beta}([0,T],W) \times \calc^{\beta}([0,T],L(V,W))
\end{align*}
induced by Theorem \ref{thm-RPDE-main} is locally Lipschitz continuous.
\end{proposition}

Now, we wish to apply Theorem \ref{thm-RPDE-main} in order to establish existence and uniqueness results for stochastic partial differential equations. For this purpose, we extend the results from Section \ref{sec-BB-rough-path} to the infinite dimensional setting in the upcoming sections.

\subsection{Infinite dimensional Wiener process as a rough path}\label{sec-Wiener-enhanced}

In this section we demonstrate how typical sample paths of an infinite dimensional Wiener process can be realized as rough paths. We refer, for example, to \cite{Da_Prato}, \cite{Atma-book} or \cite{Liu-Roeckner} for more details about infinite dimensional Wiener processes, and to \cite[Sec. 2.10.1]{Tappe} for more details about the material of this section.

Let $U$ be a separable Hilbert space, and let $X$ be an $U$-valued $Q$-Wiener process for some nuclear, self-adjoint, positive definite linear operator $Q \in L_1^{++}(U)$; see, e.g. \cite[Def. 4.2]{Da_Prato}. Then there exist an orthonormal basis $\{ e_k \}_{k \in \bbn}$ of $U$ and a sequence $(\lambda_k)_{k \in \bbn} \subset (0,\infty)$ with $\sum_{k=1}^{\infty} \lambda_k < \infty$ such that
\begin{align}\label{Q-diagonal}
Q e_k = \lambda_k e_k \quad \text{for all $k \in \bbn$.}
\end{align}
According to \cite[Prop. 4.3]{Da_Prato} the sequence $(\beta^k)_{k \in \bbn}$ defined as
\begin{align}\label{beta-j}
\beta^k := \frac{1}{\sqrt{\lambda_k}} \langle X,e_k \rangle_U, \quad k \in \bbn
\end{align}
is a sequence of independent real-valued standard Wiener processes, and according to \cite[Prop. 2.1.10]{Liu-Roeckner} the $Q$-Wiener process $X$ admits the series representation
\begin{align}\label{series-Wiener}
X_t = \sum_{k=1}^{\infty} \sqrt{\lambda_k} \beta_t^k e_k, \quad t \in [0,T].
\end{align}
In the spirit of this series representation, we define the second order process $\bbx$ as
\begin{align}\label{series-Wiener-second-order}
\bbx_{s,t} := \sum_{j,k=1}^{\infty} \sqrt{\lambda_j \lambda_k} \bigg( \int_s^t \beta_{s,r}^j d \beta_r^k \bigg) (e_j \otimes e_k), \quad s,t \in [0,T],
\end{align}
which is the infinite dimensional analogue of the It\^{o}-enhancement \eqref{Ito-enhancement}. Now, we define the \emph{It\^{o}-enhanced $Q$-Wiener process} $\mathbf{X} := (X,\bbx)$, and can generalize Proposition \ref{prop-BB-Ito-rough-path} as follows.

\begin{proposition}\cite[Prop. 2.22]{Tappe}\label{prop-Wiener-rough-path}
For each $\alpha \in (\frac{1}{3},\frac{1}{2})$ we have $\bbp$-almost surely
$$\mathbf{X} \in \scrc^{\alpha}([0,T],U).$$
\end{proposition}

\subsection{Coincidence of the two integrals}

In this section we recall a result showing that the rough It\^{o} integral and the stochastic It\^{o} integral also coincide in infinite dimension. We refer to \cite[Sec. 2.10.2]{Tappe} for more details about the material of this section.

Let us briefly recall some preliminaries about the infinite dimensional It\^{o} integral. As in the previous section, let $\mathbf{X} = (X,\bbx)$ be an It\^{o}-enhanced $Q$-Wiener process with values in a separable Hilbert space $U$. The space $U_0 := Q^{1/2}(U)$, endowed with the inner product
\begin{align*}
\langle u,v \rangle_{U_0} := \langle Q^{-1/2}u, Q^{-1/2}v \rangle_U, \quad u,v \in U_0
\end{align*}
is another separable Hilbert space. Note that
\begin{align*}
Q^{1/2} : (U,\| \cdot \|_U) \to (U_0,\| \cdot \|_{U_0})
\end{align*}
is an isometric isomorphism, and that $\{ \sqrt{\lambda_k} e_k \}_{k \in \bbn}$ is an orthonormal basis of $U_0$, where $\{ e_k \}_{k \in \bbn}$ and $(\lambda_k)_{k \in \bbn}$ are the sequences appearing in \eqref{Q-diagonal}.

Now, let $H$ be another separable Hilbert space. We denote by $L_2(U_0,H)$ the space of all Hilbert-Schmidt operators from $U_0$ into $H$. Furthermore, we define
\begin{align*}
L(U,H)_0 := \{ \Phi|_{U_0} : \Phi \in L(U,H) \}.
\end{align*}
The next result shows that $L(U,H)$ is continuously embedded in $L_2(U_0,H)$.

\begin{lemma}\label{lemma-HS-embedded}
We have $L(U,H)_0 \subset L_2(U_0,H)$ and
\begin{align*}
| \Phi|_{U_0} |_{L_2(U_0,H)} \leq \sqrt{\tr(Q)} | \Phi |_{L(U,H)} \quad \text{for all $\Phi \in L(U,H)$.}
\end{align*}
\end{lemma}

\begin{proof}
Recalling that $\{ \sqrt{\lambda_k} e_k \}_{k \in \bbn}$ is an orthonormal basis of $U_0$, we have
\begin{align*}
| \Phi |_{L_2(U_0,H)}^2 = \sum_{k=1}^{\infty} | \Phi(\sqrt{\lambda_k} e_k) |^2 = \sum_{k=1}^{\infty} \lambda_k | \Phi(e_k) |^2 \leq \tr(Q) | \Phi |_{L(U,H)}^2,
\end{align*}
completing the proof.
\end{proof}

The $H$-valued \emph{It\^{o} integral} $\int_0^T \Phi_s d X_s$ can be defined for every $L_2(U_0,H)$-valued predictable process $\Phi$ such that
\begin{align}\label{integrable-Ito}
\bbe \bigg[ \int_0^T |\Phi_s|_{L_2(U_0,H)}^2 ds \bigg] < \infty,
\end{align}
or, more generally, for every $L_2(U_0,H)$-valued predictable process $\Phi$ satisfying
\begin{align*}
\bbp \bigg( \int_0^T |\Phi_s|_{L_2(U_0,H)}^2 ds < \infty \bigg) = 1.
\end{align*}
We refer to \cite{Da_Prato}, \cite{Atma-book} or \cite{Liu-Roeckner} for further details.

Now, let $\alpha \in (\frac{1}{3},\frac{1}{2})$ be an arbitrary index. According to Proposition \ref{prop-Wiener-rough-path} there is a $\bbp$-nullset $N_1$ such that $\mathbf{X} \in \scrc^{\alpha}([0,T],U)$ on $N_1^c$. Keeping in mind Lemma \ref{lemma-HS-embedded}, we can generalize Proposition \ref{prop-integrals-Ito-coincide} as follows.

\begin{proposition}\cite[Prop. 2.23]{Tappe}\label{prop-integrals-coincide}
Let $Y$ be a continuous $L(U,H)$-valued process, and let $Y'$ be a continuous $L(U,L(U,H))$-valued processes such that $(Y,Y') \in \scrd_X^{2 \alpha}([0,T],L(U,H))$ on $N_1^c$. Then the following statements are true:
\begin{enumerate}
\item The $H$-valued rough integral
\begin{align*}
\int_0^T Y_s \, d \mathbf{X}_s = \lim_{|\Pi| \to 0} \sum_{[u,v] \in \Pi} ( Y_u X_{u,v} + Y_u' \bbx_{u,v} )
\end{align*}
exists on $N_1^c$.

\item If $Y$ and $Y'$ are adapted and bounded, then there is a $\bbp$-nullset $N$ with $N_1 \subset N$ such that
\begin{align*}
\int_0^T Y_s \, d \mathbf{X}_s = \int_0^T Y_s|_{U_0} \, dX_s \quad \text{on $N^c$.}
\end{align*}
\end{enumerate}
\end{proposition}

\subsection{Stochastic partial differential equations}

Now, we are ready for our study of stochastic partial differential equations (SPDEs). We refer to \cite[Sec. 2.10.3]{Tappe} for more details about the material of this section.

Let $\mathbf{X} = (X,\bbx)$ be an It\^{o}-enhanced $Q$-Wiener process with values in a separable Hilbert space $U$, as introduced in Section \ref{sec-Wiener-enhanced}. Furthermore, let $W$ be a Banach space, and let $A$ be the generator of a $C_0$-semigroup $(S_t)_{t \geq 0}$ on $W$. Consider the $W$-valued random RPDE
\begin{align}\label{random-SPDE}
\left\{
\begin{array}{rcl}
dY_t & = & (A Y_t + f_0(t,Y_t)) dt + f(t,Y_t) d \mathbf{X}_t
\\ Y_0 & = & \xi,
\end{array}
\right.
\end{align}
and, in case $W$ is a separable Hilbert space $H$, the $H$-valued SPDE
\begin{align}\label{SPDE-Wiener}
\left\{
\begin{array}{rcl}
dY_t & = & (A Y_t + f_0(t,Y_t)) dt + f(t,Y_t)|_{U_0} d X_t
\\ Y_0 & = & \xi
\end{array}
\right.
\end{align}
with mappings $f_0 : [0,T] \times W \to W$ and $f : [0,T] \times W \to L(U,W)$.

\begin{definition}
Let $\xi \in W$ be arbitrary. An $W$-valued continuous, adapted process $Y = (Y_t)_{t \in [0,T_0]}$ for some $T_0 \in (0,T]$ is called a \emph{local mild solution} to the SPDE \eqref{SPDE-Wiener} with $Y_0 = \xi$ if
\begin{align*}
Y_t = S_t \xi + \int_0^t S_{t-s} f_0(s,Y_s) ds + \int_0^t S_{t-s} f(s,Y_s) dX_s, \quad t \in [0,T_0].
\end{align*}
If we can choose $T_0 = T$, then we also call $Y$ a \emph{(global) solution} to the SPDE \eqref{SPDE-Wiener} with $Y_0 = \xi$.
\end{definition}

Let $\beta \in (\frac{1}{3},\frac{1}{2})$ be an index. By Proposition \ref{prop-Wiener-rough-path} there is a $\bbp$-nullset $N$ such that $\mathbf{X} \in \scrc^{\beta}([0,T],U)$ on $N^c$. In the upcoming result, local and global solutions to the random RPDE \eqref{random-SPDE} are defined according to Definition \ref{def-mild-solution} for each $\omega \in N^c$. Its proof is similar to that of Theorem \ref{thm-Ito-SDEs}. The essential ingredients are Theorem \ref{thm-RPDE-main}, Proposition \ref{prop-Ito-Lyons-PDE} and Proposition \ref{prop-integrals-coincide}.

\begin{theorem}\cite[Thm. 2.4]{Tappe}
Let $f_0 \in \Lip([0,T] \times W,D(A))$ and $$f \in C_b^{2\beta,3}([0,T] \times H, L(U,D(A^2)))$$ be mappings such that
\begin{align*}
&f_0|_{[0,T] \times D(A)} \in \Lip([0,T] \times D(A),D(A^2)),
\\ &f|_{[0,T] \times D(A)} \in C_b^{2\beta,3}([0,T] \times D(A), L(U,D(A^3))).
\end{align*}
Then for every $\xi \in D(A^2)$ the following statements are true:
\begin{enumerate}
\item There exists a unique mild solution $(Y,Y') \in \scrd_X^{2 \beta}([0,T],W)$ to the random RPDE \eqref{random-SPDE} with $Y_0 = \xi$ on $N^c$.

\item If $W$ is a separable Hilbert space $H$, then there is a $\bbp$-nullset $N_1$ with $N \subset N_1$ such that the associated stochastic process $Y$ restricted to $N_1^c$ is also a mild solution to the It\^{o} SPDE \eqref{SPDE-Wiener}.
\end{enumerate}
\end{theorem}

\subsection{Stochastic partial differential equations driven by fractional Brownian motion}

In the last section we briefly outline how our results can also be used for the study of SPDEs driven by fractional Brownian motion. We refer to \cite[Sec. 2.11]{Tappe} for more details about the material of this section.

Let us recall the definition of an infinite dimensional fractional Brownian motion. As in the previous sections, let $U$ be a separable Hilbert space. For an $U$-valued Gaussian process $X$ we introduce the \emph{covariance function}
\begin{align*}
R : \bbr_+^2 \to L_2(U), \quad (s,t) \mapsto \bbe[X_s \otimes X_t].
\end{align*}
Let $Q \in L_1^{++}(U)$ be a nuclear, self-adjoint, positive definite linear operator.

\begin{definition}
A centered Gaussian process $X$ is called a \emph{$Q$-fractional Brownian motion} with Hurst index $H \in (\frac{1}{3},\frac{1}{2}]$ if its covariance function is given by
\begin{align*}
R(s,t) =  \frac{1}{2} \Big( s^{2H} + t^{2H} - |t-s|^{2H} \Big) Q, \quad s,t \in \bbr_+.
\end{align*}
\end{definition}

In what follows, let $X$ be a $Q$-fractional Brownian motion with Hurst index $H \in (\frac{1}{3},\frac{1}{2}]$. There exist an orthonormal basis $\{ e_k \}_{k \in \bbn}$ of $U$ and a sequence $(\lambda_k)_{k \in \bbn} \subset (0,\infty)$ with $\sum_{k=1}^{\infty} \lambda_k < \infty$ such that
\begin{align*}
Q e_k = \lambda_k e_k \quad \text{for all $k \in \bbn$,}
\end{align*}
and the $Q$-fractional Brownian motion $X$ admits the series representation
\begin{align*}
X_t = \sum_{k=1}^{\infty} \sqrt{\lambda_k} \beta_t^k e_k, \quad t \in [0,T]
\end{align*}
with a sequence of independent real-valued fractional Brownian motions $(\beta^k)_{k \in \bbn}$ with Hurst index $H$; see, for example  \cite{Grecksch-Anh}, \cite{Duncan-Maslowski} or \cite{Grecksch}. Note the similarity to the series representation \eqref{series-Wiener} for $Q$-Wiener processes. The following result concerning the second order process of a fractional Brownian motion is essentially a consequence of \cite[Lemma 2.4]{Hesse-Neamtu-local}.

\begin{proposition}\cite[Prop. 2.24]{Tappe}\label{prop-frac-rough-path}
There exists a L\'{e}vy area $\bbx$ such that $\bbp$-almost surely
$$\mathbf{X} := (X,\bbx) \in \scrc_g^{\alpha}([0,T],U)$$ for each $\alpha \in (\frac{1}{3},H)$.
\end{proposition}

\begin{remark}
If $H = \frac{1}{2}$, then $X$ is a $Q$-Wiener process (see Section \ref{sec-Wiener-enhanced}), and the weakly geometric rough path $\mathbf{X} = (X,\bbx)$ according to Proposition \ref{prop-frac-rough-path} is a Stratonovich-enhanced $Q$-Wiener process, thus providing a generalization of Proposition \ref{prop-BB-Strat-rough-path-g} to the infinite dimensional setting.
\end{remark}

Now, let $W$ be a Banach space, and let $A$ be the generator of a $C_0$-semigroup $(S_t)_{t \geq 0}$ on $W$. Consider the random RPDE
\begin{align}\label{RPDE-fractional}
\left\{
\begin{array}{rcl}
dY_t & = & (A Y_t + f_0(t,Y_t)) dt + f(t,Y_t) d \mathbf{X}_t
\\ Y_0 & = & \xi
\end{array}
\right.
\end{align}
with mappings $f_0 : [0,T] \times W \to W$ and $f : [0,T] \times W \to L(U,W)$. Let $\beta \in (\frac{1}{3},H)$ be an index. By Proposition \ref{prop-frac-rough-path} there is a $\bbp$-nullset $N$ such that $\mathbf{X} \in \scrc_g^{\beta}([0,T],U)$ on $N^c$. As an immediate consequence of Theorem \ref{thm-RPDE-main} we obtain the following result.

\begin{theorem}\cite[Thm. 2.5]{Tappe}\label{thm-frac}
Let $f_0 \in \Lip([0,T] \times W,D(A))$ and $$f \in C_b^{2\beta,3}([0,T] \times H, L(U,D(A^2)))$$ be mappings such that
\begin{align*}
&f_0|_{[0,T] \times D(A)} \in \Lip([0,T] \times D(A),D(A^2)),
\\ &f|_{[0,T] \times D(A)} \in C_b^{2\beta,3}([0,T] \times D(A), L(U,D(A^3))).
\end{align*}
Then for every $\xi \in D(A^2)$ there exists a unique mild solution $$(Y,Y') \in \scrd_X^{2 \beta}([0,T],W)$$ to the random RPDE \eqref{RPDE-fractional} with $Y_0 = \xi$ on $N^c$.
\end{theorem}

\begin{remark}
If $W$ is a separable Hilbert space $H$, then there exists a theory of stochastic integration (see \cite{Duncan}), and one may check that the associated stochastic process $Y$ from Theorem \ref{thm-frac} is also a mild solution to the fractional SPDE
\begin{align*}
\left\{
\begin{array}{rcl}
dY_t & = & (A Y_t + f_0(t,Y_t)) dt + f(t,Y_t) d X_t
\\ Y_0 & = & \xi.
\end{array}
\right.
\end{align*}
\end{remark}

\end{document}